\documentclass{amsart}
\pdfoutput=1
\usepackage[english]{babel}
\usepackage[utf8]{inputenc}
\usepackage[T1]{fontenc}
\usepackage{amsmath, latexsym, amssymb, amsfonts, amsthm}
\usepackage{stmaryrd}

\usepackage{algorithm}
\usepackage[noend]{algpseudocode}

\usepackage{enumitem}
\usepackage{caption}
\usepackage{subcaption}
\usepackage{multirow}
\usepackage[hidelinks]{hyperref}
\usepackage{tikz}
\usetikzlibrary{arrows.meta}
\usetikzlibrary{shapes.geometric}
\usetikzlibrary{calc}
\usepackage{environ}
\usepackage{xspace}

\usepackage[numbers, sort, longnamesfirst]{natbib}
\usepackage{doi, url}

\setlist[enumerate,1]{label=$\mathrm{(\arabic{enumi})}$}

\usepackage{thm-restate}

\newtheorem{thm}{Theorem}[section]
\newtheorem{lem}[thm]{Lemma}
\newtheorem{prop}[thm]{Proposition}
\newtheorem{cor}[thm]{Corollary}

\declaretheorem[name=Theorem, numberlike=thm]{re-thm}
\declaretheorem[name=Proposition, numberlike=thm]{re-prop}

\theoremstyle{definition}
\newtheorem{defn}[thm]{Definition}
\newtheorem{ex}[thm]{Example}

\theoremstyle{remark}
\newtheorem{rem}[thm]{Remark}

\newcommand{\N}{{\mathbb N}}
\newcommand{\R}{{\mathbb R}}
\newcommand{\Z}{{\mathbb Z}}

\newcommand{\cC}{{\mathcal C}}
\newcommand{\cV}{{\mathcal V}}

\DeclareMathOperator{\Cl}{Cl}
\DeclareMathOperator{\Ex}{Ex}   % for exit set
\DeclareMathOperator{\Fix}{Fix}   % for fixed points

\DeclareMathOperator{\image}{im}
\DeclareMathOperator{\rank}{rank}
\DeclareMathOperator{\card}{card}
\DeclareMathOperator{\dom}{dom}
\DeclareMathOperator{\Int}{Int}
\DeclareMathOperator{\Bd}{Bd}

\DeclareMathOperator{\St}{St}

\DeclareMathOperator{\Low}{Low}

\newcommand{\mdm}{\texttt{mdm}\xspace}

\newcommand{\maxdim}{k} % dimension of range of mdm functions
\newcommand{\rightconnects}[1]{\rightarrowtriangle_{#1}}
\newcommand{\leftconnects}[1]{\leftarrowtriangle_{#1}}

\newcommand{\GenerateMDM}{\texttt{\upshape GenerateMDM}\xspace}
\newcommand{\ExpandMDM}{\texttt{\upshape ExpandMDM}\xspace}
\newcommand{\ComputeG}{\texttt{\upshape ComputeG}\xspace}
\newcommand{\Matching}{\texttt{\upshape Matching}\xspace}
\newcommand{\ComputeDiscreteGradient}{\texttt{\upshape ComputeDiscreteGradient}\xspace}

\newcommand{\AddCofacets}{\texttt{\upshape add\_cofacets}\xspace}
\newcommand{\LevelSets}{\texttt{\upshape LevelSets}\xspace}
\newcommand{\NumUnprocFacets}{\texttt{\upshape num\_unproc\_facets}\xspace}
\newcommand{\UnprocFacet}{\texttt{\upshape unprocessed\_facet}\xspace}
\newcommand{\processed}{\texttt{\upshape processed}\xspace}
\newcommand{\PQone}{\texttt{\upshape PQone}\xspace}
\newcommand{\PQzero}{\texttt{\upshape PQzero}\xspace}
\newcommand{\True}{\textbf{\upshape True}\xspace}
\newcommand{\False}{\textbf{\upshape False}\xspace}

\newcommand{\Sim}[1]{\ensuremath{\sim_{#1}}}
\newcommand{\Hom}{\ensuremath{H_*}}
\newcommand{\Head}[2]{\ensuremath{\mathcal{H}_{#1}(#2)}}
\newcommand{\Tail}[2]{\ensuremath{\mathcal{T}_{#1}(#2)}}

\newcommand{\Pareto}[1]{\ensuremath{\mathcal{P}_{#1}}}

\makeatletter
\newsavebox{\measure@tikzpicture}
\NewEnviron{scaletikzpicturetowidth}[1]{%
	\def\tikz@width{#1}%
	\def\tikzscale{1}\begin{lrbox}{\measure@tikzpicture}%
		\BODY
	\end{lrbox}%
	\pgfmathparse{#1/\wd\measure@tikzpicture}%
	\edef\tikzscale{\pgfmathresult}%
	\BODY
}
\makeatother

\title[Analyzing Multifiltering Functions Using MDM Theory]{Analyzing Multifiltering Functions Using Multiparameter Discrete {Morse} Theory}
\author{Guillaume Brouillette}
\address{D{\'e}partement de math{\'e}matiques\\ Universit{\'e} de Sherbrooke\\ Sherbrooke, QC, Canada}
\email{Guillaume.Brouillette2@USherbrooke.ca}

\thanks{This research was supported by the Natural Sciences and Engineering Council of Canada (NSERC) under grant number 569623 and the Fonds de recherche du Qu{\'e}bec -- Nature et technologies (FRQNT) under grant number 289126. The author would like to thank his doctoral advisors Tomasz Kaczynski and Madjid Allili for their support and guidance throughout the research and writing processes.}

\begin{document}

	\begin{abstract}
		A multiparameter filtration, or a multifiltration, may in many cases be seen as the collection of sublevel sets of a vector function, which we call a multifiltering function. The main objective of this paper is to obtain a better understanding of such functions through multiparameter discrete Morse (\mdm) theory, which is an extension of Morse-Forman theory to vector-valued functions. Notably, we prove algorithmically that any multifiltering function defined on a simplicial complex can always be approximated by a compatible \mdm function. Moreover, we define the Pareto set of a discrete multifiltering function and show that the concept links directly to that of critical simplices of a \mdm function. Finally, we experiment with these notions using triangular meshes.
	\end{abstract}

	\maketitle

	\section{Introduction}

	Topological data analysis (TDA) is a fast-growing branch of mathematics which proposes a panoply of tools to better understand and visualize data \citep{Carlsson2021, Peikert2012, Tierny2017}. One of the most popular method in TDA is persistent homology, which aims to detect the most significant topological features of a space filtered in accordance with some parameter or (real) function \citep{Edelsbrunner2002, Edelsbrunner2008}. In many contexts, notably when the given space is generated from a noisy sample of points, it is more convenient to consider multiple parameters, such as the scale and density of the data \citep{Botnan2023}. This motivated \citet{Carlsson2009} to extend the concept to multiparameter persistent homology. Multipersistence has already proven itself to be useful in applications, for instance in the analysis of biomedical data \citep{Vipond2021, Xia2015}.

	However, it is a challenge to compute multipersistent homology efficiently \citep{Carlsson2010}. Different ways to lessen the computational burden while preserving the topological information extracted have been proposed in the literature \citep{Cagliari2010, Cerri2011, Fugacci2019, Fugacci2023, Kerber2021, Lesnick2022}.

	One strategy we take a particular interest in is the use of discrete Morse theory, introduced by \citet{Forman1998}. Akin to its smooth counterpart \citep{Knudson2015, Matsumoto2002}, it provides a way to compute the homology of a whole space using only the critical simplices of a (discrete) Morse function defined on it and its (discrete) gradient field. While there exists other notable discrete or piecewise linear (PL) Morse theories \citep{Bestvina2008, Bloch2013, Edelsbrunner2010, Grunert2023}, Forman's version has the advantage of being highly versatile and adapts well to the present context. Indeed, over the years, it led to many applications in TDA \citep{DelgadoFriedrichs2015, Gyulassy2014, Hu2021, King2017} as well as the development of a novel combinatorial approach to dynamics \citep{Batko2020, Kaczynski2016, Lipinski2023, Mrozek2017}.

	More specifically, discrete Morse theory can be used in order to ease the computation of the persistent homology associated to a filtration \citep{Mischaikow2013}. Essentially, this involves reducing the size of the input data by keeping only the elements which are critical, that is, necessary for computations. In particular, when considering an injective filtering function defined on either the vertices of a simplicial complex \citep{King2005} or a cubical grid representing a digital image \citep{Robins2011}, it can be used to generate a discrete gradient field whose critical points represent the homological changes in the corresponding filtration.

	These ideas are pushed a step further to the multiparameter framework by \citet{Allili2017}. In their paper, the authors prove that a discrete gradient field, also called a Morse matching, can be used to reduce a Lefschetz complex, e.g. a simplicial or cubical complex, while preserving multipersistent homology. Moreover, they present the algorithm \Matching that takes as input a simplicial complex and a vector function defined on its vertices and outputs a Morse matching compatible with the multifiltration induced by the given multifiltering function. By restraining the considered input functions to component-wise injective vertex maps, the algorithm \Matching is improved in \citep{Allili2019} and a more efficient version, \ComputeDiscreteGradient, is proposed thereafter by Scaramuccia, Iuricich, De Floriani and Landi \cite{Scaramuccia2020}.

	Experiments in \citep{Allili2019} show that critical simplices output by these algorithms tend to form clusters which resemble smooth Pareto sets, as defined by \citet{Smale1975, Wan1975}, who extended the classical Morse theory to vector-valued functions. In the last few years, many links have been made between the Pareto sets of smooth $\R^2$-valued filtering functions and the homological changes in their corresponding filtration \citep{AssifPK2021, Budney2023, Cerri2019}. In discrete settings however, to our knowledge, only a few contributions on related subjects may be found in the literature. In \citep{Edelsbrunner2008a}, Jacobi sets of vector-valued PL mappings are used to characterize Reeb spaces, whereas in \citep{Huettenberger2013}, Pareto sets are defined for such maps and used to analyze and visualize data.

	This motivates the need to extend Morse-Forman theory to vector-valued functions. In \citep{Allili2019}, fundamental definitions needed to do so are proposed, and the elaboration of multiparameter discrete Morse (\mdm) theory is completed by \citet{Brouillette2022}. In this paper, many central results of Forman are generalized to the vectorial setting and the notion of critical components is introduced. Although it gives a solid theoretical framework, the work in \citep{Brouillette2022} does not address specifically how \mdm functions could be used in practice.

	Therefore, the objective of the present paper is to pick up where \citet{Brouillette2022} left off by bridging theory and practice. We do so in two ways. First, we propose an algorithm, \GenerateMDM, that takes as input any multifiltering function $f:K\rightarrow\R^\maxdim$ defined on a simplicial complex $K$ and generates a \mdm function which approximates $f$ and whose gradient field is compatible with the multifiltration induced by $f$, as stated in Theorem \ref{theo:gMDM}. Second, we define the concept of Pareto set for a discrete multifiltering function $f$ and see experimentally how its connected components compare to the critical components of the \mdm function output by \GenerateMDM when given $f$ as input.

	The paper is structured as follows. In Section \ref{sec:Preliminaries}, we recall some known concepts, mainly about multipersistent homology. In Section \ref{sec:MDMTheory}, a brief overview of \mdm theory is presented. Then, our main algorithm \GenerateMDM and subroutines are described in Section \ref{sec:Algorithms}. We analyze its complexity and compare it with previous similar algorithms \Matching and \ComputeDiscreteGradient. %Next, we prove in Section \ref{sec:Correctness} that \GenerateMDM outputs the desired result.
	Thereafter, the notion of discrete Pareto set is introduced in Section \ref{sec:Pareto}. Finally, in Section \ref{sec:ExperimentalResults}, we see how \GenerateMDM performs in practice and we experiment with the concept of discrete Pareto set and that of critical components of a \mdm function on triangular meshes.

	\section{Background}\label{sec:Preliminaries}

	\subsection{Simplicial complexes}

	Let $K_0$ be a finite set. We consider the elements of $K_0$ as vertices and may be seen either as abstract objects or as points in $\R^n$. A simplicial complex $K$ with vertices in $K_0$ is a collection of subsets $\sigma\subseteq K_0$ such that $\sigma\in K$ implies that every subset of $\sigma$ is also in $K$. A set $\sigma\in K$ of $p+1$ vertices is called a \emph{$p$-simplex}, or a \emph{simplex} of dimension $p=:\dim\sigma$, of $K$. We sometimes specify the dimension of a simplex by using the superscript $\sigma^{(p)}$ and we note $K_p$ the set of $p$-simplices of $K$. Also, we note $\dim K := \max_{\sigma\in K}\dim\sigma$ the \emph{dimension} of $K$. Furthermore, if $\tau\subseteq\sigma\in K$, we say that $\tau$ is a \emph{face} of $\sigma$, noted as $\tau\leq\sigma$, and $\sigma$ a \emph{coface} of $\tau$. We write $\tau < \sigma$ when $\tau\leq\sigma$ and $\tau\neq\sigma$. When $\tau<\sigma$ and $\dim\tau = \dim\sigma-1$, then $\tau$ is a \emph{facet} of $\sigma$ and $\sigma$ a \emph{cofacet} of $\tau$.

	Moreover, the set $K$ may be seen as a topological space by endowing it with the Alexandrov topology \citep{Barmak2011, McCord1966, Stong1966}. A set of simplices $S\subset K$ is open in this topology if $\tau\geq\sigma\in S$ implies that $\tau\in S$. Hence, the smallest neighbourhood of a simplex $\sigma\in K$ is its \emph{star} $\St\sigma= \{\tau\in K\ |\ \tau\geq\sigma\}$. Also, it could be shown that for any $S\subseteq K$, its \emph{closure} is $\Cl S = \{\tau\in K \ |\ \tau\leq\sigma\text{ for some }\sigma\in S\}$ and its \emph{interior} is $\Int S = \{\sigma\in S\ |\ \tau\in S\text{ for all }\tau\geq \sigma\}$. Using these sets, we can define the \emph{boundary} and the \emph{exit set} of $S$ as $\Bd S := \Cl S\backslash\Int S$ and $\Ex S := \Cl S\backslash S$ respectively.

	Finally, notice that in the Alexandrov topology, two simplices $\sigma,\tau\in K$ are neighbours, i.e. one of them belongs in the star of the other, if and only if $\sigma\leq\tau$ or $\sigma\geq\tau$. Following this idea, we could prove that the concepts of connectedness and path-connectedness are equivalent in this topology and $S\subseteq K$ is connected if and only if for all $\sigma,\tau\in S$, there is a sequence (a path) $\sigma=\sigma_0,\sigma_1,...,\sigma_r=\tau$ in $S$ such that, for each $i=1,...,r$, we have either $\sigma_{i-1}\leq\sigma_i$ or $\sigma_{i-1}\geq\sigma_i$.

	\subsection{Homology}\label{sec:Homology}

	We now outline some notions related to the homology and relative homology of simplicial complexes. For more details, we recommend \citep{Edelsbrunner2010} for a first introduction to the subject in the context of simplicial complexes and \citep{Hatcher2002} for a comprehensive and global presentation of algebraic topology.

	The homology of a simplicial complex $K$ is a collection $\Hom(K) = \{H_p(K)\}_{p\in\Z}$ of modules $H_p(K)$ over some fixed ring (more precisely a principal ideal domain, usually $\Z$ or $\Z_2$) which describe the fundamental topological features of $K$. The rank of $H_0(K)$ counts the number of connected components of $K$, that of $H_p(K)$ counts the number of $p$-dimensional holes for $1\leq p\leq\dim K$ and $H_p(K)\cong 0$ otherwise. We call $\beta_p(K):=\rank H_p(K)$ the \emph{$p^\text{th}$ Betti number} of $K$. We say two simplicial complexes $K$ and $K'$ have the same homology when $H_p(K)\cong H_p(K')$ for each $p\in\Z$, which we note $\Hom(K)\cong\Hom(K')$. Moreover, for a subcomplex $L\subseteq K$, the relative homology of $K$ with respect to $L$, noted $\Hom(K,L)$, may be seen as the homology of $K$ with all simplices in $L$ identified to a single point. The relative homology $\Hom(K,L)$ is useful when interested in the homology of $K$ only outside of the subcomplex $L$.

	Furthermore, note that the ring chosen to define $\Hom(K)$ plays an important role on the computation of the homology modules. In general, we choose $\Z$ since it yields more information about $K$, such as the torsion of the space. Nonetheless, we may also consider to use a field, e.g. $\Z_2$, to ease the computation of $\Hom(K)$. We illustrate the impact of choosing $\Z_2$ over $\Z$ with the following example. If $K$ triangulates the Klein bottle, it is known that
	\begin{gather*}
		H_p(K)\cong\begin{cases}
			0 & \text{ if } p=2,\\
			\Z\oplus\Z_2 & \text{ if } p=1,\\
			\Z & \text{ if } p=0,
		\end{cases}
		\qquad
		H_p(K;\Z_2)\cong\begin{cases}
			\Z_2 & \text{ if } p=2,\\
			\Z_2^2 & \text{ if } p=1,\\
			\Z_2 & \text{ if } p=0.
		\end{cases}
	\end{gather*}
	Therefore, we have $\beta_0(K) = 1$, $\beta_1(K) = 1$ and $\beta_2(K)=0$ when considering homology with coefficients in $\Z$, whereas $\beta_0(K;\Z_2) = 1$, $\beta_1(K;\Z_2) = 2$ and $\beta_2(K;\Z_2)=1$ when using coefficients in $\Z_2$. The $\Z_2$ summand in $H_1(K)$ corresponds to the torsion of the simplicial complex, which represents in this case the non-orientability of the Klein bottle. This information is lost when considering coefficients in $\Z_2$. Thus, throughout this paper, we use homology with coefficients in $\Z$, unless clearly stated otherwise (as in Table \ref{table:OptimalityCheck}).

	\subsection{Multiparameter persistent homology}\label{sec:Multipersistence}

	We may generalize the concept of homology to multiparameter persistent homology (also called multipersistent homology, multiparameter persistence or simply multipersistence) by considering multifiltered simplicial complexes \citep{Botnan2023, Carlsson2009}.

	Let $\preceq$ be the coordinate-wise order on $\R^\maxdim$, i.e. for $u,u'\in\R^\maxdim$, we have $u\preceq u'$ iff $u_i\leq u_i'$ for each $i=1,...,\maxdim$. We also note $u\precneqq u'$ if $u\preceq u'$ and $u\neq u'$. A simplicial complex $K$ is \emph{multifiltered} when it is considered together with a \emph{multifiltration}, which is a finite family of subcomplexes $\{K(u)\}_{u\in\R^\maxdim}$ of $K$ such that $u\preceq u'$ implies $K(u)\subseteq K(u')$. An example of a bifiltration, i.e. a multifiltration with $\maxdim=2$, is illustrated in Figure \ref{figSub:Multifiltration}.

	\pgfmathsetmacro{\gap}{0.382*0.5}
	\begin{figure}[ht]
		\centering
		\subcaptionbox{\label{figSub:Multifiltration}}[0.5\textwidth]{
			\centering
			\begin{tikzpicture}[scale=1.7]

				\coordinate (A) at (\gap, \gap);
				\coordinate (B) at (1-\gap, \gap);
				\coordinate (C) at (\gap, 1-\gap);
				\coordinate (D) at (1-\gap, 1-\gap);

				\draw[step=1, gray!50, thin] (0,0) grid (3,3);

				\node at (0.5, -\gap){\footnotesize $0$};
				\node at (1.5, -\gap){\footnotesize $1$};
				\node at (2.5, -\gap){\footnotesize $2$};
				\node at (-\gap, 0.5){\footnotesize $0$};
				\node at (-\gap, 1.5){\footnotesize $1$};
				\node at (-\gap, 2.5){\footnotesize $2$};

				% (0,0)

				% (0,1)
				\node[blue] at ($(C) + (0,1)$){$\bullet$};

				% (0,2)
				\node at ($(C) + (0,2)$){$\bullet$};

				% (1,0)
				\node[blue] at ($(B) + (1,0)$){$\bullet$};

				% (1,1)
				\node[blue] at ($(A) + (1,1)$){$\bullet$};
				\node at ($(B) + (1,1)$){$\bullet$};
				\node at ($(C) + (1,1)$){$\bullet$};

				% (1,2)
				\draw[very thick, blue] ($(A) + (1,2)$) -- ($(C) + (1,2)$);

				\node at ($(A) + (1,2)$){$\bullet$};
				\node at ($(B) + (1,2)$){$\bullet$};
				\node at ($(C) + (1,2)$){$\bullet$};

				% (2,0)
				\draw[very thick, blue] ($(B) + (2,0)$) -- ($(D) + (2,0)$);

				\node at ($(B) + (2,0)$){$\bullet$};
				\node[blue] at ($(D) + (2,0)$){$\bullet$};

				% (2,1)
				\fill[blue!25] ($(A) + (2,1)$) -- ($(B) + (2,1)$) -- ($(D) + (2,1)$) -- cycle;

				\draw[very thick, blue] ($(A) + (2,1)$) -- ($(B) + (2,1)$);
				\draw[very thick, blue] ($(A) + (2,1)$) -- ($(D) + (2,1)$);
				\draw[very thick] ($(B) + (2,1)$) -- ($(D) + (2,1)$);

				\node at ($(A) + (2,1)$){$\bullet$};
				\node at ($(B) + (2,1)$){$\bullet$};
				\node at ($(C) + (2,1)$){$\bullet$};
				\node at ($(D) + (2,1)$){$\bullet$};

				% (2,2)
				\fill[black!10] ($(A) + (2,2)$) -- ($(B) + (2,2)$) -- ($(D) + (2,2)$) -- cycle;
				\fill[blue!25] ($(A) + (2,2)$) -- ($(C) + (2,2)$) -- ($(D) + (2,2)$) -- cycle;

				\draw[very thick] ($(A) + (2,2)$) -- ($(B) + (2,2)$);
				\draw[very thick] ($(A) + (2,2)$) -- ($(C) + (2,2)$);
				\draw[very thick] ($(A) + (2,2)$) -- ($(D) + (2,2)$);
				\draw[very thick] ($(B) + (2,2)$) -- ($(D) + (2,2)$);
				\draw[very thick, blue] ($(C) + (2,2)$) -- ($(D) + (2,2)$);

				\node at ($(A) + (2,2)$){$\bullet$};
				\node at ($(B) + (2,2)$){$\bullet$};
				\node at ($(C) + (2,2)$){$\bullet$};
				\node at ($(D) + (2,2)$){$\bullet$};
			\end{tikzpicture}
		}
		\subcaptionbox{\label{figSub:MultipersistentModules}}[0.4\textwidth]{
			\centering
			\begin{tikzpicture}[scale=1.7]
				\draw[white, opacity=0] (0,-\gap-0.5) rectangle (2,2.5);

				\node (Z1) at (0,0) {$0$};
				\node (Z2) at (0,1) {$\Z$};
				\node (Z3) at (0,2) {$\Z$};
				\node (Z4) at (1,0) {$\Z$};
				\node (Z5) at (1,1) {$\Z^3$};
				\node (Z6) at (1,2) {$\Z^2$};
				\node (Z7) at (2,0) {$\Z$};
				\node (Z8) at (2,1) {$\Z^2$};
				\node (Z9) at (2,2) {$\Z$};

				\path [->] (Z1) edge (Z2);
				\path [->] (Z1) edge (Z4);
				\path [->] (Z2) edge (Z3);
				\path [->] (Z2) edge (Z5);
				\path [->] (Z3) edge (Z6);
				\path [->] (Z4) edge (Z5);
				\path [->] (Z4) edge (Z7);
				\path [->] (Z5) edge (Z6);
				\path [->] (Z5) edge (Z8);
				\path [->] (Z6) edge (Z9);
				\path [->] (Z7) edge (Z8);
				\path [->] (Z8) edge (Z9);
			\end{tikzpicture}
		}
		\caption{In \subref{figSub:Multifiltration}, a multifiltered simplicial complex. The blue simplices represent the ones newly added at each step. In \subref{figSub:MultipersistentModules}, a commutative diagram isomorphic to the associated $0^\text{th}$ multiparameter persistence module.}\label{fig:Multifiltration}
	\end{figure}

	Moreover, for every dimension $p$, each inclusion $K(u)\hookrightarrow K(u')$ induces a linear map $\iota_p^{u,u'}:H_p(K(u))\rightarrow H_p(K(u'))$ on the associated homology modules. Essentially, the image of a map $\iota_p^{u,u'}$ represents the homology classes of $K(u)$ of dimension $p$ still alive in $K(u')$. We define the \emph{$p^\text{th}$ multiparameter persistence module} of a given multifiltration as the family of all homology modules $H_p(K(u))$ together with the linear maps $\iota_p^{u,u'}:H_p(K(u))\rightarrow H_p(K(u'))$. Such a module is depicted in Figure \ref{figSub:MultipersistentModules}. Multiparameter persistence could be summarized as the study of the multiparameter persistence modules of a multifiltration $\{K(u)\}_{u\in\R^\maxdim}$ or, put simply, of the homological changes in $K(u)$ as the multiple parameters $u_1,...,u_\maxdim$ increase.

	In this paper, we are mostly interested in multifiltrations induced by a vector map. More precisely, given $f:K\rightarrow\R^\maxdim$, we consider the collection of sublevel sets $K(u):=\{\sigma\in K\ |\ f(\sigma)\preceq u\}$. It is easy to see that $K(u)\subseteq K(u')$ when $u\preceq u'$. However, for $K(u)$ to be a subcomplex of $K$ for all $u$, it is necessary and sufficient to assume that $\tau<\sigma$ implies $f(\tau)\preceq f(\sigma)$, which leads to the following definition.

	\begin{defn}
		We call $f:K\rightarrow\R^\maxdim$ an \emph{admissible map}, or a \emph{(multi)filtering function}, if $f(\tau)\preceq f(\sigma)$ for all pairs $\tau<\sigma$ in $K$.
	\end{defn}

	An example of multifiltering function is shown in Figure \ref{figSub:MultifilteringFunction}.

	\begin{figure}[ht]
		\centering
		\subcaptionbox{\label{figSub:MultifilteringFunction}}[0.45\textwidth]{
			\centering
			\begin{tikzpicture}[scale=3]
				\coordinate (A) at (0, 0);
				\coordinate (B) at (1, 0);
				\coordinate (C) at (0, 1);
				\coordinate (D) at (1, 1);

				\fill[black!10] (A) rectangle (D);
				\draw[thick] (A) -- (B);
				\draw[thick] (A) -- (C);
				\draw[thick] (A) -- (D);
				\draw[thick] (B) -- (D);
				\draw[thick] (C) -- (D);

				\node at (A){$\bullet$};
				\node at (B){$\bullet$};
				\node at (C){$\bullet$};
				\node at (D){$\bullet$};

				\node[below left] at (A) {$(1,1)$};
				\node[below right] at (B) {$(1,0)$};
				\node[above left] at (C) {$(0,1)$};
				\node[above right] at (D) {$(2,0)$};

				\node[below] at ($0.5*(A)+0.5*(B)$) {$(2,1)$};
				\node[left] at ($0.5*(A)+0.5*(C)$) {$(1,2)$};
				\node[xshift=0.125in, yshift=-0.125in] at ($0.5*(A)+0.5*(D)$) {$(2,1)$};
				\node[right] at ($0.5*(B)+0.5*(D)$) {$(2,0)$};
				\node[above] at ($0.5*(C)+0.5*(D)$) {$(2,2)$};

				\node[below] at ($0.3*(A)+0.4*(B)+0.3*(D)$) {$(2,1)$};
				\node at ($0.3*(A)+0.4*(C)+0.3*(D)$) {$(2,2)$};
			\end{tikzpicture}
		}
		\subcaptionbox{\label{figSub:MultifilteringFunctionGradient}}[0.45\textwidth]{
			\centering
			\begin{tikzpicture}[scale=3]
				\coordinate (A) at (0, 0);
				\coordinate (B) at (1, 0);
				\coordinate (C) at (0, 1);
				\coordinate (D) at (1, 1);

				\node[white, opacity=0, below] at ($0.5*(A)+0.5*(B)$) {$(2,1)$};

				\fill[black!10] (A) rectangle (D);
				\draw[thick] (A) -- (B);
				\draw[ultra thick, red] (A) -- (C);
				\draw[ultra thick, red] (A) -- (D);
				\draw[thick] (B) -- (D);
				\draw[thick] (C) -- (D);

				\node[red] at (A){$\bullet$};
				\node[red] at (B){$\bullet$};
				\node[red] at (C){$\bullet$};
				\node at (D){$\bullet$};

				\draw[ultra thick, -stealth] (D) -- ($0.5*(B)+0.5*(D)$);
				\draw[ultra thick, -stealth] ($0.5*(A)+0.5*(B)$) -- ($0.33*(A)+0.33*(B)+0.33*(D)$);
				\draw[ultra thick, -stealth] ($0.5*(C)+0.5*(D)$) -- ($0.33*(A)+0.33*(C)+0.33*(D)$);
			\end{tikzpicture}
		}
		\caption{In \subref{figSub:MultifilteringFunction}, the multifiltering function $f$ which induces the multifiltration in Figure \ref{figSub:Multifiltration}. In \subref{figSub:MultifilteringFunctionGradient}, a combinatorial vector field $\cV$ compatible with $f$. Pairs $\tau<\sigma$ such that $\cV(\tau) = \sigma$ are represented by the arrows, while the fixed points of $\cV$ are the simplices in red.}\label{fig:MultifilteringFunction}
	\end{figure}

	In previous related papers \citep{Allili2019, King2005, Scaramuccia2020}, the definition of admissible maps was more specific. Namely, a vertex map $f:K_0\rightarrow\R^\maxdim$ was considered and assumed to be component-wise injective, i.e. each $f_i$ had to be injective. Then, $f$ was extended to a multifiltering function on all $K$ by setting $f_i(\sigma) := \max_{v\in\sigma}f_i(v)$ for each $\sigma\in K$ and $i=1,...,\maxdim$. Here, we call $f$ a \emph{max-extension} of a vertex map if $f_i(\sigma) = \max_{v\in\sigma}f_i(v)$ for all $\sigma\in K$ and $i=1,...,\maxdim$. Although this type of map is very convenient, especially when working with point data, we will see in the examples below that some common filtering functions are not max-extensions. Hence, our broader definition of admissible maps makes it possible to apply the algorithms of this paper to a wider range of input.

	Finally, notice that for all $\sigma,\tau\in K$, we have $f(\tau)\preceq f(\sigma)$ iff $f_i(\tau)\leq f_i(\sigma)$ for each $i=1,...,\maxdim$. Therefore, $f$ is admissible iff each $f_i:K\rightarrow\R$ also is, so we can construct multifiltering function by combining multiple (real-valued) filtering functions. Thus, we end this subsection with a few examples of real-valued maps which may be combined to obtain a multifiltering function:
	\begin{itemize}
		\item Consider a simplicial complex $K$ embedded in $\R^n$. Then, we can use the max-extension of the projection $(x_1,x_2,...,x_n)\mapsto x_j$ of the vertices on the $j^\text{th}$ coordinate axis as a filtering function on $K$, or combine multiple projections to obtain a multifiltering function. In fact, we will see in Section \ref{sec:ExperimentalResults} that in many of our experiments, we use simplicial complexes embedded in $\R^3$ along with the max-extension of $(x,y,z)\mapsto (x,y)$.

		\item Let $\gamma:K_0\rightarrow(0,\infty)$ be a density function on the vertices $K_0\subset\R^n$ of a simplicial complex $K$, i.e. $\gamma$ associates the points in dense regions of $K_0$ to high values and those in sparse regions to low values. Then, we can use the max-extension of the inverse function $1/\gamma$ to obtain a filtration of $K$ in which simplices in denser regions appear first and those in sparser regions appear last.

		\item Let $K$ be the Delaunay triangulation of a set of points $K_0\subset\R^n$ and consider its filtration in alpha complexes $\{\mathrm{Alpha}_r(K_0)\}_{r\in\R}$. Then, for every $\sigma\in K$, note $r_\sigma$ the minimal radius parameter $r$ such that $\sigma\in\mathrm{Alpha}_r(K_0)$. The sublevel set $K(r)$ of the radius map $\sigma\mapsto r_\sigma$ is exactly $\mathrm{Alpha}_r(K_0)$ for all $r\in\R$, meaning that we can see the filtration $\{\mathrm{Alpha}_r(K_0)\}_{r\in\R}$ as the sublevel set of the radius map $\sigma\mapsto r_\sigma$. A similar reasoning is also valid for point sets filtered in \v{C}ech and Vietoris-Rips complexes.

		Notice that a radius map is not a max-extension of a vertex map since $r_{\{v\}} = 0$ for all $v\in K_0$. Therefore, this is an example of filtering function which can be considered as input for the algorithms presented in this paper, but not for those in \citep{Allili2019, King2005, Scaramuccia2020}.
	\end{itemize}

	\subsection{Combinatorial vector fields}

	To ease the computation of multipersistent homology, we can use combinatorial vector fields. Intuitively, we can see a combinatorial vector field on a simplicial complex $K$ as a partition of $K$ into singletons and pairs $\tau<\sigma$ such that $\tau$ is a facet of $\sigma$. More formally, a \emph{discrete vector field}, or a \emph{combinatorial vector field}, on a simplicial complex $K$ is an injective partial self-map $\cV:K\nrightarrow K$ such that
	\begin{enumerate}
		\item for each $\sigma\in\dom\cV$, either $\cV(\sigma) = \sigma$ or $\cV(\sigma)$ is a cofacet of $\sigma$;
		\item $\dom\cV\cup\image\cV = K$;
		\item $\dom\cV\cap\image\cV = \Fix\cV$.
	\end{enumerate}
	A combinatorial vector field $\cV$ is said to be \emph{acyclic} if all its nontrivial $\cV$-paths do not loop, i.e. for any sequence of simplices
	\begin{align*}
		\tau_0^{(p)},\sigma_0^{(p+1)},\tau_1^{(p)},\sigma_1^{(p+1)},\tau_2^{(p)},...,\sigma_{n-1}^{(p+1)},\tau_n^{(p)}
	\end{align*}
	in $K$ such that $\tau_i\in\dom\cV$, $\cV(\tau_i) = \sigma_i$ and $\sigma_i>\tau_{i+1}\neq\tau_i$ for each $i=0,...,n-1$ (where $n\geq 1$), we have $\tau_n\neq\tau_0$.

	If there exists an acyclic vector field $\cV$ on $K$ and a subcomplex $L\subset K$ such that $\Fix\cV\subseteq L$, then $K$ \emph{collapses} onto $L$ \citep{Brouillette2022}, which we note $K\searrow L$, and it is known that this implies $\Hom(K)\cong\Hom(L)$ \citep{Cohen1973}. This idea translates well in the context of multipersistence. Indeed, if $K$ is a simplicial complex with multifiltration $\{K(u)\}_{u\in\R^\maxdim}$, an acyclic discrete vector field $\cV:K\nrightarrow K$ can be used to reduce the size of $K$ while preserving its multipersistent homology modules \citep{Allili2017}. To do so, it needs to be \emph{compatible} with $\{K(u)\}_{u\in\R^\maxdim}$, meaning that $\sigma\in K(u)\Leftrightarrow \cV(\sigma)\in K(u)$ for all multiparameter $u$ and all $\sigma\in\dom\cV$. This leads to the following definition.

	\begin{defn}
		Consider a multifiltering function $f:K\rightarrow\R^\maxdim$ on a simplicial complex $K$. We say an acyclic combinatorial vector field $\cV$ is \emph{compatible with $f$} or \emph{$f$-compatible} if $f(\sigma) = f(\cV(\sigma))$ for all $\sigma\in\dom\cV$.
	\end{defn}

	In Figure \ref{figSub:MultifilteringFunctionGradient} is illustrated a discrete vector field compatible with the multifiltering function $f$ from Figure \ref{figSub:MultifilteringFunction}.

	\section{Multiparameter discrete Morse theory}\label{sec:MDMTheory}

	The concept of acyclic combinatorial vector fields was first introduced by \citet{Forman1998} as a central element of discrete Morse theory, which is adapted to the multiparameter setting in \citep{Allili2019, Brouillette2022}. We review here the key concepts and results of this extended theory.

	\subsection{Main definitions} We first define the concepts of \mdm function and gradient field, which are at the heart of both the theory and the algorithms that we will present in this paper.

	\begin{defn}\label{def:MDM}
		Consider $g:K\rightarrow\R^\maxdim$ defined on a simplicial complex $K$. For $\sigma^{(p)}\in K$, let
		\begin{align*}
			\Head{g}{\sigma} &= \left\lbrace \gamma^{(p+1)}>\sigma\ |\ g(\gamma)\preceq g(\sigma)\right\rbrace;\\
			\Tail{g}{\sigma} &= \left\lbrace \alpha^{(p-1)}<\sigma\ |\ g(\alpha)\succeq g(\sigma)\right\rbrace.
		\end{align*}
		We say $g$ is \emph{multidimensional discrete Morse}, or simply \mdm, if the following conditions hold for all $\sigma^{(p)}\in K$:
		\begin{enumerate}
			\item\label{enum:defMDMenum1} $\card \Head{g}{\sigma} \leq 1$.
			\item\label{enum:defMDMenum2} $\card \Tail{g}{\sigma} \leq 1$.
			\item\label{enum:defMDMenum3} For all cofacet $\gamma^{(p+1)} > \sigma$, either $g(\gamma)\preceq g(\sigma)$ or $g(\gamma)\succneqq g(\sigma)$.
			\item\label{enum:defMDMenum4} For all facet $\alpha^{(p-1)} < \sigma$, either $g(\alpha)\succeq g(\sigma)$ or $g(\alpha)\precneqq g(\sigma)$.
		\end{enumerate}
		If $\card \Head{g}{\sigma} = \card \Tail{g}{\sigma} = 0$, we say $\sigma$ is a \emph{critical simplex of index $p$} of $g$.
	\end{defn}

	When $g:K\rightarrow\R^\maxdim$ is \mdm, for all $\sigma\in K$, we can show that at most one of the sets $\Head{g}{\sigma}$ or $\Tail{g}{\sigma}$ is nonempty. Thus, the combinatorial vector field below is well defined.

	\begin{defn}\label{def:GradientField}
		The \emph{gradient vector field} of a \mdm function $g:K\rightarrow\R^\maxdim$ is the discrete vector field $\cV$ such that
		\begin{align*}
			\cV(\sigma) = \begin{cases}
				\sigma & \text{ if $\sigma$ is critical},\\
				\gamma & \text{ if } \Head{g}{\sigma} = \{\gamma\}\text{ for some }\gamma>\sigma,\\
				\text{undefined} & \text{ if } \Tail{g}{\sigma} \neq\emptyset.
			\end{cases}
		\end{align*}
		For any admissible map $f$, we say $g$ is \emph{$f$-compatible} if its gradient vector field is itself $f$-compatible.
	\end{defn}

	\begin{rem}
		Our definition of a $f$-compatible \mdm function is different from that in \citep{Allili2019}. Indeed, for a \mdm function to be $f$-compatible in the sense of \citep{Allili2019}, its gradient field has to be exactly the vector field output by the algorithm \Matching presented therein. Here, we say a \mdm function is $f$-compatible if its gradient field is compatible with the multifiltration induced by $f$, independently of any algorithm.
	\end{rem}

	Notice that for a \mdm function $g$, we have $\Head{g}{\sigma} = \{\gamma\}$ iff $\Tail{g}{\gamma} = \{\sigma\}$. Hence, the gradient field $\cV$ of $g$ is such that $\cV(\sigma) = \gamma$ iff $\cV(\gamma)$ is undefined and $\cV^{-1}(\gamma) = \sigma$, and the critical simplices of $g$ are exactly the fixed points $\cV$. Thus, we often refer to the fixed points of a gradient field $\cV$ as the critical simplices of $\cV$ instead. Moreover, we know that a discrete vector field is acyclic if and only if it is the gradient field of some \mdm function \cite[Proposition 4.9]{Brouillette2022}.

	Furthermore, it was shown in \citep{Allili2019} that if $f:K\rightarrow\R^\maxdim$ is an admissible map which is the max-extension of a component-wise injective vertex map, there exists a $f$-compatible \mdm function. One of the main result of this paper, Theorem \ref{theo:gMDM}, extends this result. Indeed, we will show that for any admissible map $f:K\rightarrow\R^\maxdim$, we can algorithmically define a \mdm function $g:K\rightarrow\R^\maxdim$ which, in addition to being $f$-compatible, is obtained from small perturbations of $f$, and is therefore as close as desired to the initial input.

	\subsection{Morse inequalities and (relative) perfectness}\label{sec:MorseIneqRelativePerfect}

	One of the central results of Morse theory are the Morse inequalities. These still hold in the multiparameter discrete setting. Namely, let $m_p$ be the number of critical points of index $p$ of a \mdm function $g:K\rightarrow\R^\maxdim$. For all $p\in\Z$, we have
	\begin{gather*}
		m_p \geq \beta_p(K).
	\end{gather*}
	We say $g$ and its gradient field are \emph{perfect} when the equality holds for each $p$ \citep{Ayala2012, Fugacci2020}. Note that these inequalities are also valid for Betti numbers with coefficients $\beta_p(K;\Z_2)$.

	In the same vein, inequalities that apply specifically to a gradient field $\cV$ compatible with a multifiltering function $f$ are established by \citet{Landi2022}. Indeed, for each $u\in f(K)$ and each $p\in\Z$, the number $m_p(u)$ of critical simplices of $\cV$ in the level set $f^{-1}(u) =: L_u$ is bounded by
	\begin{gather*}
		m_p(u) \geq \rank H_p\left(K(u), \bigcup_{u'\precneqq u}K(u')\right).
	\end{gather*}
	When the equality holds for each $u\in f(K)$ and each $p\in\Z$, we say $\cV$ is \emph{relative-perfect} or \emph{perfect relatively to $f$}. The concept was first introduced in \citep{Fugacci2020} for real-valued filtering functions, while related inequalities have been proved in \citep{Guidolin2023} thereafter.

	\begin{ex}\label{ex:RelativePerfectCircle}
		Consider the triangulated circle $K$ embedded in $\R^2$ as in Figure \ref{fig:RelativePerfectCircle} and let $f:K\rightarrow\R^2$ be the max-extension of the map which associates each vertex to its coordinates $(x,y)\in\R^2$. In this particular case, $f$ happens to be a \mdm function, and its gradient field is as in Figure \ref{fig:RelativePerfectCircle}. The Betti numbers of the circle are $\beta_0(K)=\beta_1(K) = 1$, so $f$ has too many critical simplices to be a perfect \mdm function. Nonetheless, it is relative-perfect. Indeed, if $u\in\R^\maxdim$ is such that $m_p(u) = 0$ for each $p$, meaning that $L_u$ does not contain any critical simplex, then we trivially have the equality $m_p(u) = \rank H_p\left(K(u), \bigcup_{u'\precneqq u}K(u')\right) = 0$. Otherwise, if $u\in\R^\maxdim$ is such that $m_p(u)>0$ for some $p$, we can check that $L_u$ contains a single critical simplex $\sigma$ of index $p\in\{0,1\}$, i.e. either a critical vertex or a critical edge. We could then compute that $H_p\left(K(u), \bigcup_{u'\precneqq u}K(u')\right)\cong\Z$, so $m_p(u) = \rank H_p\left(K(u), \bigcup_{u'\precneqq u}K(u')\right) = 1$.

		\begin{figure}[ht]
			\centering
			\begin{tikzpicture}[scale=1.75]
				\draw[help lines, color=gray!50, dashed, step=0.5] (-1.2,-1.2) grid (1.2,1.2);
				\draw[->] (-1.2,0)--(1.2,0) node[right]{$x$};
				\draw[->] (0,-1.2)--(0,1.2) node[above]{$y$};
				\draw[ultra thick, orange] (0:1) -- (30:1)node{$\bullet$} -- (60:1)node{$\bullet$} -- (90:1);
				\draw[thick] (90:1)node{$\bullet$} -- (120:1)node{$\bullet$} -- (150:1)node{$\bullet$} -- (180:1);
				\draw[thick] (270:1) -- (300:1)node{$\bullet$} -- (330:1)node{$\bullet$} -- (360:1)node{$\bullet$};
				\draw[ultra thick, red] (180:1)node{$\bullet$} -- (210:1)node{$\bullet$} -- (240:1)node{$\bullet$} -- (270:1)node{$\bullet$};
				\draw[ultra thick, -stealth] (90:1)node{$\bullet$} -- ({(cos(90)+cos(120))/2},{(sin(90)+sin(120))/2});
				\draw[ultra thick, -stealth] (120:1)node{$\bullet$} -- ({(cos(150)+cos(120))/2},{(sin(150)+sin(120))/2});
				\draw[ultra thick, -stealth] (150:1)node{$\bullet$} -- ({(cos(150)+cos(180))/2},{(sin(150)+sin(180))/2});
				\draw[ultra thick, -stealth] (0:1)node{$\bullet$} -- ({(cos(0)+cos(330))/2},{(sin(0)+sin(330))/2});
				\draw[ultra thick, -stealth] (330:1)node{$\bullet$} -- ({(cos(300)+cos(330))/2},{(sin(300)+sin(330))/2});
				\draw[ultra thick, -stealth] (300:1)node{$\bullet$} -- ({(cos(270)+cos(300))/2},{(sin(270)+sin(300))/2});
			\end{tikzpicture}
			\caption{The gradient field $\cV$ of the \mdm function from Example \ref{ex:RelativePerfectCircle}. Pairs $v<\sigma$ such that $\cV(v) = \sigma$ are represented by the arrows, while the critical simplices are shown in red and orange.}\label{fig:RelativePerfectCircle}
		\end{figure}
	\end{ex}

	\subsection{Critical components}\label{sec:MDMcritComponents}

	From both theory \citep{Budney2023, Smale1975, Wan1975} and experiments \citep{Allili2019, Scaramuccia2018}, we deduce that in many cases, critical points of a \mdm function are not isolated as it is the case for classical Morse functions. They instead appear in clusters, as we can see in Example \ref{ex:RelativePerfectCircle}, where two obvious critical connected components (represented in red and orange in Figure \ref{fig:RelativePerfectCircle}) are observed for the considered \mdm function. However, experimentally, we can also notice that these clusters are not always connected \citep{Allili2019, Brouillette2022}. Thus, another way to partition critical simplices into components is proposed in \citep{Brouillette2022}.

	Let $R$ be a relation on the set $\cC$ of critical simplices of a \mdm function and consider its transitive closure $\bar{R}$. Namely, we have $\sigma\bar{R}\tau$ if there is a sequence $\sigma=\sigma_0,\sigma_1,...,\sigma_n=\tau\in\cC$ such that $\sigma_{i-1}R\sigma_i$ for each $i=1,...,n$. If $R$ is reflexive and symmetric, by definition, it follows that $\bar{R}$ is an equivalence relation, so it can be used to partition $\cC$. We then call a critical component (with respect to $\bar{R}$) an equivalence class of $\bar{R}$. As an example, if $R$ is such that $\sigma R\tau\Leftrightarrow \sigma\leq\tau\text{ or }\sigma\geq\tau$, then $\bar{R}$ yields the partition of $\cC$ into connected components (with respect to the Alexandrov topology).

	As mentioned above, connectedness is too strict of a criteria to define critical components. Thus, we consider the following definition, first introduced as Definition 7.3 in \citep{Brouillette2022}. It makes use of the idea of dynamical connectedness between critical simplices. More precisely, for two critical simplices $\sigma$ and $\tau$, we write $\sigma\rightconnects{g}\tau$ when $\sigma$ is connected to $\tau$ in the flow induced by the gradient field $\cV$ of $g$. This is the case, for example, when $\sigma\geq\tau$ or when there is a $\cV$-path going from a face of $\sigma$ to a coface of $\tau$. For a formal definition of the flow of $\cV$, see \citep{Brouillette2022}.
	\pagebreak
	\begin{defn}\label{def:CriticalComponentSimG}
		Let $g:K\rightarrow\R^\maxdim$ be a \mdm function and $\cC$ be the set of its critical simplices. Consider the relation $R_g$ defined on $\cC$ such that $\sigma R_g\tau$ if
		\begin{enumerate}
			\item\label{item:DefR_g1} $g_i(\sigma) = g_i(\tau)$ for some $i = 1,...,\maxdim$;
			\item\label{item:DefR_g2} either $\sigma\rightconnects{g}\tau$ or $\sigma\leftconnects{g}\tau$.
		\end{enumerate}
		Then, $\Sim{g}:=\bar{R}_g$ is an equivalence relation, and we define a \emph{critical component} of $g$ (with respect to $\Sim{g}$) as an equivalence class of $\Sim{g}$.
	\end{defn}

	The intuition behind condition \ref{item:DefR_g1} is that, when $g$ is a multifiltering function, then it is possible for $\sigma$ and $\tau$ to enter the induced multifiltration at a same "step" if $g_i(\sigma) = g_i(\tau)$ for some $i$. Also, it is well known from Morse-Forman theory \citep{Forman1998} that the homology of a simplicial complex may be computed from the critical simplices of a gradient field defined on it and the connections between them. Thus, condition \ref{item:DefR_g2} implies that $\sigma$ and $\tau$ are connected and can interact with each other on the homological level, which motivates the definition of $\Sim{g}$.

	For toy examples, defining critical components using $\Sim{g}$ yields desirable results, as argued in \citep{Brouillette2022}. For instance, the \mdm function from Example \ref{ex:RelativePerfectCircle} has two critical components with respect to $\Sim{g}$, represented in red and orange in Figure \ref{fig:RelativePerfectCircle}. In Section \ref{sec:ExperimentsCritComp}, we will see how this definition of critical components of a \mdm function performs on a larger scale. Also, a new definition of critical components will be proposed, Definition \ref{def:CritComponentsF}, which is adapted specifically to \mdm functions which are compatible with a given multifiltering function.

	\section{Algorithms}\label{sec:Algorithms}

	In this section, we see how a multifiltering function $f:K\rightarrow\R^\maxdim$ defined on a simplicial complex $K$ can be used to build a compatible \mdm function $g:K\rightarrow\R^\maxdim$. The proposed algorithms combine elements of those in \citep{Allili2019, Robins2011, Scaramuccia2020}.

	\subsection{Description}\label{sec:AlgoDescription}

	We explain here the general idea of the algorithms.

	First of all, apart from the complex $K$ and the map $f$, the main algorithm takes two input parameters: an injective real-valued map $I:K\rightarrow\R$, called an \emph{indexing map}, and a small parameter $\epsilon>0$. The indexing map $I$ is used to order the simplices of $K$ while the parameter $\epsilon$ is used to ensure the \mdm function $g$ computed is such that, for all $\sigma\in K$ and all $i=1,...,\maxdim$,
	\begin{gather}
		|g_i(\sigma) - f_i(\sigma)| < \epsilon.\label{eq:g-f<eps}
	\end{gather}

	The indexing map is \emph{admissible} if it satisfies the following property:
	\begin{gather*}
		\alpha < \sigma\Rightarrow I(\alpha) < I(\sigma).
	\end{gather*}
	There are various ways to define an admissible indexing map. A simple one is to assign a number to each simplex as $K$ is constructed. More precisely, a simplicial complex $K$ can be implemented so that each time a new simplex $\sigma$ is inserted in $K$, we ensure all its faces $\alpha<\sigma$ are inserted beforehand and we define $I(\sigma)$ as the number of simplices that were inserted in $K$ before $\sigma$. Thus, by proceeding this way, we include the computing of $I$ in the implementation of $K$. Note that the choice of indexing map may affect the output of the algorithm. This will be discussed in more details in Section \ref{sec:ExperimentsTriangulationIndexingMap}.

	In order to generate a \mdm function $g:K\rightarrow\R^\maxdim$, we opt for a divide-and-conquer approach by partitioning the domain $K$ into level sets $\{L_u\ |\ u\in f(K)\}$. To compute the partition, we use an auxiliary function \LevelSets which produces a dictionary associating each $u\in f(K)$ to its level set $L_u$. More specifically, \LevelSets creates an empty dictionary and then, for each $\sigma\in K$, either adds $\sigma$ to $L_{f(\sigma)}$ if defined or creates an entry associating $f(\sigma)$ to $\{\sigma\}$ if $L_{f(\sigma)}$ is yet undefined. The function \LevelSets returns the level sets of $K$ ordered so that for any $u,u'\in f(K)$, if $u'\precneqq u$, then $L_{u'}$ comes before $L_u$. This can be done by using the lexicographical order on $f(K)\subset\R^\maxdim$.

	Now, we explain how \GenerateMDM, as described in Algorithm \ref{algo:GenerateMDM}, generates the \mdm function $g:K\rightarrow\R^\maxdim$ and its gradient field $\cV:K\nrightarrow K$, represented by dictionaries. First, in order to define $g$, a parameter $\delta > 0$ is computed. To ensure $g$ verifies equation \ref{eq:g-f<eps}, we choose $\delta$ so that $\delta\leq \frac{\epsilon}{|K|}$ and $\delta \leq \frac{|f_1(\sigma)-f_1(\tau)|}{|K|}$ for all $\sigma,\tau\in K$ with $f_1(\sigma)\neq f_1(\tau)$. The reason why $\delta$ is defined as such will become clear in the proofs of Proposition \ref{prop:ComputeGworks} and Theorem \ref{theo:gMDM}. Then, function \LevelSets partitions the complex $K$ into level sets $L_u$ and both $g$ and $\cV$ are generated locally on each $L_u$. As mentioned above, the level sets are considered in some order such that if $u'\precneqq u$, then $L_{u'}$ comes before $L_u$.

	\begin{algorithm}\caption{GenerateMDM($K,f,I, \epsilon$)}\label{algo:GenerateMDM}
		\begin{algorithmic}[1]
			\Require A finite simplicial complex $K$; an admissible map $f:K\rightarrow\R^\maxdim$; an admissible index mapping $I$; a parameter $\epsilon > 0$.
			\Ensure Two dictionaries $g$ and $\cV$, representing a \mdm function $g:K\rightarrow\R^\maxdim$ and its gradient field $\cV:K\nrightarrow K$.
			\State Define $g$ and $\cV$ as two empty dictionaries
			\If{$f_1$ is constant}
			\State $\delta = \frac{\epsilon}{|K|}$\label{algoLine:DefineDeltaF1Constant}
			\Else
			\State $\delta = \frac{1}{|K|}\cdot\min\{\epsilon, \varepsilon\}$, where $\varepsilon \leq \min\lbrace |f_1(\sigma)-f_1(\tau)| : f_1(\sigma)\neq f_1(\tau)\rbrace$\label{algoLine:DefineDeltaF1NonConstant}
			\EndIf
			\ForAll{$L_u\in \LevelSets(K,f)$}\label{algoLine:LevelSets}
			\State $(g,\cV) = \ExpandMDM(f, g, \cV, L_u, I, \delta)$\label{algoLine:ExpandMDM}
			\EndFor\\
			\Return $(g, \cV)$
		\end{algorithmic}
	\end{algorithm}

	The construction of $g$ and $\cV$ on some level set $L_u$ is done by \ExpandMDM. The function is described in Algorithm \ref{algo:ExpandMDM} and follows the idea of Algorithm 2 in \citep{Scaramuccia2020}, which itself uses elements of algorithms in \citep{Allili2019, Robins2011}.

	We consider two (minimal) priority queues \PQzero and \PQone and a dictionary \processed associating each $\sigma\in L_u$ to a boolean value, which is initially \False for each $\sigma$. Then, we add to \PQzero and \PQone all simplices in $L_u$ with respectively zero and one facet in $L_u$ still unprocessed by the algorithm. Every simplex $\sigma$ in either queue is given priority $I(\sigma)$. Next, for as long as \PQone is nonempty, we consider $\sigma\in\PQone$ with minimal priority, that is with minimal value $I(\sigma)$, and its unique unprocessed facet $\tau\in L_u$, if it is still unprocessed. The two simplices are then \emph{processed as a pair}, meaning that we set $\cV(\tau)=\sigma$ and define both $g(\sigma)$ and $g(\tau)$ using the same value in $\R^\maxdim$. The values of $g$ are set using function \ComputeG, presented in Algorithm \ref{algo:ComputeG}, which will be explained in details below. After processing $\sigma$ and $\tau$, routine \AddCofacets updates \PQone by adding to it all their cofacets $\gamma\in L_u$ with $\NumUnprocFacets(\gamma,L_u)=1$. When \PQone is empty, meaning that no simplex is available for pairing, then the simplex $\sigma\in\PQzero$ with minimal $I(\sigma)$ is treated similarly: it is \emph{processed as critical} by defining $\cV(\sigma) = \sigma$, the value $g(\sigma)$ is set using the function \ComputeG and \PQone is updated with routine \AddCofacets. This process ends when both \PQone and \PQzero are empty.

	\begin{algorithm}\caption{ExpandMDM($f, g,\cV,L_u,I,\delta$)}\label{algo:ExpandMDM}
		\begin{algorithmic}[1]
			\Require An admissible map $f:K\rightarrow\R^\maxdim$; one of its level set $L_u$; an admissible index mapping $I$; a parameter $\delta > 0$; two dictionaries $g$ and $\cV$, representing a \mdm function and its gradient field, to be defined on $L_u$.
			\Ensure Dictionaries $g$ and $\cV$ extended to $L_u$.
			\State Define \processed as a dictionary associating each $\sigma\in L_u$ to the boolean value \False
			\State Define \PQzero and \PQone as two empty priority queues
			\ForAll{$\sigma\in L_u$}\label{algoLine:InitialAddLoopBegins}
			\If{$\NumUnprocFacets(\sigma,L_u)=0$} \label{algoLine:InitialCheckZero}
			\State add $\sigma$ to \PQzero \label{algoLine:InitialAddToPQzero}
			\ElsIf{$\NumUnprocFacets(\sigma,L_u)=1$} \label{algoLine:InitialCheckOne}
			\State add $\sigma$ to \PQone \label{algoLine:InitialAddToPQone}
			\EndIf
			\EndFor
			\While{$\PQone\neq\emptyset$ or $\PQzero\neq\emptyset$} \label{algoLine:LoopPQonePQzeroBegins}
			\While{$\PQone\neq\emptyset$}\label{algoLine:LoopPQoneBegins}
			\State $\sigma = \texttt{PQone.pop\_front}$ \label{algoLine:PopFromPQone}
			\If{$\NumUnprocFacets(\sigma,L_u)=0$}\label{algoLine:CheckZero}
			\State add $\sigma$ to \PQzero \label{algoLine:MoveFromPQoneToPQzero}
			\Else
			\State $\tau = \UnprocFacet(\sigma,L_u)$ \label{algoLine:UnprocessedFacet}
			\State define $g(\sigma)=g(\tau) = \ComputeG(f,g,\delta,\sigma,\tau)$ \label{algoLine:DefineGpair}
			\State define $\cV(\tau)=\sigma$ \label{algoLine:DefineVpair}
			\State $\processed(\tau) = \True$; $\processed(\sigma) = \True$ \label{algoLine:DeclareProcessAfterPairing}
			\State $\AddCofacets(\sigma,L_u, I, \PQone)$; $\AddCofacets(\tau,L_u, I, \PQone)$ \label{algoLine:AddCofacetsAfterPairing}
			\EndIf
			\EndWhile
			\If{$\PQzero\neq\emptyset$} \label{algoLine:PQzeroNeqEmpty}
			\State $\sigma = \texttt{PQzero.pop\_front}$ \label{algoLine:PopFromPQzero}
			\If{$\processed(\sigma) = \False$}
			\State define $g(\sigma) = \ComputeG(f,g,\delta,\sigma)$ \label{algoLine:DefineGcrit}
			\State define $\cV(\sigma)=\sigma$ \label{algoLine:DefineVcrit}
			\State $\processed(\sigma) = \True$ \label{algoLine:DeclareProcessAfterCritical}
			\State $\AddCofacets(\sigma,L_u, I, \PQone)$ \label{algoLine:AddCofacetsAfterCritical}
			\EndIf
			\EndIf
			\EndWhile\\
			\Return $(g,\cV)$
		\end{algorithmic}
	\end{algorithm}

	Finally, each simplex $\sigma\in K$ is associated to a value $g(\sigma)\in\R^\maxdim$ using the function \ComputeG, as described by Algorithm \ref{algo:ComputeG}. Using the map $f$, the parameter $\delta>0$ and the dictionary $g$, which is not yet defined on a simplex $\sigma$ and one of its facet $\tau<\sigma$, \ComputeG computes and returns a value $w\in\R^\maxdim$ used to define $g(\sigma)$ and $g(\tau)$. Note that the facet $\tau$ is optional: the algorithm works as well when only a simplex $\sigma$ is given. The algorithm proceeds differently whether $\sigma$ is a vertex or not. When considering a vertex $\sigma$, \ComputeG simply returns $f(\sigma)$. Otherwise, $w$ is computed as follows. We consider the set $A$ of facets $\alpha$ of $\sigma$ such that $\alpha\neq\tau$. If no simplex $\tau$ is given to \ComputeG as parameter, $A$ is simply the set of facets of $\sigma$. Then, let $w = (w_1,...,w_\maxdim)$ where
	\begin{align*}
		w_1 &= \max\left(\{f_1(\sigma)\}\cup\{g_1(\alpha)\ |\ \alpha\in A\}\right),\\
		w_i &= f_i(\sigma)\text{ for } i=2,...,\maxdim.
	\end{align*}
	If $w=g(\alpha)$ for some $\alpha\in A$, then we increase by $\delta$ the value of $w_1$. By doing so, we have that $\alpha\in \Tail{g}{\sigma}\Rightarrow \alpha=\tau$, which ensures conditions \ref{enum:defMDMenum1} and \ref{enum:defMDMenum2} of Definition \ref{def:MDM} of a \mdm function are satisfied, as we will see in the proofs of Proposition \ref{prop:ComputeGworks} and Theorem \ref{theo:gMDM}. Finally, $w$ is returned and used in function \ExpandMDM to define $g(\sigma)$ and $g(\tau)$.

	\begin{algorithm}\caption{ComputeG($f,g,\delta,\sigma,\tau$)}\label{algo:ComputeG}
		\begin{algorithmic}[1]
			\Require An admissible map $f:K\rightarrow\R^\maxdim$; a parameter $\delta>0$; a simplex $\sigma$; a facet $\tau<\sigma$ (optional); a dictionary $g$ with $g(\alpha)$ defined for all facets $\alpha<\sigma$ such that $\alpha\neq \tau$.
			\Ensure A value $w\in\R^\maxdim$.
			\If {$\sigma$ is a vertex}
			\State $w = f(\sigma)$ \label{algoLine:DefU=fSigma}
			\Else \label{algoLine:ElseNotVertex}
			\State $A = \left\lbrace \text{ facets } \alpha<\sigma\ |\ \alpha\neq \tau\right\rbrace$
			\State $w_1 = \max\left(\{f_1(\sigma)\}\cup\{g_1(\alpha)\ |\ \alpha\in A\}\right)$ \label{algoLine:DefineU1}
			\State $w_i=f_i(\sigma)$ for $i=2,...,\maxdim$
			\State $w = (w_1,...,w_\maxdim)$ \label{algoLine:DefineU}
			\If{$w = g(\alpha)$ for some $\alpha\in A$}\label{algoLine:IfU=GAlpha}
			\State $w_1 \mathrel{+}= \delta$\label{algoLine:U1+=Delta}
			\EndIf
			\EndIf\\
			\Return $w$
		\end{algorithmic}
	\end{algorithm}

	\begin{figure}[ht]
		\centering
		\subcaptionbox{\label{figSub:AlgorithmExampleI}}[0.32\textwidth]{
			\centering
			\begin{tikzpicture}[scale=2]
				\coordinate (A) at (0, 0);
				\coordinate (B) at (1, 0);
				\coordinate (C) at (0, 1);
				\coordinate (D) at (1, 1);

				\node[white, opacity=0, below] at ($0.5*(A)+0.5*(B)$) {\small $(1,2)$};

				\fill[black!10] (A) rectangle (D);
				\draw[thick] (A) -- (B);
				\draw[thick] (A) -- (C);
				\draw[thick] (A) -- (D);
				\draw[thick] (B) -- (D);
				\draw[thick] (C) -- (D);

				\node at (A){$\bullet$};
				\node at (B){$\bullet$};
				\node at (C){$\bullet$};
				\node at (D){$\bullet$};

				\node[below left] at (A) {$v_1$};
				\node[below right] at (B) {$v_2$};
				\node[above left] at (C) {$v_0$};
				\node[above right] at (D) {$v_3$};

				\node[below] at ($0.5*(A)+0.5*(B)$) {$e_5$};
				\node[left] at ($0.5*(A)+0.5*(C)$) {$e_4$};
				\node[left] at ($0.5*(A)+0.5*(D)$) {$e_8$};
				\node[right] at ($0.5*(B)+0.5*(D)$) {$e_6$};
				\node[above] at ($0.5*(C)+0.5*(D)$) {$e_7$};

				\node at ($0.25*(A)+0.5*(B)+0.25*(D)$) {$t_{10}$};
				\node at ($0.25*(A)+0.5*(C)+0.25*(D)$) {$t_9$};
			\end{tikzpicture}
		}
		\subcaptionbox{\label{figSub:AlgorithmExampleF}}[0.32\textwidth]{
			\centering
			\begin{tikzpicture}[scale=2]
				\coordinate (A) at (0, 0);
				\coordinate (B) at (1, 0);
				\coordinate (C) at (0, 1);
				\coordinate (D) at (1, 1);

				\fill[orange!25] (A) rectangle (D);
				\draw[ultra thick, purple] (A) -- (B);
				\draw[ultra thick, purple] (A) -- (C);
				\draw[ultra thick, orange] (A) -- (D);
				\draw[ultra thick, blue] (B) -- (D);
				\draw[ultra thick, blue] (C) -- (D);

				\node[purple] at (A){$\bullet$};
				\node[green] at (B){$\bullet$};
				\node[green] at (C){$\bullet$};
				\node[blue] at (D){$\bullet$};

				\node[below left] at (A) {\small $(1,2)$};
				\node[below right] at (B) {\small $(0,0)$};
				\node[above left] at (C) {\small $(0,0)$};
				\node[above right] at (D) {\small $(2,1)$};

				\node[below] at ($0.5*(A)+0.5*(B)$) {\small $(1,2)$};
				\node[left] at ($0.5*(A)+0.5*(C)$) {\small $(1,2)$};
				\node[xshift=-0.2in] at ($0.5*(A)+0.5*(D)$) {\small $(2,2)$};
				\node[right] at ($0.5*(B)+0.5*(D)$) {\small $(2,1)$};
				\node[above] at ($0.5*(C)+0.5*(D)$) {\small $(2,1)$};

				\node at ($0.25*(A)+0.5*(B)+0.25*(D)$) {\small $(2,2)$};
				\node at ($0.25*(A)+0.5*(C)+0.25*(D)$) {\small $(2,2)$};
			\end{tikzpicture}
		}
		\subcaptionbox{\label{figSub:AlgorithmExampleOut}}[0.32\textwidth]{
			\centering
			\begin{tikzpicture}[scale=2]
				\coordinate (A) at (0, 0);
				\coordinate (B) at (1, 0);
				\coordinate (C) at (0, 1);
				\coordinate (D) at (1, 1);

				\fill[red!25] (A) -- (B) -- (D) -- cycle;
				\fill[black!10] (A) -- (C) -- (D) -- cycle;

				\draw[ultra thick, red] (A) -- (B);
				\draw[thick] (A) -- (C);
				\draw[thick] (A) -- (D);
				\draw[thick] (B) -- (D);
				\draw[ultra thick, red] (C) -- (D);

				\draw[ultra thick, -stealth] (A) -- ($0.5*(A)+0.5*(C)$);
				\draw[ultra thick, -stealth] (D) -- ($0.5*(B)+0.5*(D)$);
				\draw[ultra thick, -stealth] ($0.5*(A)+0.5*(D)$) -- ($0.33*(A)+0.33*(C)+0.33*(D)$);

				\node at (A){$\bullet$};
				\node[red] at (B){$\bullet$};
				\node[red] at (C){$\bullet$};
				\node at (D){$\bullet$};

				\node[below left] at (A) {\small $(1,2)$};
				\node[below right] at (B) {\small $(0,0)$};
				\node[above left] at (C) {\small $(0,0)$};
				\node[above right] at (D) {\small $(2,1)$};

				\node[below] at ($0.5*(A)+0.5*(B)$) {\small $(1+\delta,2)$};
				\node[left] at ($0.5*(A)+0.5*(C)$) {\small $(1,2)$};
				\node[right] at ($0.5*(B)+0.5*(D)$) {\small $(2,1)$};
				\node[above] at ($0.5*(C)+0.5*(D)$) {\small $(2+\delta,1)$};

				\node[below] at ($0.35*(A)+0.35*(B)+0.3*(D)$) {\small $(2+2\delta,2)$};
				\node at ($0.3*(A)+0.35*(C)+0.35*(D)$) {\small $(2+\delta,2)$};
			\end{tikzpicture}
		}
		\caption{In \subref{figSub:AlgorithmExampleI}, a simplicial complex $K$ with its simplices labelled in accordance with some admissible index map $I$. In \subref{figSub:AlgorithmExampleF}, an admissible map $f$ defined on $K$. Its four level sets are represented in green, purple, blue and orange. In \subref{figSub:AlgorithmExampleOut}, the output $g$ and $\cV$ of \GenerateMDM, where $\delta$ is as defined in Algorithm \ref{algo:GenerateMDM}. The red simplices represent those that are critical.} \label{fig:AlgorithmExample}
	\end{figure}

	An example of output for \GenerateMDM is shown in Figure \ref{fig:AlgorithmExample}.

	\begin{rem}
		The procedure \ComputeG returns a value using all facets $\alpha$ of a simplex $\sigma$, where each $\alpha$ may or may not be in the same level set as $\sigma$. That being said, \ExpandMDM only pairs together simplices that belong in a same level set. Hence, if we were to adapt the main algorithm \GenerateMDM in order to produce only a gradient field, it would be fairly easy to parallelize the procedure by calling \ExpandMDM separately on each level set.
	\end{rem}

	\subsection{Correctness}\label{sec:Correctness}

	We now explain broadly why the proposed algorithms produce the desired result. The formal proofs are available in Appendix \ref{sec:ProofCorrectness}.

	First, we consider the routine \ExpandMDM, described in Algorithm \ref{algo:ExpandMDM}. When \ExpandMDM is called, we see that all simplices with $0$ or $1$ facet in $L_u$, facet which is necessarily unprocessed at this point, immediately enter \PQzero or \PQone. These simplices eventually get processed, so the number of unprocessed facets of their cofacets gradually decreases. When the number of unprocessed facets of a simplex reaches $1$, it is added to \PQone by the subroutine \AddCofacets, so it can also eventually be processed. In other words, \ExpandMDM first processes the lower-dimensional simplices and slowly works its way up to higher-dimensional simplices until all of $L_u$ is processed. Since the processed simplices never re-enter \PQzero or \PQone, we obtain the next result.

	\begin{restatable}{re-prop}{propProcessedOnce}\label{prop:ProcessedOnce}
		For all $u\in f(K)$, each simplex in $L_u$ is processed exactly once by Algorithm \ref{algo:ExpandMDM}.
	\end{restatable}

	Hence, since Proposition \ref{prop:ProcessedOnce} stands for every level set $L_u\subset K$ and each $L_u$ is processed exactly once by the main routine \GenerateMDM, presented in Algorithm \ref{algo:GenerateMDM}, we deduce that every $\sigma\in K$ is processed exactly once by \GenerateMDM.

	Moreover, recall that if two simplices $\alpha<\sigma$ are part of a same priority queue, $\alpha$ has priority over $\sigma$ because $I(\alpha)<I(\sigma)$. Hence, from the reasoning leading to Proposition \ref{prop:ProcessedOnce}, we deduce the following key observation: the facets of a given simplex $\sigma$ are always processed either before $\sigma$, or at the same time by being paired by \ExpandMDM. Therefore, every time \ComputeG is called in Algorithm \ref{algo:ExpandMDM} to define some value $g(\sigma)$, and optionally $g(\tau)$ when $\sigma$ is paired with a facet $\tau$, we know the value $g(\alpha)$ has been defined beforehand for all facets in $A = \{\text{facets }\alpha<\sigma\ |\ \alpha\neq\tau\}$. Thus, \ComputeG always returns a value $w\in\R^\maxdim$ without error. More specifically, we can characterize $g$ as follows.

	\begin{restatable}{re-prop}{propComputeGworks}\label{prop:ComputeGworks}
		Consider function \ComputeG described in Algorithm \ref{algo:ComputeG}.
		\begin{enumerate}
			\item\label{enum:propComputeG1} When called in Algorithm \ref{algo:ExpandMDM}, \ComputeG returns a value $w\in\R^\maxdim$ without error, so the output dictionary $g$ of Algorithm \ref{algo:GenerateMDM} is such that $g(\sigma)$ is well defined for all $\sigma\in K$.
			\item\label{enum:propComputeG2} For all facets $\alpha$ of $\sigma\in K$, we have $g(\alpha)\preceq g(\sigma)$, where $g(\alpha) = g(\sigma)$ if and only if $\sigma$ and $\alpha$ were paired by \ExpandMDM.
			\item\label{enum:propComputeG3} For all $\sigma\in K$, we have
			\begin{align*}
				g_1(\sigma) &= f_1(\sigma) + m \delta \text{ for some } m\in\N,\\
				g_i(\sigma) &= f_i(\sigma) \text{ for each } i=2,...,\maxdim
			\end{align*}
			where $m\geq 0$ is bounded by the number of simplices that were processed before $\sigma$.
		\end{enumerate}
	\end{restatable}

	Now, we could show that the dictionary $\cV$ output by \GenerateMDM pairs together two simplices $\tau<\sigma$ if and only if $\sigma$ and $\tau$ belong in a same level set $L_u$ and $g(\tau)=g(\sigma)$. Otherwise, when $\tau<\sigma$ are not paired, we have $g(\tau)\precneqq g(\sigma)$. Since every simplex is processed exactly once, it belongs in at most one pair defined by $\cV$, and we can deduce that $g$ is in fact a \mdm function. Also, from part \ref{enum:propComputeG3} of Proposition \ref{prop:ComputeGworks}, we see that
	\begin{gather*}
		\lVert g(\sigma) - f(\sigma)\rVert = |g_1(\sigma) - f_1(\sigma)| = m\delta < |K|\delta \leq \epsilon
	\end{gather*}
	because $\delta$ is defined in Algorithm \ref{algo:GenerateMDM} so that $\delta\leq\frac{\epsilon}{|K|}$. Thus, we conclude that \GenerateMDM produces the desired result, as stated in the following theorem.

	\begin{restatable}{re-thm}{theogMDM}\label{theo:gMDM}
		For any simplicial complex $K$, multifiltering function $f$, admissible indexing map $I:K\rightarrow\R$ and $\epsilon>0$ given as input in \GenerateMDM, the outputs $g$ and $\cV$ correspond to a $f$-compatible \mdm function $g:K\rightarrow\R^\maxdim$ and its gradient field $\cV:K\nrightarrow K$ such that, for all $\sigma\in K$,
		\begin{gather*}
			\lVert g(\sigma) - f(\sigma)\rVert < \epsilon
		\end{gather*}
		where $\lVert\cdot\rVert$ may be any $p$-norm on $\R^\maxdim$.
	\end{restatable}

	Furthermore, we can see \GenerateMDM as an algorithm which outputs a \mdm function $g$ from an input admissible map $f$ by dividing its level sets into pairs and singletons. When $f$ is itself \mdm, the connected components of its level sets are all of cardinality $1$ or $2$, so \GenerateMDM simply returns $g:=f$.

	\begin{restatable}{re-prop}{propfMDMiffFisG}\label{prop:fMDMiffFisG}
		Let $g:K\rightarrow\R^\maxdim$ be the \mdm function produced by \GenerateMDM when given $f:K\rightarrow\R^\maxdim$ as input. The function $f$ is itself \mdm if and only if $f=g$.
	\end{restatable}

	\subsection{Complexity analysis}\label{sec:ComplexityAnalysis}

	To compute the complexity of each algorithm, we make the following assumptions:
	\begin{itemize}
		\item For each $\sigma\in K$, the set of its facets is stored in the structure implementing $K$, so it can be accessed in constant time.
		\item Inserting an entry, namely a key and its associated value, to a dictionary is done in constant time. Accessing an entry in a dictionary also requires constant time. These assumptions are generally verified when the dictionary is implemented with a suitable hash table \citep{Cormen2022}.
		\item Both \PQzero and \PQone are priority queues implemented so that removing the element $\sigma$ with minimal $I(\sigma)$ requires constant time, while inserting an element is logarithmic in the size of the queue.
		\item The vector function $f$ is given as input, thus precomputed. A value $f(\sigma)$ is accessed in constant time.
		\item The index mapping $I$ is precomputed. As explained previously at the beginning of Section \ref{sec:AlgoDescription}, the computation of $I$ can be included in the implementation of $K$.
	\end{itemize}

	In \GenerateMDM, there are three instructions which could potentially be carried out in nonconstant time: defining $\varepsilon$ at line \ref{algoLine:DefineDeltaF1NonConstant} of Algorithm \ref{algo:GenerateMDM}, calling \LevelSets at line \ref{algoLine:LevelSets} and calling \ExpandMDM for each $L_u$ at line \ref{algoLine:ExpandMDM}. In practice, since we are working with floating-point numbers, $\varepsilon$ may simply be defined as a small value used as a threshold to compare values of $f_1(K)\subset\R$, which is done in constant time. Then, to split $K$ into level sets, the function \LevelSets creates a dictionary which associates each $u\in f(K)$ to $L_u$. To construct such a dictionary, for each $\sigma\in K$, \LevelSets either adds $\sigma$ to $L_u$ if it is defined or creates an entry associating $u$ to $\{\sigma\}$ otherwise. In both cases, this takes constant time, so the dictionary is defined in $O(|K|)$ time. Since the keys $u\in f(K)$ of the dictionary then have to be sorted so that $u'$ comes before $u$ when $u'\precneqq u$, we deduce that calling \LevelSets has a computational cost of $O(|K| + |f(K)|\log |f(K)|)$.

	Moreover, the complexity of \GenerateMDM depends greatly on the execution of \ExpandMDM. For each level set $L_u$, we know from Section \ref{sec:Correctness} that each $\sigma\in L_u$ is added to each priority queue \PQzero and \PQone at most once and is processed exactly once. We compute the cost of processing each $\sigma\in L_u$ by considering each step of its process separately.
	\begin{enumerate}
		\item Inserting $\sigma$ in a priority queue takes logarithmic time in the size of the queue, which is bounded by $|L_u|$, so the computational cost of adding $\sigma$ to \PQzero and \PQone is at most $O(2\log |L_u|)$.
		\item If $\sigma$ is paired with a facet $\tau$, retrieving $\tau$ takes constant time.
		\item From Algorithm \ref{algo:ComputeG}, we see that to define $g(\sigma)$, in the worst case, \ComputeG needs to compute $\max\left(\{f_1(\sigma)\}\cup\{g_1(\alpha)\ |\ \alpha\in A\}\right)$, access $f_i(\sigma)$ for $i=2,...,\maxdim$ and compare the resulting vector value to $g(\alpha)$ for all $\alpha\in A$. Because $|A|\leq \dim\sigma+1 \leq d+1$, where $d:=\dim K$, each of these operations is done in $O(d\maxdim)$ time at most.
		\item Associating the key $\sigma$ to a value in the dictionaries $g$ and $\cV$ and declaring $\processed(\sigma)=\True$ requires constant time.
		\item The cost of \AddCofacets when applied to $\sigma$ depends on its number of cofacets, which may be quite large when $K$ is an arbitrary simplicial complex. Nonetheless, we know that each $\gamma\in L_u$ will be checked by \AddCofacets at most as many times as the number of facets of $\gamma$, which is bounded by $d+1$. Hence, we can approximate of the cost of \AddCofacets is $O(d)$ for each simplex in $L_u$.
	\end{enumerate}
	Thus, overall, for each $\sigma\in L_u$, \ExpandMDM has a computational cost of
	\begin{gather*}
		O\left(2\log|L_u| + d\maxdim + d\right) = O(\log|L_u| + d\maxdim).
	\end{gather*}
	Therefore, the function \ExpandMDM takes at most $O\left(|L_u|\log|L_u| + d\maxdim|L_u|\right)$ time to execute on a given level set $L_u$, and
	\begin{gather*}
		O\left(\sum_{u\in f(K)}\left(|L_u|\log|L_u| + d\maxdim|L_u|\right)\right)
	\end{gather*}
	time to process all level sets. Also, we see that $\sum_{u\in f(K)} d\maxdim|L_u| = d\maxdim|K|$ because the level sets $L_u$ form a partition of $K$. Moreover, let $\lambda:=\max_{u\in f(K)}|L_u|$ be the size of the largest level set in $K$. Then, $\sum_{u\in f(K)}|L_u|\log|L_u|\leq |f(K)|\lambda\log\lambda = |f(K)|\log\lambda^\lambda$, and it follows that the computational cost of \ExpandMDM may be overestimated by
	\begin{gather*}
		O\left(|f(K)|\log\lambda^\lambda + d\maxdim|K|\right).
	\end{gather*}

	Finally, by adding the costs of \LevelSets and \ExpandMDM, we obtain a total cost for \GenerateMDM of $O\left(|K| + |f(K)|\log |f(K)| + |f(K)|\log\lambda^\lambda + d\maxdim|K|\right)$, which can be rewritten as
	\begin{gather*}
		O\left(d\maxdim|K| + |f(K)|\log\lambda^\lambda|f(K)|\right).
	\end{gather*}
	Since $\lambda^\lambda|f(K)|\leq \left(\lambda|f(K)|\right)^\lambda$, we conclude with the following result.

	\begin{prop}\label{prop:ComputationalCost}
		Let $f:K\rightarrow\R^\maxdim$ be an admissible input map with $d:= \dim K$ and $\lambda:=\max_{u\in f(K)}|L_u|$. The computational complexity of \GenerateMDM is
		\begin{gather*}
			O\left(d\maxdim|K| + \lambda|f(K)|\log\lambda|f(K)|\right).
		\end{gather*}
	\end{prop}

	It is possible to refine the result by making extra assumptions. For instance, if we assume that $d$ and $\maxdim$ are small constants, we find the algorithm executes in $O\left(|K| + \lambda|f(K)|\log\lambda|f(K)|\right)$ time. Furthermore, in most applications, the considered input functions induce many small level sets, so $|f(K)|$ is generally a great value while $\lambda$ is quite small. Considering this, we can approximate the running time of \GenerateMDM by $O\left(|K| + |f(K)|\log|f(K)|\right)$.

	\subsection{Comparison with previous algorithms}\label{sec:PreviousAlgorithms}

	The main ideas of \GenerateMDM and \ExpandMDM essentially come from the algorithms \ComputeDiscreteGradient and \Matching proposed by \citet{Scaramuccia2020} and \citet{Allili2019} respectively. Recall that they use as input a simplicial complex $K$ and a component-wise injective vertex map $f:K_0\rightarrow\R^\maxdim$ and output a discrete gradient field compatible with the max-extension of $f$. Hence, in comparison with these algorithms, \GenerateMDM has three main advantages:
	\begin{enumerate}
		\item In addition to a discrete gradient field, it computes a compatible \mdm function.
		\item It does not require the input map $f$ or any of its components to be injective.
		\item It can process not only the max-extension of a vertex map, but also any multifiltering function defined on $K$. This is convenient, notably, when interested in filtering $K$ using a radius map (see end of Section \ref{sec:Multipersistence} for more details).
	\end{enumerate}
	It is also worth noting that \GenerateMDM and \ExpandMDM could easily be adapted to generate only either a \mdm function or a discrete gradient field if needed, making the proposed approach quite versatile.

	Furthermore, since an arbitrary vertex map $f$ may always be perturbed slightly to be made component-wise injective, relaxing the injectivity hypothesis on $f$ may seem like a negligible gain over the previous algorithms. Nonetheless, there are at least two reasons that make this improvement worthwhile.

	\begin{figure}[ht]
		\centering
		\subcaptionbox{\label{figSub:OutputCompareGenerateMDM}}[0.3\textwidth]{
			\centering
			\begin{tikzpicture}[scale=1.5]
				\coordinate (A) at (0,0);
				\coordinate (B) at (1,0);
				\coordinate (C) at (2,0);

				\coordinate (AB) at ($0.5*(A)+0.5*(B)$);
				\coordinate (BC) at ($0.5*(B)+0.5*(C)$);

				\draw[thick] (A) -- (B) -- (C);

				\node[red] at (A){$\bullet$};
				\node at (B){$\bullet$};
				\node at (C){$\bullet$};

				\draw[ultra thick, -latex] (B) -- (AB);
				\draw[ultra thick, -latex] (C) -- (BC);

				\node[below] at (A) {$0$};
				\node[below] at (B) {$0$};
				\node[below] at (C) {$0$};
			\end{tikzpicture}
		}
		\subcaptionbox{\label{figSub:OutputCompareMatchingOK}}[0.3\textwidth]{
			\centering
			\begin{tikzpicture}[scale=1.5]
				\coordinate (A) at (0,0);
				\coordinate (B) at (1,0.25);
				\coordinate (C) at (2,0.5);

				\coordinate (AB) at ($0.5*(A)+0.5*(B)$);
				\coordinate (BC) at ($0.5*(B)+0.5*(C)$);

				\draw[thick] (A) -- (B) -- (C);

				\node[red] at (A){$\bullet$};
				\node at (B){$\bullet$};
				\node at (C){$\bullet$};

				\draw[ultra thick, -latex] (B) -- (AB);
				\draw[ultra thick, -latex] (C) -- (BC);

				\node[below] at (A) {$0$};
				\node[below] at (B) {$\epsilon$};
				\node[below] at (C) {$2\epsilon$};
			\end{tikzpicture}
		}
		\subcaptionbox{\label{figSub:OutputCompareMatchingBad}}[0.3\textwidth]{
			\centering
			\begin{tikzpicture}[scale=1.5]
				\coordinate (A) at (0,0);
				\coordinate (B) at (1,0.5);
				\coordinate (C) at (2,0.25);

				\coordinate (AB) at ($0.5*(A)+0.5*(B)$);
				\coordinate (BC) at ($0.5*(B)+0.5*(C)$);

				\draw[thick] (A) -- (B);
				\draw[red,very thick] (B) -- (C);

				\node[red] at (A){$\bullet$};
				\node at (B){$\bullet$};
				\node[red] at (C){$\bullet$};

				\draw[ultra thick, -latex] (B) -- (AB);

				\node[below] at (A) {$0$};
				\node[above] at (B) {$2\epsilon$};
				\node[below] at (C) {$\epsilon$};
			\end{tikzpicture}
		}
		\caption{In \subref{figSub:OutputCompareGenerateMDM}, the input complex $K$ and map $f:K_0\rightarrow\R$ and the gradient field output by \GenerateMDM, assuming the leftmost vertex $v$ has minimal value $I(v)$. In \subref{figSub:OutputCompareMatchingOK} and \subref{figSub:OutputCompareMatchingBad}, two different outputs of \Matching \citep{Allili2019} and \ComputeDiscreteGradient \citep{Scaramuccia2020}, depending on the way $f$ is perturbed before given as input.}
		\label{fig:OutputCompare}
	\end{figure}

	First, perturbing $f$ before computing a compatible discrete gradient may induce some spurious critical points, as shown in Figure \ref{fig:OutputCompare}. For this particular example, we could verify that for both modified input maps in Figures \ref{figSub:OutputCompareMatchingOK} and \ref{figSub:OutputCompareMatchingBad}, the output does not depend on the chosen indexing map for both algorithms \Matching and \ComputeDiscreteGradient. However, there are a few different indexing maps admissible for \GenerateMDM, meaning that the output in Figure \ref{figSub:OutputCompareGenerateMDM} is not unique. Nonetheless, in this example, any indexing map $I$ given as input produces a gradient field with only one critical point, which is the vertex $v$ with minimal value $I(v)$.

	Second, being able to process a map $f$ that is not injective means that the algorithm may also be used to generate an arbitrary \mdm function and its gradient field on a given complex. Indeed, if we consider $f$ to be a constant map, such as the one in Figure \ref{figSub:OutputCompareGenerateMDM}, then there exists a unique level set $L_u$ in $K$, which is the whole complex. Thus, \GenerateMDM calls \ExpandMDM only once and it processes the whole complex $K = L_u$ at one time. Also, we see that in this specific case, $|f(K)|=1$ and $\lambda := \max_{u\in f(K)}|L_u| = |K|$. Hence, from Proposition \ref{prop:ComputationalCost} in Section \ref{sec:ComplexityAnalysis}, we have that \GenerateMDM takes at most $O(d\maxdim |K| + |K|\log|K|)$ time, or simply $O(|K|\log|K|)$ time if we assume $d$ and $\maxdim$ to be small.

	Experimentally, we see that a \mdm function generated this way is often perfect, meaning that its number of critical simplices of index $p$ equals the $p^\text{th}$ Betti number of $K$ (see Section \ref{sec:ExperimentsOptimality}). However, for large datasets, \GenerateMDM does not always minimize the number of critical simplices of the output, which is to be expected since computing an optimal discrete gradient field on general simplicial complexes is $\mathcal{NP}$-hard \citep{Joswig2006, Lewiner2003}. Still, the algorithms designed to find an optimal gradient field require either $\dim K\leq 2$ \citep{Lewiner2003} or $K$ to have a manifold structure \citep{Bauer2012}, while heuristics which aims to do this for general simplicial complexes run in $O(|K|^2)$ time \citep{Lewiner2003a}. This suggests that in addition to being well adapted to the context of multipersistent homology, our approach is efficient for generating an arbitrary discrete gradient field with an almost minimal number of critical simplices on a general simplicial complex.

	That said, previous algorithms that generate discrete gradient fields outperform \GenerateMDM in terms of speed. In particular, \ComputeDiscreteGradient computes a gradient field compatible with the max-extension of a component-wise injective vertex map in $O(|K| + |K_0|\log|K_0|)$ time, assuming $\dim K$ is low and the star of each vertex is small. However, under the same hypotheses and assuming $\maxdim$ is also small, \GenerateMDM takes $O(|K| + |f(K)|\log|f(K)|)$ time where $|K_0|\leq |f(K)| \leq |K|$. Hence, at best, it can be as efficient as the procedure \ComputeDiscreteGradient, but it is slower when the number $|f(K)|$ of level sets of $f$ is great. Nonetheless, as mentioned above, function \GenerateMDM proves itself to be advantageous on many other levels.

	\section{Geometric interpretation of the critical simplices}\label{sec:Pareto}

	In recent papers \citep{AssifPK2021, Budney2023, Cerri2019}, it is made clear that there is a strong connection between the homological changes in the bifiltration induced by a smooth $\R^2$-valued map and the Pareto set (as defined in \citep{Smale1975, Wan1975}) of that map. Moreover, we will see from experimental results (see Section \ref{sec:ExperimentalResults}) that for a given admissible function, \GenerateMDM produces critical simplices in clusters, which are very similar to Pareto sets of analogous smooth maps. This observation was also made for algorithms \Matching \citep{Allili2019} and \ComputeDiscreteGradient \citep{Scaramuccia2018}.

	In her PhD thesis, \citet{Scaramuccia2018} gives many insights into the different definitions of Pareto sets found in the literature and highlights the issues encountered when trying to define an appropriate variant in the discrete setting. Using concepts of combinatorial dynamics \citep{Lipinski2023, Mrozek2017}, we propose here a new approach to characterize the critical simplices for a multifiltering function and see how they relate to the concept of Pareto set for a vector-valued smooth map and to the critical simplices output by \GenerateMDM.

	\subsection{Defining critical simplices for a multifiltering function}

	Consider a combinatorial gradient field $\cV$ compatible with a multifiltering function $f:K\rightarrow\R^\maxdim$. As stated in Section \ref{sec:MorseIneqRelativePerfect}, we know that for any $u\in f(K)$, the number $m_p(u)$ of critical $p$-simplices of $\cV$ in $L_u$ is at least $\rank H_p(K(u), \bigcup_{u'\precneqq u}K(u'))$. This expression can be rewritten in terms of $L_u$. Indeed, notice that $\bigcup_{u'\precneqq u}K(u')$ is simply $K(u)\backslash L_u$. Then, using the excision theorem (see \citep{Edelsbrunner2010}) to excise the set $K(u)\backslash\Cl L_u$, we see that
	\begin{gather*}
		\Hom\left(K(u), \bigcup_{u'\precneqq u}K(u')\right) = \Hom\left(K(u), K(u)\backslash L_u\right) \cong \Hom\big(\Cl L_u, \Ex L_u\big),
	\end{gather*}
	which leads to the next proposition.

	\begin{prop}\label{prop:ParetoLevelSetContainCritSimplex}
		Let $\cV$ be an acyclic discrete vector field compatible with a multifiltering function $f:K\rightarrow\R^\maxdim$. For all $u\in f(K)$, if $\Hom\big(\Cl L_u, \Ex L_u\big)$ is nonzero, then $\cV$ has a critical simplex in $L_u$. Conversely, if $\cV$ is perfect relatively to $f$ and has a critical simplex in $L_u$, then $\Hom\big(\Cl L_u, \Ex L_u\big)$ is nonzero.
	\end{prop}

	This motivates the following definition.

	\begin{defn}
		For a multifiltering function $f:K\rightarrow\R^\maxdim$, we say that $u\in f(K)$ is a \emph{Pareto critical value} of $f$ if the relative homology $\Hom(\Cl L_u, \Ex L_u)$ is nonzero. Similarly, $\sigma\in K$ is \emph{Pareto critical} for $f$ if it belongs in a connected component $C$ of $L_{f(\sigma)}$ such that $\Hom(\Cl C, \Ex C)$ is nonzero. The Pareto critical simplices of $f$ constitute the \emph{Pareto set} of $f$, noted $\Pareto{f}$.
	\end{defn}

	\begin{rem} When $f$ is the max-extension of a component-wise injective vertex map, each level set $L_u$ is connected. Thus, when assuming this hypothesis, we have that $\sigma\in K$ is Pareto critical iff $f(\sigma)$ is a Pareto critical value. \end{rem}

	Since the gradient field of a \mdm function is a particular case of acyclic discrete vector field, we obtain the following result directly from Proposition \ref{prop:ParetoLevelSetContainCritSimplex}.

	\begin{cor}\label{coro:CriticalImpliesPareto}
		Let $g:K\rightarrow\R^\maxdim$ be a \mdm function compatible with a multifiltering function $f:K\rightarrow\R^\maxdim$. For every Pareto critical value $u$ of $f$, the \mdm function $g$ necessarily has a critical simplex in $L_u$. Conversely, when $g$ is perfect relatively to $f$, for every critical simplex $\sigma$ of $g$, we have that $\sigma\in\Pareto{f}$.
	\end{cor}

	\begin{rem}
		Notice that, when $g$ is not relative-perfect, we may not assume that all critical simplices of $g$ are Pareto for $f$. Indeed, we may find a counterexample by letting $f$ be such that $\Pareto{f} \neq K$ and $g$ be any \mdm function for which all simplices are critical. In that case, $g$ is trivially $f$-compatible, but there exists critical simplices of $g$ which are not Pareto for $f$. That being said, the relative-perfectness hypothesis is not necessary for every critical simplex of $g$ to be in $\Pareto{f}$. For instance, if $\Pareto{f} = K$, which is the case notably when $f=0$, then all critical simplices of any $f$-compatible \mdm function $g$ trivially belong in $\Pareto{f}$, whether or not $g$ is relative-perfect.
	\end{rem}

	In the original setting \citep{Smale1975, Wan1975}, a Pareto set is defined using concepts of differential topology, so our proposed definition may seem somewhat unrelated. Nonetheless, as mentioned earlier, in the smooth setting, it was shown that the Pareto set of a smooth vector-valued map is directly linked to the homological changes in the multifiltration induced by this map \citep{AssifPK2021, Budney2023, Cerri2019}. Moreover, it is rather common to use homology to characterize singularities in other discrete settings. Notably, it is done in PL Morse theory in order to define both critical points of real-valued functions \citep{Edelsbrunner2010, Knudson2015, Lewiner2013} and critical (Jacobi) sets of vector-valued mappings \citep{Edelsbrunner2008a}.

	Our definition is also inspired from the theory of combinatorial multivector fields, which provides discrete analogues to many concepts from smooth dynamics \citep{Lipinski2023}. Indeed, consider a multifiltering function $f:K\rightarrow\R^\maxdim$ and some $u\in f(K)$. For any two simplices $\sigma\leq\sigma'$, if $f(\sigma)=f(\sigma')=u$, then for all $\tau\in K$ such that $\sigma\leq\tau\leq\sigma'$, we have $f(\tau) = u$ because $f(\sigma)\preceq f(\tau)\preceq f(\sigma')$ by the definition of a multifiltering function. Hence, $\sigma,\sigma'\in L_u$ implies $\tau\in L_u$ whenever $\sigma\leq\tau\leq\sigma'$, meaning that $L_u$ is convex with respect to the face relation in $K$. In other words, each $L_u$ is a combinatorial multivector in the sense of \citep{Lipinski2023}. Therefore, the partition of $K$ in level sets is a well-defined combinatorial multivector field for which each $L_u$ is a multivector. Finally, in \citep{Lipinski2023}, a multivector $L_u$ is defined as critical  when $\Hom(\Cl L_u, \Ex L_u)$ is nonzero, so our proposed definition of a Pareto critical value agrees with the combinatorial multivector field theory.

	\begin{figure}[ht]
		\centering
		\subcaptionbox{\label{figSub:ParetoRegularNeighbourhood}}[0.3\textwidth]{
			\centering
			\includegraphics[width=\linewidth]{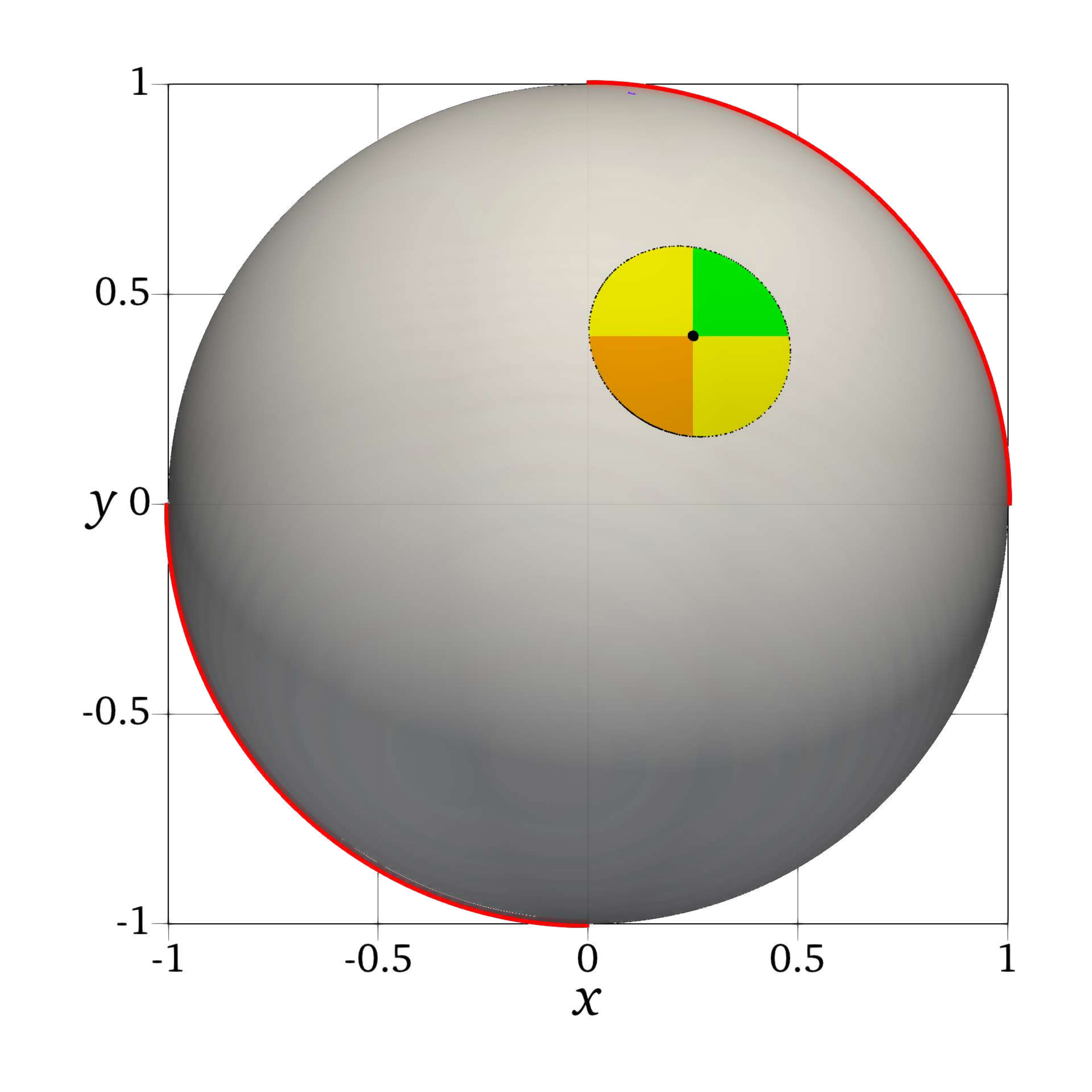}
		}
		\subcaptionbox{\label{figSub:ParetoMinNeighbourhood}}[0.3\textwidth]{
			\centering
			\includegraphics[width=\linewidth]{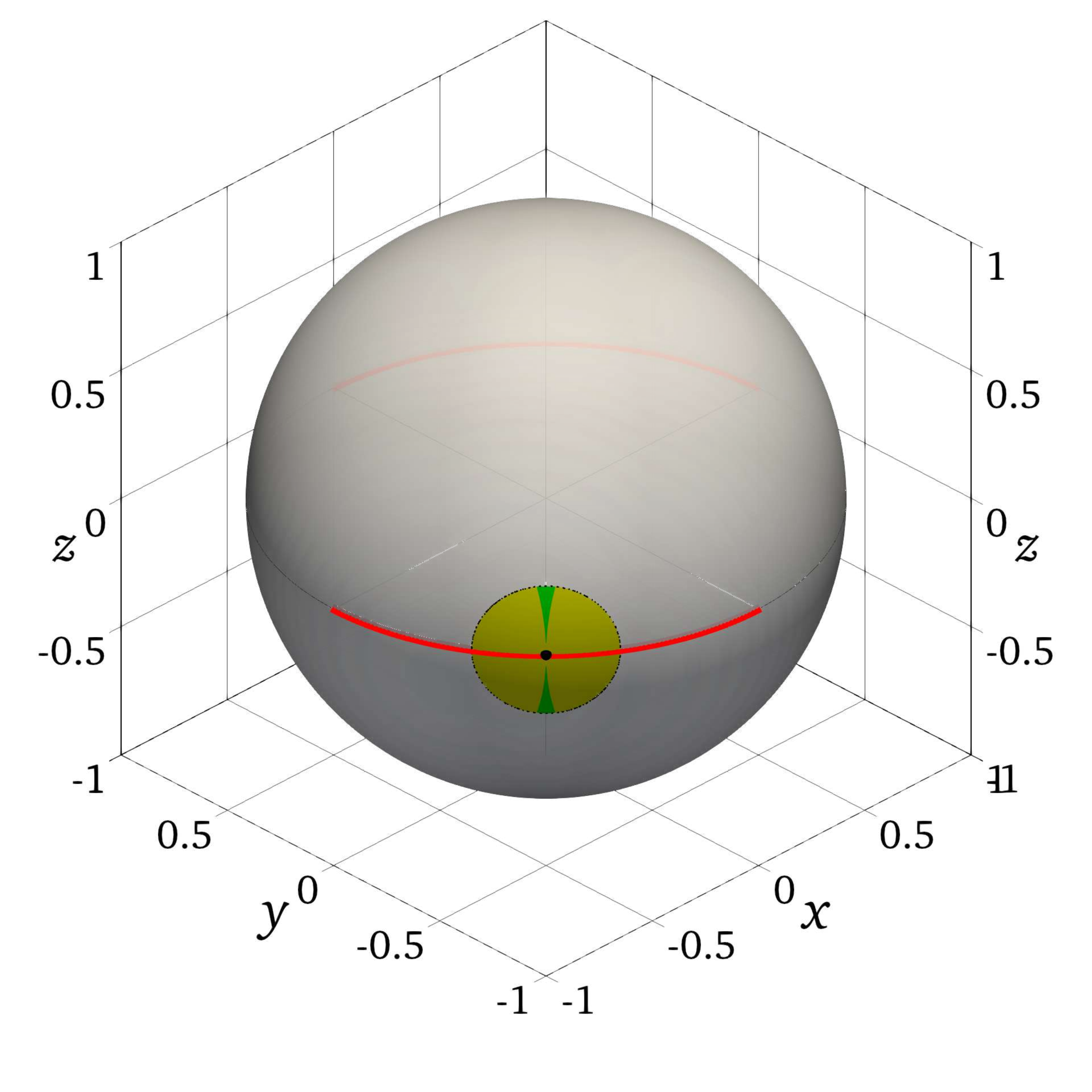}
		}
		\subcaptionbox{\label{figSub:ParetoMaxNeighbourhood}}[0.3\textwidth]{
			\centering
			\includegraphics[width=\linewidth]{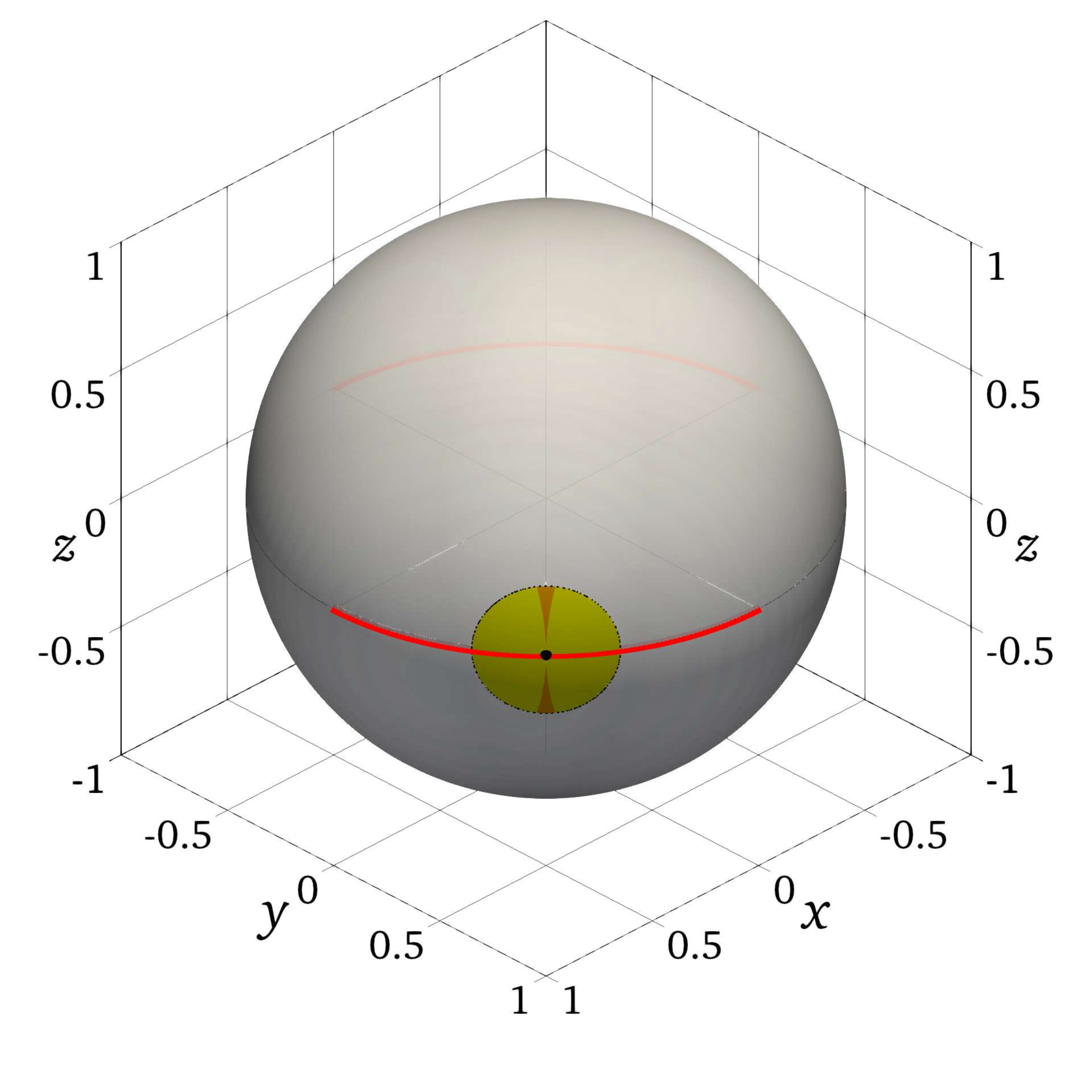}
		}
		\caption{In each figure is shown a neighbourhood $U$ of some point $p$ on the sphere $S^2$. For $f:S^2\rightarrow\R^2$ which maps each point to its coordinates $(x,y)\in\R^2$, the Pareto set of $f$ consists of the two red arcs, while points $q\in U$ such that $f(q)\preceq f(p)$ are orange, those for which $f(p)\preceq f(q)$ are green and all others are yellow. In \subref{figSub:ParetoRegularNeighbourhood}, the considered point $p$ is regular. In \subref{figSub:ParetoMinNeighbourhood}, $p$ is a Pareto minimum and in \subref{figSub:ParetoMaxNeighbourhood}, it is a Pareto maximum.}
		\label{fig:ParetoNeighbourhoods}
	\end{figure}

	To further justify the proposed definition of Pareto singularity, we present concisely the idea of Pareto points for smooth maps and give an interpretation which is common to both discrete and smooth settings. The reader is referred to \citep{Lee2012} for more details regarding the basic concepts of differential geometry. Consider a smooth map $f=(f_1,...,f_\maxdim):M\rightarrow\R^\maxdim$ on a manifold $M$ and let $p\in M$. We say $p$ is in the Pareto set $\theta$ of $f$ if there exists no direction in which the gradient vectors of all components $f_i$ at $p$ agree. In other words, away from $\theta$, all $f_i$ increase in some common direction. For all $p\notin \theta$ and for some small enough neighbourhood $U$ of $p$, we could show that this common direction may be represented by the subset $\{q\in U\ |\ f_i(q) > f_i(p)\text{ for each }i=1,...,\maxdim\}$, which is homeomorphic to the intersection of a convex cone with some small neighbourhood, so it is a contractible subset of $U$. An example of such a subset is represented in green in Figure \ref{figSub:ParetoRegularNeighbourhood}. Generally, this is not true for points $p\in \theta$. For example, in Figure \ref{figSub:ParetoMinNeighbourhood}, we see that the subset $\{q\in U\ |\ f_i(q) > f_i(p)\text{ for each }i=1,...,\maxdim\}$ has two distinct components while, in Figure \ref{figSub:ParetoMaxNeighbourhood}, it is empty.

	This brings us back to the idea behind the proposed definition of a Pareto critical point of a discrete multifiltering function $f:K\rightarrow\R^\maxdim$. Let $\sigma\in K$ such that $f(\sigma)=u$. As explained in \citep{Lipinski2023}, the level set $L_u$ may be seen as a black box inside of which we know nothing of the dynamics. Instead, to understand the behaviour of $f$ around $\sigma$, we have to look at the boundary of $L_u$. Notice that $\Ex L_u = \Cl L_u\backslash L_u\subseteq\Cl L_u\backslash \Int L_u = \Bd L_u$, so the exit set $\Ex L_u$ is part of the boundary of $L_u$, and from the definition of a multifiltering function, we could show that $\Ex L_u = \{\tau\in \Cl L_u\ |\ f(\tau)\precneqq f(\sigma)\}$. Hence, $\Ex L_u$ corresponds to the part of the domain neighbouring $\sigma$ for which the components of $f$ decrease. Thus, when the relative homology of $\Cl L_u$ with respect to $\Ex L_u$ is trivial, we interpret that all components of $f$ behave similarly and agree in some common direction around $\sigma$, represented by $\Ex L_u$.

	\subsection{Pareto primary simplices}

	We end this section by discussing the particular case where $f:K\rightarrow\R^\maxdim$ is the max-extension of a component-wise injective map defined on the set $K_0$ of vertices of $K$. Under these conditions, it was shown in \citep{Allili2019} that there exists a subset of simplices $S\subseteq K$, called \emph{primary} simplices, such the lower stars $\Low_f(\sigma) = \{\tau\geq\sigma\ |\ f(\tau) \preceq f(\sigma)\}$ of primary simplices $\sigma\in S$ partition $K$. Also, it was noted in \citep{Scaramuccia2020} that these lower stars correspond to the level sets $L_u$ of $f$. Indeed, since $f$ is component-wise injective on $K_0$ and $f_i(\tau) = \max_{v\in\tau} f_i(v)$ for all $\tau\in K$, then for every $u=(u_1,...,u_\maxdim)\in f(K)$ and each $i=1,...,\maxdim$, there exists a unique vertex $v_i\in K_0$ such that $f_i(v_i) = u_i$. We could prove that $v_1, v_2,...,v_\maxdim$ are exactly the vertices of $\sigma = \bigcap_{\tau\in L_u}\tau$ (thus $\dim\sigma\leq\maxdim-1$) and that $\sigma$ is the unique simplex in $K$ such that $\Low_f(\sigma) = L_u$.

	In other words, under the previous assumptions, we can identify each level set $L_u$ of $f$ to a unique primary simplex $\sigma\in L_u$. Although this statement is not always verified for an arbitrary multifiltering function, when they exist, primary simplices allow for a better understanding of the multifiltration induced by $f$.

	\begin{defn}
		For a multifiltering function $f:K\rightarrow\R^\maxdim$ and $u\in f(K)$, if there exists a simplex $\sigma\in L_u$ such that $\Low_f(\sigma) = L_u$, then we say $\sigma$ is the \emph{primary simplex of $L_u$}. If $u$ is Pareto critical, we say $\sigma$ is a \emph{primary Pareto critical simplex} for $f$.
	\end{defn}

	Moreover, as seen in Figure \ref{figSub:SphereIndexZPareto} in Section \ref{sec:ExperimentalResults}, the set of primary Pareto critical simplices of a multifiltering function $f$ is visually very similar to the Pareto set of its analogous smooth map when the domain $K$ of $f$ is a well-chosen triangulation. See Section \ref{sec:ExperimentsTriangulationIndexingMap} for more details on the effect that the chosen triangulation may have on the multifiltration of the space.

	\section{Experimental results}\label{sec:ExperimentalResults}

	We now test the algorithms presented in Section \ref{sec:Algorithms} on various datasets. More specifically, we are interested in the number of critical simplices \GenerateMDM produces, the different ways to agglomerate them in critical components and the impacts of the triangulation and indexing map on the output \mdm function.

	To implement our algorithms and the input simplicial complexes, we use the simplex tree structure of the GUDHI library \citep{GUDHI}. After generating a \mdm function, the simplicial complex on which it is defined and all relevant data is exported to a VTK file \citep{VTK} to be opened in a scientific visualization application. Here, we use ParaView \citep{ParaView} to generate our images.

	\subsection{Space complexity reduction}\label{sec:ExperimentsOptimality}

	As shown in \citep{Allili2017}, the multipersistent homology of a simplicial complex can be computed using only the critical simplices of a gradient field defined on it. Hence, \GenerateMDM could be used as a preprocessing tool in the computation of multipersistent homology and, to be as efficient as possible, it should produce as few critical simplices as possible.

	Thus, we compare the numbers of critical simplices output by \GenerateMDM with those of algorithms in \citep{Allili2017, Allili2019, Cerri2011}. Note that the algorithm presented in \citep{Cerri2011} does not compute a gradient field, but instead reduces the number of vertices and edges of the input dataset for while preserving the multipersistent homology of the space. Also, since \ComputeDiscreteGradient \citep{Scaramuccia2020} produces an output equivalent to that of \Matching, we do not compare \GenerateMDM to \ComputeDiscreteGradient.

	For each algorithm, we consider four different triangular meshes, available online in the GNU Triangulated Surface (GTS) Library \citep{GTS}. As input admissible map, we use the max-extension of $f:K_0\rightarrow \R^2$ such that $f(v) = (|x|, |y|)$ for each vertex $v$ of coordinates $(x,y,z)$. Also, the index mapping described at the beginning of Section \ref{sec:AlgoDescription} is used in \GenerateMDM.

	\begin{table}
		\caption{Proportion of critical simplices output by Algorithm \ref{algo:GenerateMDM} and algorithms in \citep{Allili2017, Allili2019, Cerri2011} given the vertex map $v\mapsto (|x|, |y|)$ and different datasets}\label{table:MatchingComparison}
		\begin{tabular}{l||ll|rrrr}
			\multirow{2}{*}{Dataset} & \multicolumn{2}{c|}{Simplices} & \multicolumn{4}{c}{Critical simplices (\%)} \\ \cline{2-7}
			& Type & Number & Alg. \ref{algo:GenerateMDM} & \citep{Allili2019} & \citep{Allili2017} & \citep{Cerri2011} \\ \hline\hline
			\multirow{4}{*}{\texttt{tie}} & Vertices & 2014 & 8.6 & 27.5 & 11.3 & 29.2 \\
			& Edges & 5944 & 7.6 & 20.1 & 56.2 & 13.9 \\
			& Triangles & 3827 & 4.1 & 14.1 & 78.7 & -\hspace*{1.5ex} \\
			& Total & 11785 & 6.7 & 19.4 & 55.9 & -\hspace*{1.5ex} \\ \hline
			\multirow{4}{*}{\texttt{space\_shuttle}} & Vertices & 2376 & 4.1 & 9.5 & 5.1 & 11.0 \\
			& Edges & 6330 & 1.6 & 3.8 & 58.4 & 5.2 \\
			& Triangles & 3952 & 0.1 & 0.4 & 90.5 & -\hspace*{1.5ex} \\
			& Total & 12658 & 1.6 & 3.8 & 58.4 & -\hspace*{1.5ex} \\ \hline
			\multirow{4}{*}{\texttt{x\_wing}} & Vertices & 3099 & 4.9 & 19.8 & 5.6 & 18.4 \\
			& Edges & 9190 & 3.5 & 13.4 & 39.2 & 9.2 \\
			& Triangles & 6076 & 2.5 & 9.9 & 56.2 & -\hspace*{1.5ex} \\
			& Total & 18365 & 3.4 & 13.3 & 39.2 & -\hspace*{1.5ex} \\ \hline
			\multirow{4}{*}{\texttt{space\_station}} & Vertices & 5749 & 29.4 & 30.8 & 32.7 & 33.7 \\
			& Edges & 15949 & 15.1 & 16.0 & 70.0 & 17.4 \\
			& Triangles & 10237 & 7.3 & 8.0 & 91.0 & -\hspace*{1.5ex} \\
			& Total & 31935 & 15.2 & 16.1 & 70.0 & -\hspace*{1.5ex} \\
		\end{tabular}
	\end{table}

	The results are presented in Table \ref{table:MatchingComparison} where, for each dataset, the number of simplices of each dimension is shown and the proportion of critical simplices output by each algorithm is presented by dimension. It is clear that, in terms of critical simplices produced, algorithm \GenerateMDM outperforms all its predecessors. This is partially explained by the fact that the algorithms are tested on datasets for which the considered vertex map $v\mapsto (|x|,|y|)$ is not component-wise injective. As explained in Section \ref{sec:PreviousAlgorithms}, this means other algorithms which require this hypothesis to be verified have to preprocess the input function by perturbing it slightly, which can potentially induce factitious critical simplices.

	%The results are presented in Table \ref{table:MatchingComparison}. In the first block of the table, the first line shows the total number of vertices and the following four lines the percentage of vertices which are critical in the output of each algorithm. Analogous information is given in the two following blocks for edges and triangles, while the last block gives the total number of simplices for each mesh and the associated proportion of output critical simplices.

	%It is clear from Table \ref{table:MatchingComparison} that, in terms of critical simplices produced, algorithm \GenerateMDM outperforms all its predecessors. This is partially explained by the fact that the algorithms are tested on datasets for which the considered vertex map $v\mapsto (|x|,|y|)$ is not component-wise injective. As explained in Section \ref{sec:PreviousAlgorithms}, this means other algorithms which require this hypothesis to be verified have to preprocess the input function by perturbing it slightly, which can potentially induce factitious critical simplices.

	Furthermore, for every considered input and for each level set $L_u$, we can compute $\rank H_p(\Cl L_u, \Ex L_u)$ and verify that it is exactly the number of critical $p$-simplices of the output in $L_u$, meaning that \GenerateMDM produced relative-perfect gradient fields in these cases. Although this is a strongly desired result, it is not surprising since each $L_u$ is quite small in these examples. Hence, in order to test if the algorithm outputs as few critical simplices as possible on larger and more complex level sets, we run \GenerateMDM on different datasets with the constant input function $f = 0$. This way, each dataset is partitioned into a unique level set $L_0$, so it is processed as a whole inside subfunction \ExpandMDM of \GenerateMDM.

	\begin{table}
		\caption{Critical simplices output by \GenerateMDM with $f=0$ and different datasets}\label{table:OptimalityCheck}
		\begin{tabular}{l||lll|lll|lll}
			Dataset & $|K_0|$ & $|\cC_0|$ & $\beta_0$ & $|K_1|$ & $|\cC_1|$ & $\beta_1$ & $|K_2|$ & $|\cC_2|$ & $\beta_2$ \\ \hline\hline
			\texttt{sphere} & $802$ & $1$ & $1$ & $2400$ & $0$ & $0$ & $1600$ & $1$ & $1$ \\
			\texttt{torus} & $800$ & $1$ & $1$ & $2400$ & $2$ & $2$ & $1600$ & $1$ & $1$ \\
			\texttt{klein\_bottle} & $800$ & $1$ & $1$ & $2400$ & $2$ & $2^*$ & $1600$ & $1$ & $1^*$ \\
			\texttt{projective\_plane} & $1081$ & $1$ & $1$ & $3240$ & $1$ & $1^*$ & $2160$ & $1$ & $1^*$ \\
			\texttt{dunce\_hat} & $1825$ & $1$ & $1$ & $5496$ & $1$ & $0$ & $3672$ & $1$ & $0$ \\
			\texttt{tie} & $2014$ & $18$ & $18$ & $5944$ & $149$ & $148$ & $3806$ & $7$ & $6$ \\
			\texttt{space\_shuttle} & $2376$ & $5$ & $5$ & $6330$ & $7$ & $7$ & $3952$ & $0$ & $0$ \\
			\texttt{x\_wing} & $3099$ & $18$ & $18$ & $9190$ & $53$ & $50$ & $6072$ & $16$ & $13$ \\
			\texttt{space\_station} & $5749$ & $110$ & $110$ & $15949$ & $116$ & $116$ & $10233$ & $39$ & $39$ \\
		\end{tabular}
		\begin{flushright}
			\begin{footnotesize}
				$^*$Betti numbers for homology with coefficients in $\Z_2$
			\end{footnotesize}
		\end{flushright}
	\end{table}

	The results are in Table \ref{table:OptimalityCheck}. Nine different datasets are considered: five triangulations of well-known topological spaces and the four datasets from Table \ref{table:MatchingComparison}. Again, we use the indexing map as described at the beginning of Section \ref{sec:AlgoDescription}.

	Knowing that a \mdm function defined on a dataset $K$ has at least $\beta_p(K)$ critical $p$-simplices (see Section \ref{sec:MorseIneqRelativePerfect}), we see from Table \ref{table:OptimalityCheck} that the algorithm generates an almost minimal number of critical simplices. Indeed, \GenerateMDM builds optimal functions on the first four triangulations. The number of critical simplices of the \mdm function generated on the triangulation of the dunce hat is not quite minimal, but it can be shown that there exists no perfect gradient field on the dunce hat \citep{Ayala2012}. Hence, the output is as optimal as it can be in this case as well.

	Furthermore, we see that \GenerateMDM outputs a \mdm function with a few critical simplices in excess when given as input larger and more complex datasets. Nonetheless, it is worth noting that the chosen indexing map has an impact on the output and, for the four datasets \texttt{tie}, \texttt{space\_shuttle}, \texttt{x\_wing} and \texttt{space\_station}, it is possible to find alternative indexing maps for which the output is optimal.

	\subsection{Dependence on the indexing map and triangulation}\label{sec:ExperimentsTriangulationIndexingMap}

	We see from Algorithms \ref{algo:GenerateMDM} and \ref{algo:ExpandMDM} that the indexing map $I$ is only used inside \ExpandMDM and its role is to serve as a tiebreaker when there are more than one simplex that could potentially be processed. We illustrate here how the choice of $I$ affects the output of \GenerateMDM.

	To do so, we consider the following particular constructions of $I$. Let $K$ be a simplicial complex embedded in $\R^3$ with vertices $K_0 = \{v_0,...,v_n\}$. Note $(x_i,y_i,z_i)$ the coordinates of each $v_i\in K_0$. We can choose to label the vertices so that $z_i < z_j\Rightarrow i < j$, meaning that the labels of the vertices increase as we move along the $z$-axis. Then, each $\sigma\in K$ may be represented by the labels of its vertices in decreasing order, and we can define $I:K\rightarrow\N$ as the index map obtained by ordering lexicographically $K$. We could show that $I$ is an admissible indexing map such that $\max_{v_i\in\tau}z_i < \max_{v_j\in\sigma}z_j \Rightarrow I(\tau) < I(\sigma)$, so the value of $I$ globally increases in the direction of the $z$-axis. We call \emph{$z$-increasing} such an admissible index map. Alternatively, we can choose to label the vertices so that $z_i < z_j\Rightarrow i > j$ in order to obtain a \emph{$z$-decreasing} indexing map such that $\max_{v_i\in\tau}z_i < \max_{v_j\in\sigma}z_j \Rightarrow I(\tau) > I(\sigma)$. Similarly, we can define indexing maps which are increasing or decreasing along the $x$ or $y$ axis.

	\begin{figure}[ht]
		\centering
		\subcaptionbox{\label{figSub:SphereIndexZPareto}}[0.45\textwidth]{
			\centering
			\includegraphics[width=\linewidth]{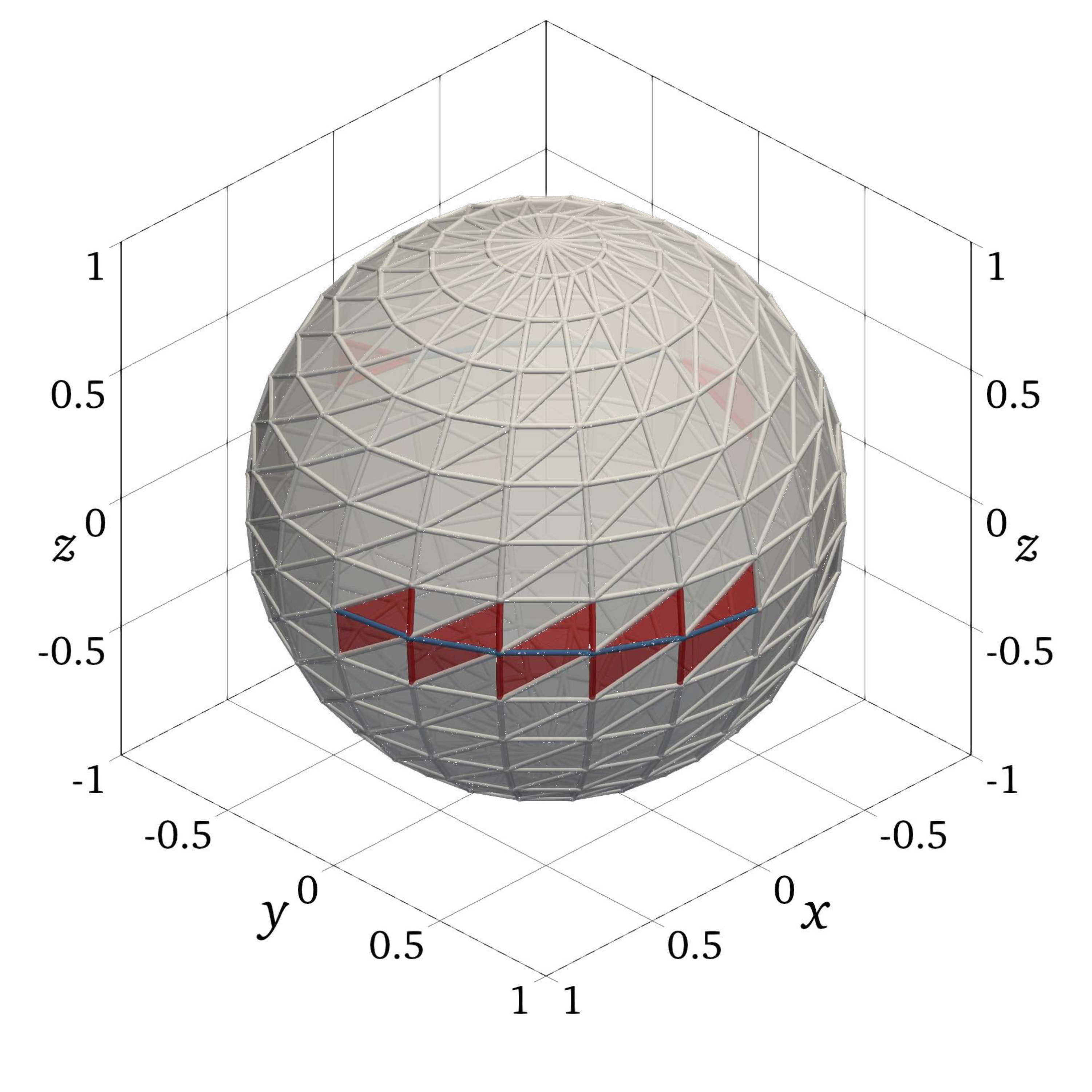}
		}

		\subcaptionbox{\label{figSub:SphereIndexZ+}}[0.45\textwidth]{
			\centering
			\includegraphics[width=\linewidth]{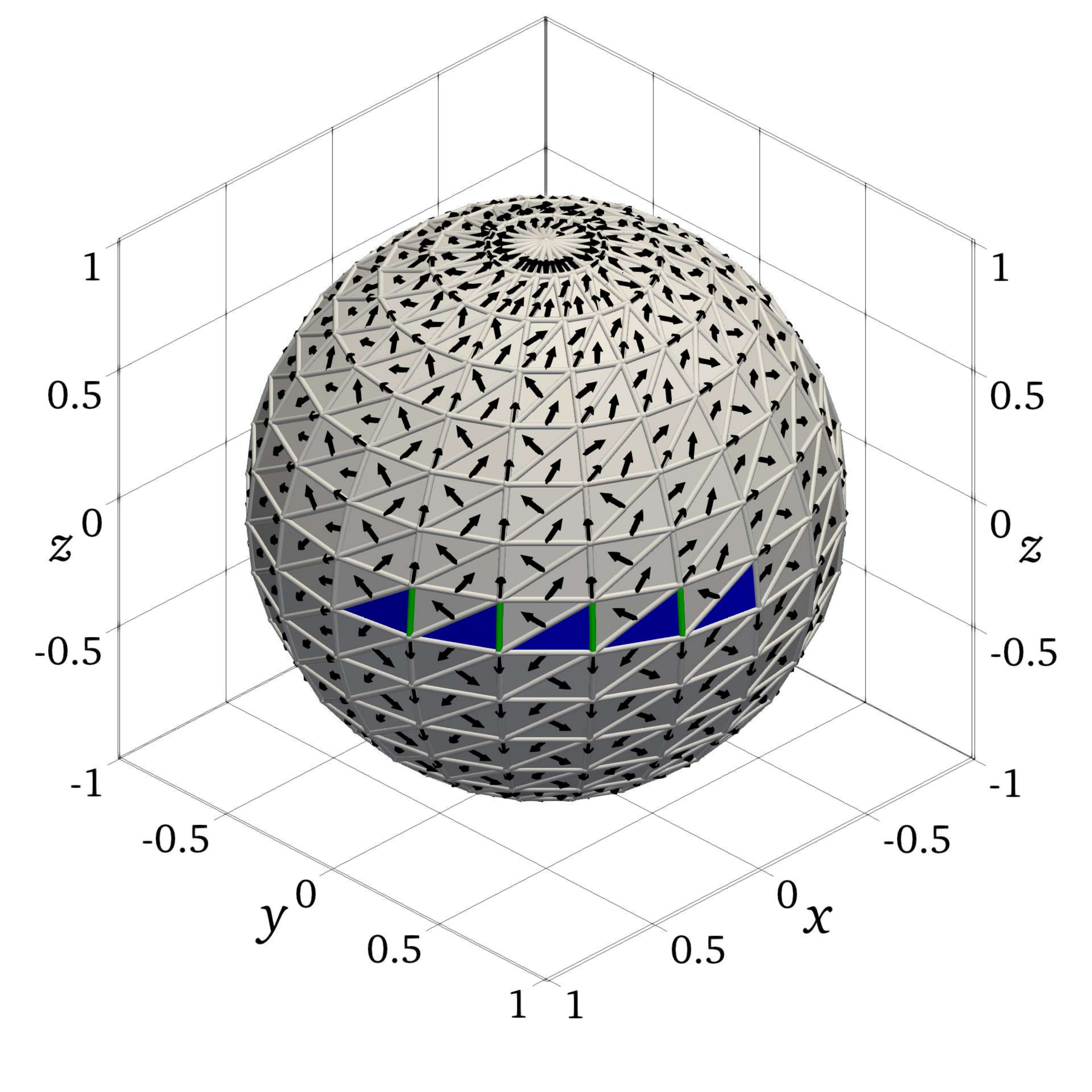}
		}
		\subcaptionbox{\label{figSub:SphereIndexZ-}}[0.45\textwidth]{
			\centering
			\includegraphics[width=\linewidth]{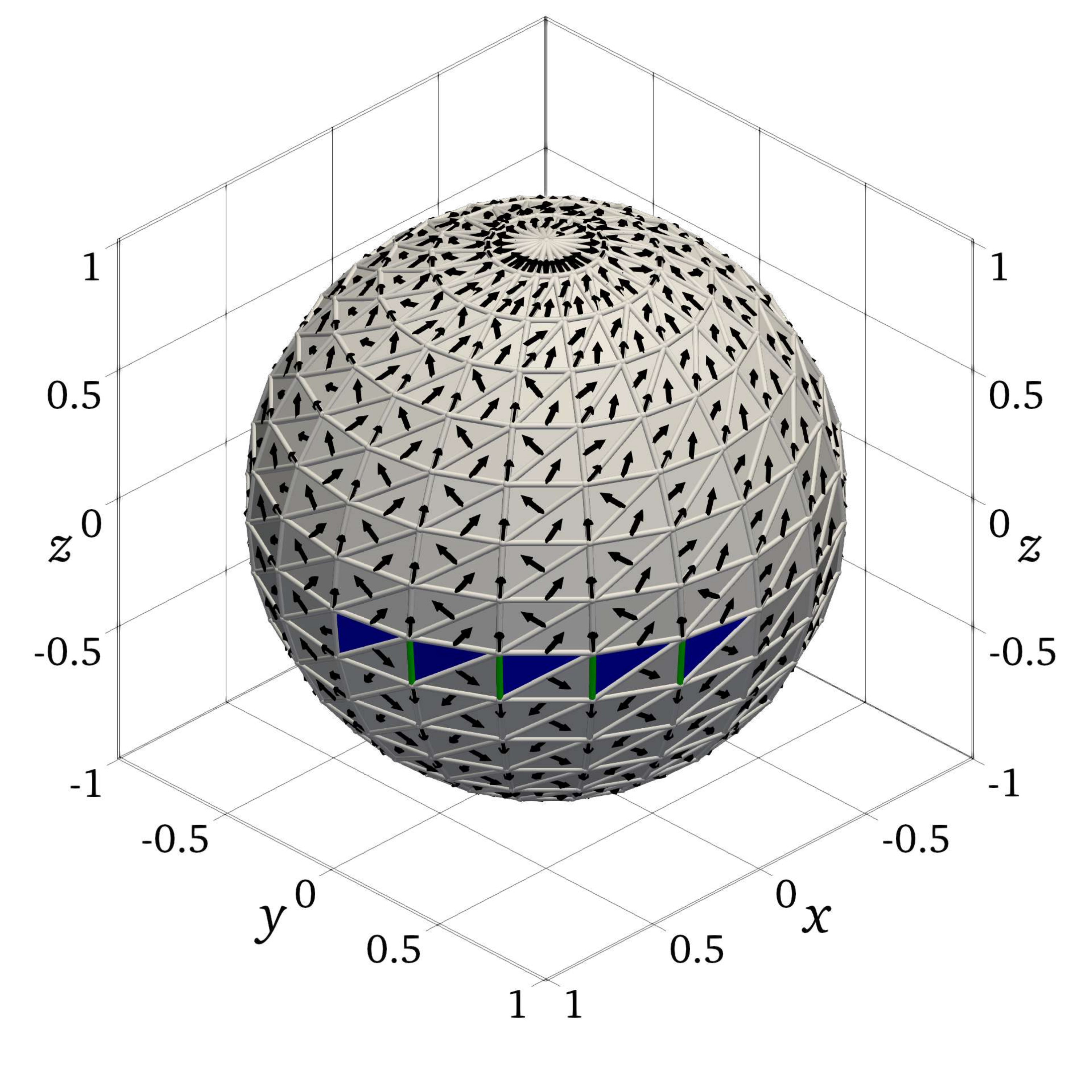}
		}
		\caption{In \subref{figSub:SphereIndexZPareto}, the Pareto set $\Pareto{f}$, where $f$ is the max-extension of the vertex map $v\mapsto(x,y)$. The primary Pareto simplices are in blue, those nonprimary are in red. In \subref{figSub:SphereIndexZ+}, the output of \GenerateMDM when given $f$ and a $z$-increasing index map as input. In \subref{figSub:SphereIndexZ-}, the output of \GenerateMDM when given $f$ and a $z$-decreasing index map as input.}
		\label{fig:SphereIndexZ}
	\end{figure}

	Now, let $f:K\rightarrow\R^2$ be defined as the max-extension of $v\mapsto(x,y)$ on a triangulated sphere $K$. The chosen triangulation $K$ and the Pareto set $\Pareto{f}$ are illustrated in Figure \ref{figSub:SphereIndexZPareto}. If we define $I$ as a $z$-increasing index map, then \GenerateMDM produces the gradient field in Figure \ref{figSub:SphereIndexZ+}. If, instead, we choose a $z$-decreasing indexing map $I$, then \GenerateMDM outputs the gradient field in Figure \ref{figSub:SphereIndexZ-}. Globally, for both outputs, we observe that the vectors point in directions where both components of $f$ decrease. However, for some level sets, we can see that \ExpandMDM matched simplices differently. For instance, the Pareto simplices which are output as critical are not the same, and we can notice that the matchings made along the Pareto curve form vectors that point in the direction in which $I$ decreases.

	\begin{figure}[ht]
		\centering
		\subcaptionbox{\label{figSub:SphereIndexX}}[0.45\textwidth]{
			\centering
			\includegraphics[width=\linewidth]{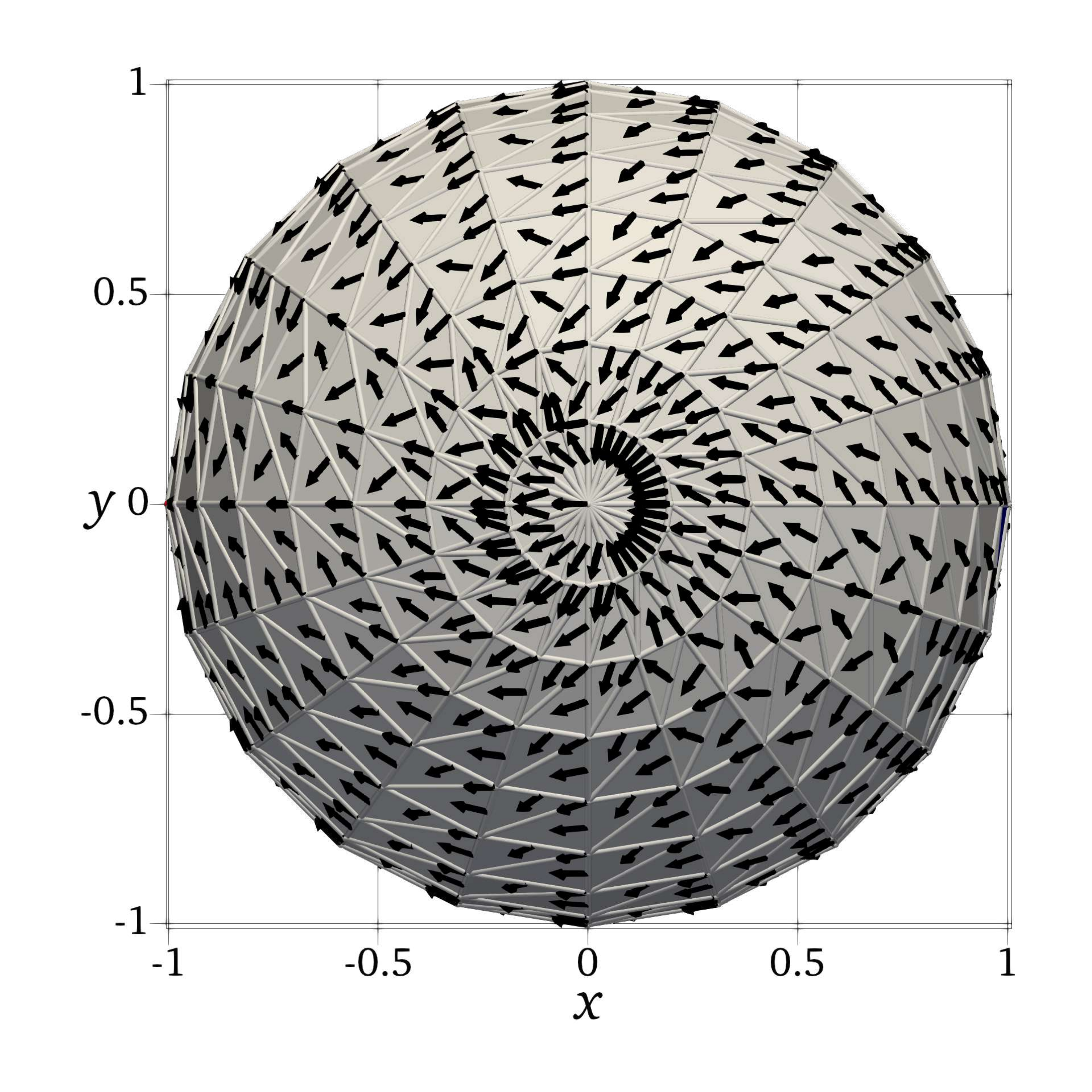}
		}
		\subcaptionbox{\label{figSub:SphereIndexY}}[0.45\textwidth]{
			\centering
			\includegraphics[width=\linewidth]{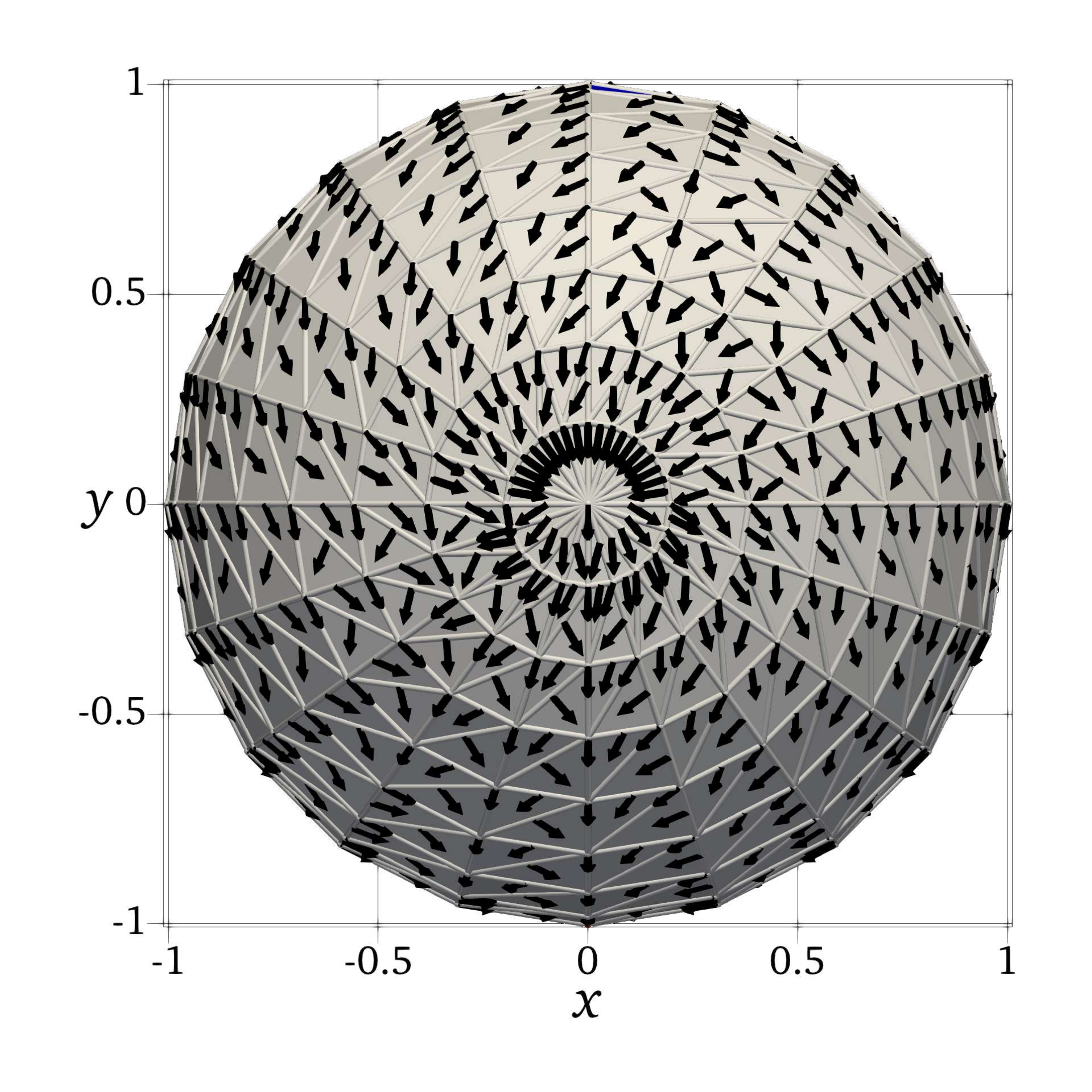}
		}
		\caption{Two outputs of \GenerateMDM when considering a triangulated sphere $K$ and $f=0$ as input along with two different indexing maps $I$. We have in \subref{figSub:SphereIndexX} the output for a $x$-increasing map $I$ and in \subref{figSub:SphereIndexY} for a $y$-increasing $I$.}
		\label{fig:SphereIndexXY}
	\end{figure}

	This last observation will be made clearer with the following example. Consider the previous triangulated sphere $K$, but let $f=0$. Then, \GenerateMDM processes $K$ as a single level set $K = L_0$ and the output depends more obviously on the indexing map $I$. In Figure \ref{figSub:SphereIndexX}, we have the output of \GenerateMDM when $I$ is $x$-increasing and, in Figure \ref{figSub:SphereIndexY}, the output when $I$ is $y$-increasing. For both outputs, we see that the gradient field generally points in the direction in which $I$ decreases.

	In the previous examples, we can verify that the outputs of \GenerateMDM are perfect relatively to the input function $f$. It is in part due to the size of level sets of $f$, which are quite small. This is not always the case when considering, for example, $f=0$ and larger datasets as inputs. Nonetheless, for all four datasets \texttt{tie}, \texttt{space\_shuttle}, \texttt{x\_wing} and \texttt{space\_station} from Tables \ref{table:MatchingComparison} and \ref{table:OptimalityCheck}, we can find indexing maps which produce perfect outputs and, in most cases, indexing maps that are increasing or decreasing along some axis produce very few extra critical simplices. In short, experiments show that the input indexing map has virtually no impact on the number of critical simplices output by \GenerateMDM. In particular, when the input function $f$ has many small level sets, \GenerateMDM always seems to produce outputs which are perfect relatively to $f$ independently of the chosen indexing map.

	\begin{figure}[ht]
		\centering
		\subcaptionbox{\label{figSub:SphereCubeGrid}}[0.45\textwidth]{
			\centering
			\includegraphics[width=\linewidth]{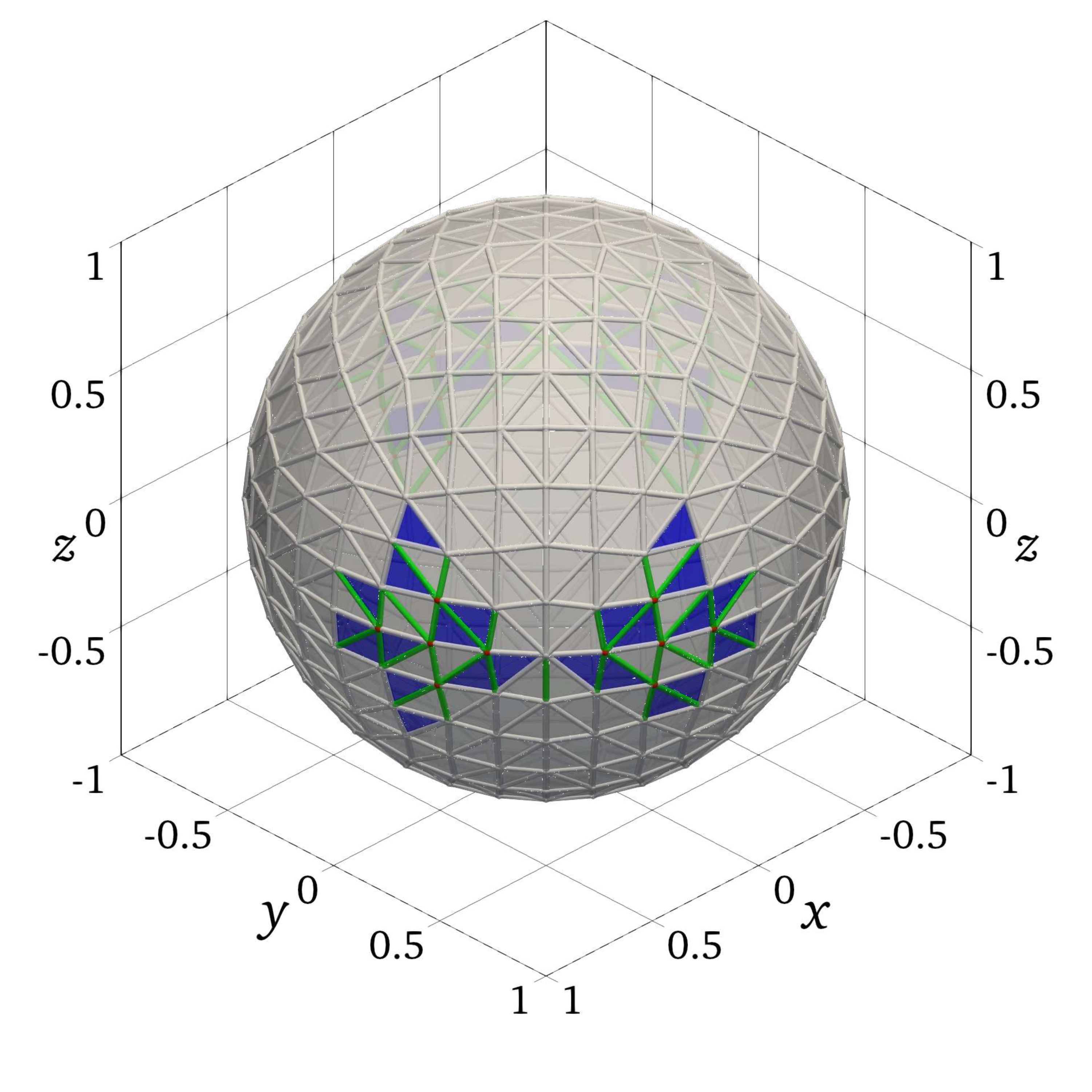}
		}
		\subcaptionbox{\label{figSub:SphereSoccerGrid}}[0.45\textwidth]{
			\centering
			\includegraphics[width=\linewidth]{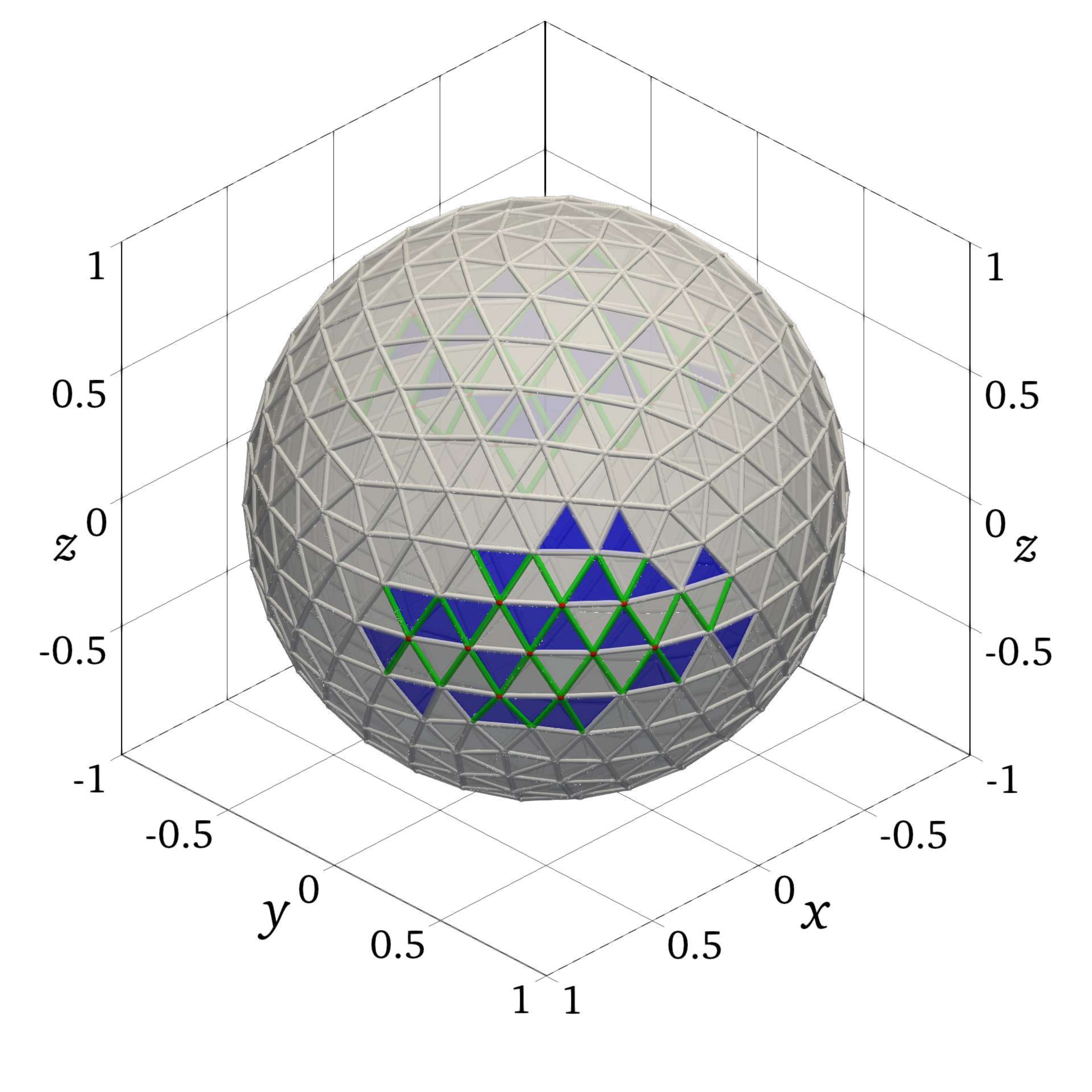}
		}
		\subcaptionbox{\label{figSub:SphereCubeGridPareto}}[0.45\textwidth]{
			\centering
			\includegraphics[width=\linewidth]{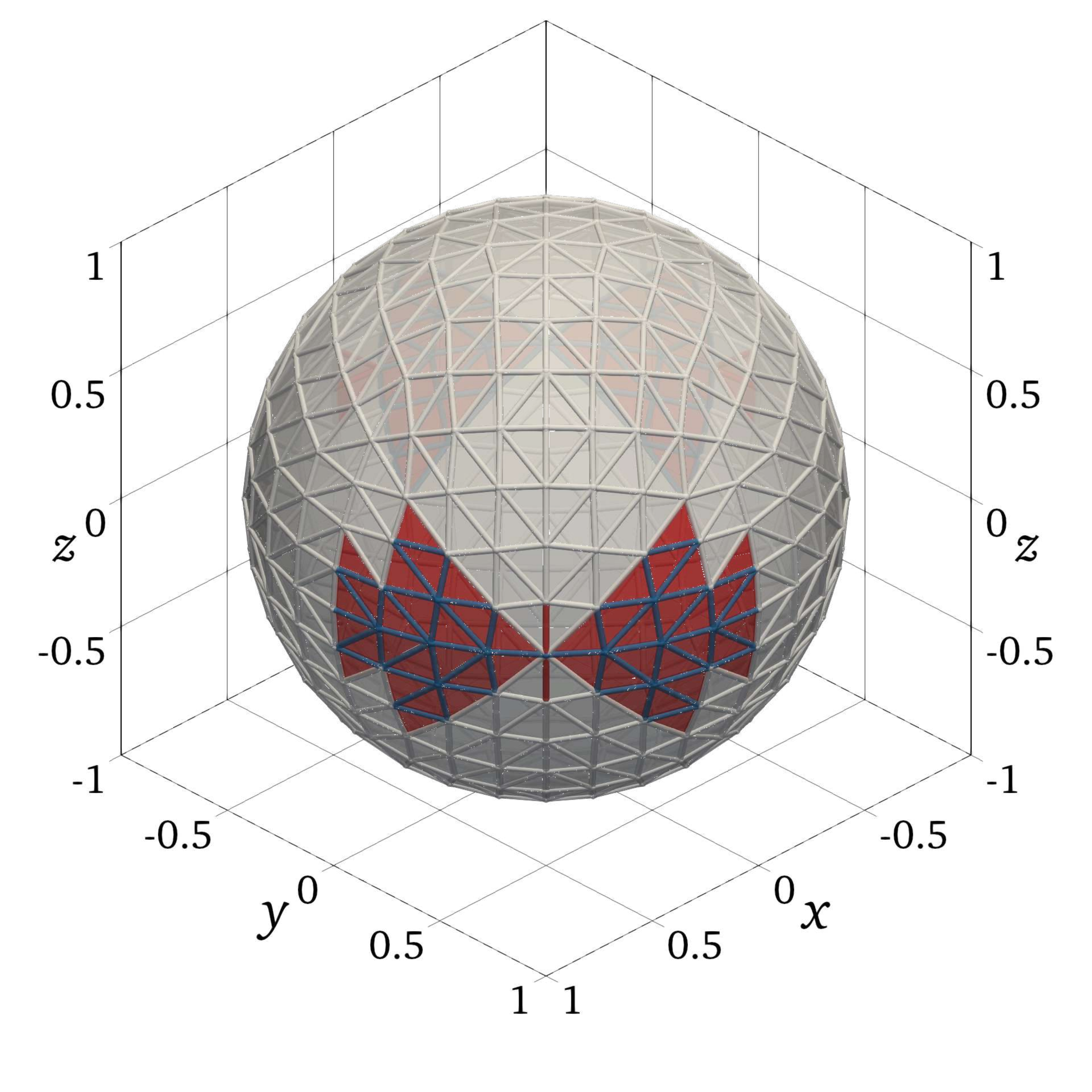}
		}
		\subcaptionbox{\label{figSub:SphereSoccerGridPareto}}[0.45\textwidth]{
			\centering
			\includegraphics[width=\linewidth]{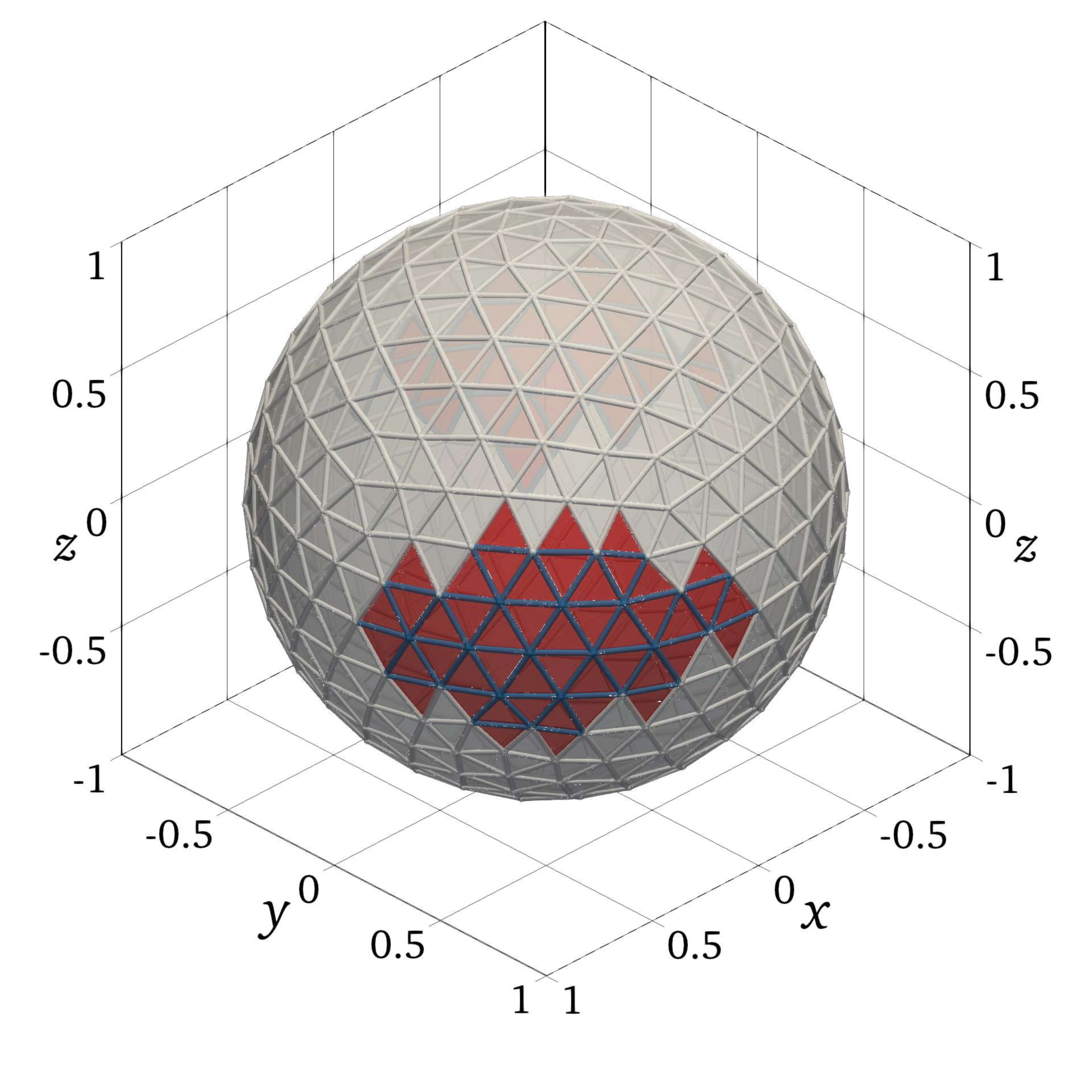}
		}
		\caption{In \subref{figSub:SphereCubeGrid}-\subref{figSub:SphereSoccerGrid}, critical simplices for different outputs of \GenerateMDM. In each subfigure, the max-extension of the vertex map $v\mapsto(x,y)$ is considered as input function, but the chosen triangulations of the sphere differ. In \subref{figSub:SphereCubeGridPareto}-\subref{figSub:SphereSoccerGridPareto}, the corresponding Pareto sets. The primary Pareto simplices are in blue, those nonprimary are in red.}
		\label{fig:SphereTriangulations}
	\end{figure}

	Furthermore, the way a given space is triangulated may influence significantly the critical simplices produced. To verify this, we use multiple meshes of the sphere generated with the Stripy library \citep{Stripy}. As shown in Figure \ref{fig:SphereTriangulations}, different triangulations yield visually dissimilar critical sets. That being said, we can see that this is due to the shape of the Pareto set of the input function on the given triangulation, and not to the algorithm itself.

	\subsection{Critical components}\label{sec:ExperimentsCritComp}

	We now see how Definition \ref{def:CriticalComponentSimG} of critical components, first introduced in \citep{Brouillette2022}, performs in practice. Recall that, for a \mdm function $g$, we partition its critical simplices using the equivalence relation $\Sim{g}$ defined as the transitive closure of $R_g$, such that $\sigma R_g\tau$ when
	\begin{enumerate}
		\item \label{item:DefR_g1Exp} $g_i(\sigma) = g_i(\tau)$ for some $i = 1,...,\maxdim$;
		\item \label{item:DefR_g2Exp} either $\sigma\rightconnects{g}\tau$ or $\sigma\leftconnects{g}\tau$.
	\end{enumerate}
	Condition \ref{item:DefR_g1Exp} ensures that $\sigma$ and $\tau$ can enter the multifiltration induced by $g$ at a same step, while condition \ref{item:DefR_g2Exp} implies that $\sigma$ and $\tau$ are connected and can interact with each other homologically.

	\begin{figure}[ht]
		\centering
		\subcaptionbox{\label{figSub:CritRelGworksMin}}[0.45\textwidth]{
			\centering
			\includegraphics[width=\linewidth]{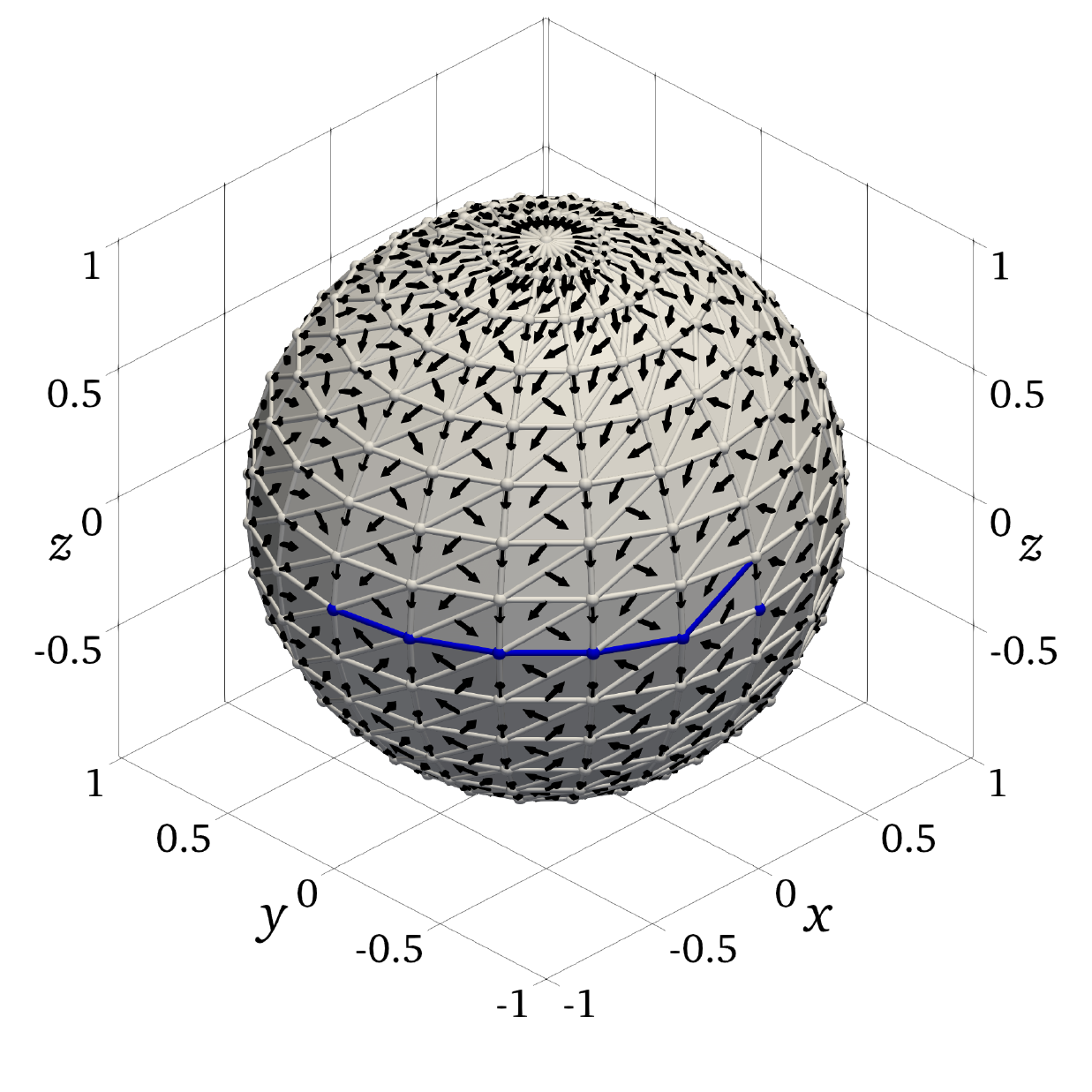}
		}
		\subcaptionbox{\label{figSub:CritRelGworksMax}}[0.45\textwidth]{
			\centering
			\includegraphics[width=\linewidth]{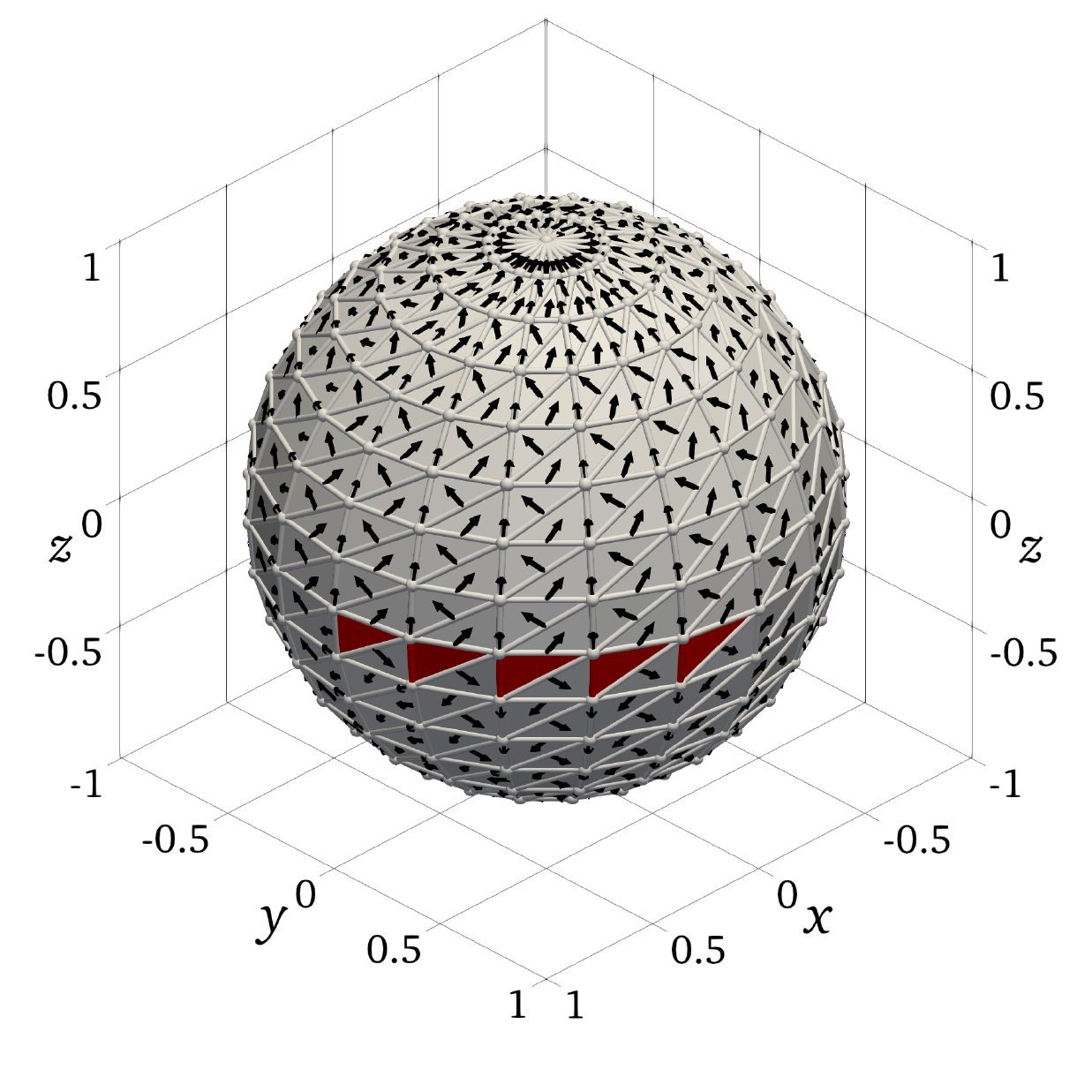}
		}
		\caption{Gradient field and critical components of $g$ with respect to $\Sim{g}$, where $g$ is the \mdm function output by \GenerateMDM when given as input the max-extension of the vertex map $v\mapsto (x,y)$.}
		\label{fig:CritRelGworks}
	\end{figure}

	The previous examples and Figure \ref{fig:ParetoNeighbourhoods} make it clear that, if we consider $K$ to be a triangulated sphere and $g:K\rightarrow\R^2$ to be a \mdm function approximating the max-extension of the vertex map $v\mapsto (x,y)$, we can expect the critical simplices of $g$ to form two distinct critical components along the equator of $K$. As illustrated in Figure \ref{fig:CritRelGworks}, $\Sim{g}$ yields the desired result for the given triangulation. Namely, we could characterize the blue component in Figure \ref{figSub:CritRelGworksMin} as a Pareto minimal component and the red one in Figure \ref{figSub:CritRelGworksMax} as Pareto maximal.

	It is worth noting that the critical component in Figure \ref{figSub:CritRelGworksMin} is not connected in the topological sense since one of its endpoint, the vertex with coordinates $(0,-1, 0)$, is isolated. Nonetheless, the gradient field of $g$ connects the critical arc to the isolated critical vertex, which explains why $\Sim{g}$ gives the expected result.

	\begin{figure}[ht]
		\centering
		\subcaptionbox{\label{figSub:CritRelGworksNot}}[0.45\textwidth]{
			\centering
			\includegraphics[width=\linewidth]{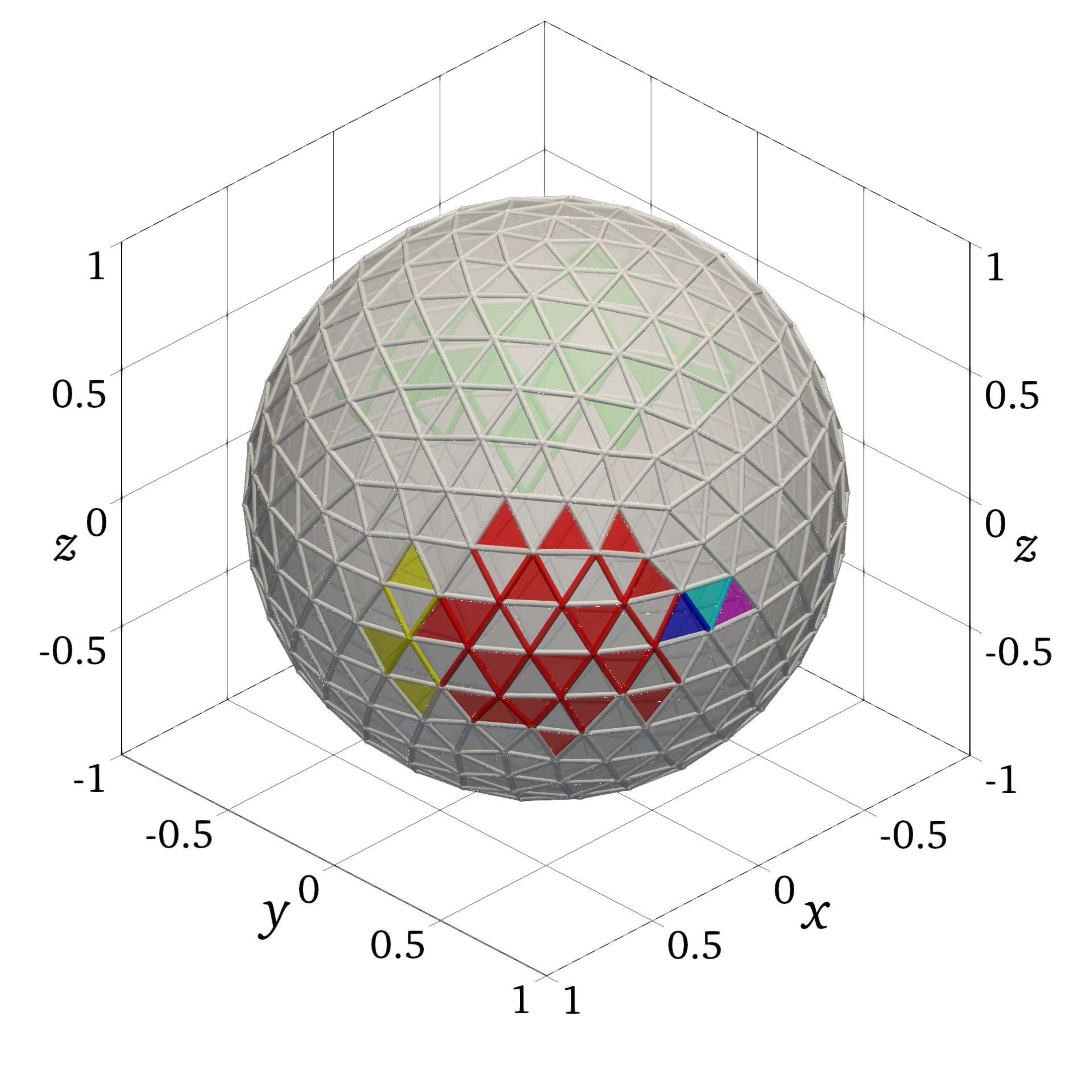}
		}
		\subcaptionbox{\label{figSub:CritRelG'works}}[0.45\textwidth]{
			\centering
			\includegraphics[width=\linewidth]{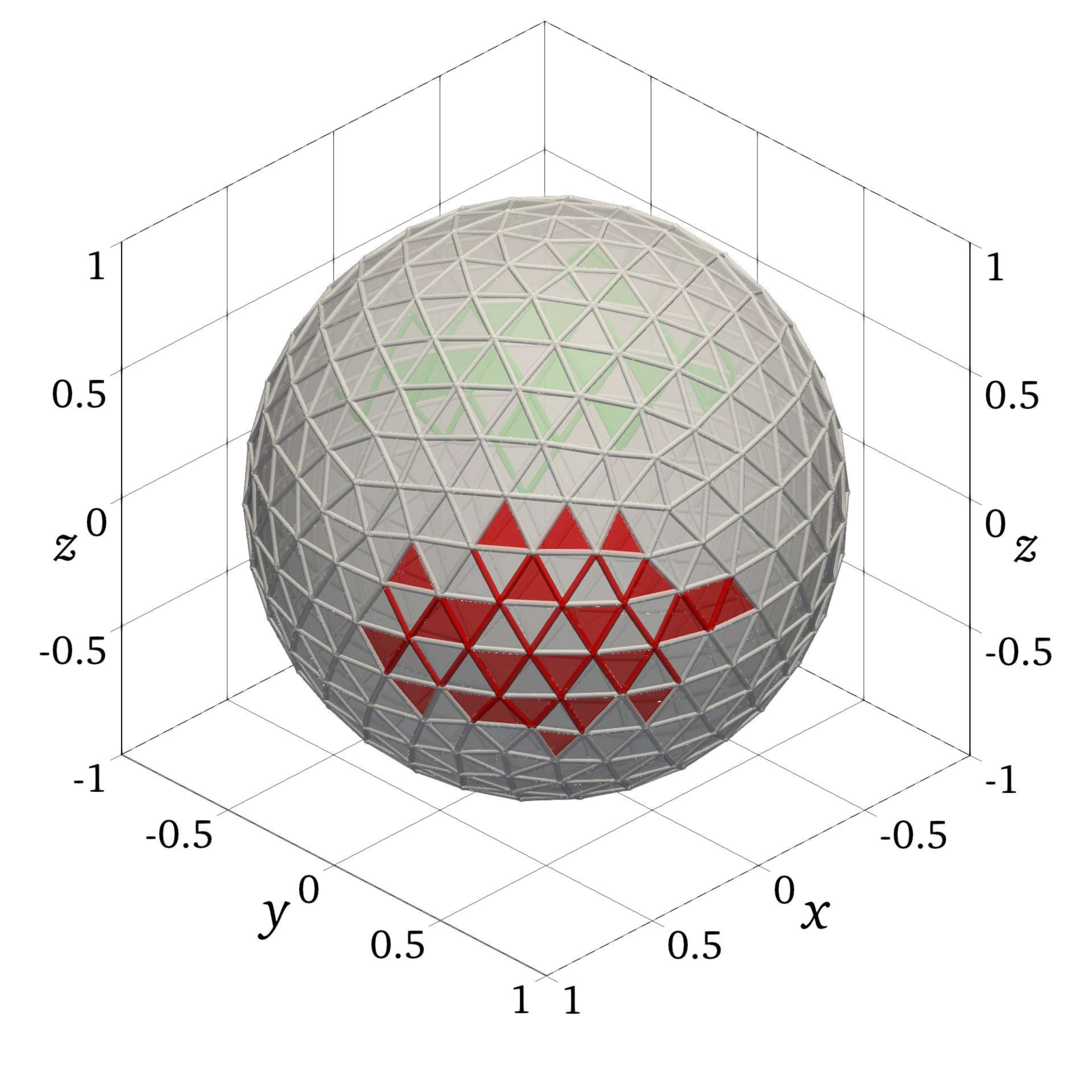}
		}
		\caption{Critical components of $g$, where $g$ is the \mdm function output by \GenerateMDM when given as input the max-extension of the vertex map $v\mapsto (x,y)$. In \subref{figSub:CritRelGworksNot}, the critical components are those with respect to $\Sim{g}$. In \subref{figSub:CritRelG'works}, we consider $\Sim{g}'$.}
		\label{fig:CritRelGworksNot}
	\end{figure}

	However, for some triangulations, the critical components induced by $\Sim{g}$ are not as desired. In Figure \ref{figSub:CritRelGworksNot}, for the given triangulation (which is the same as that in Figures \ref{figSub:SphereSoccerGrid} and \ref{figSub:SphereSoccerGridPareto}), we see that some simplices which should belong in a same critical component are separated in a few different components.

	\begin{figure}[ht]
		\centering
		\subcaptionbox{\label{figSub:AlgorithmExampleCompG}}[0.45\textwidth]{
			\centering
			\begin{tikzpicture}[scale=3]
				\coordinate (A) at (0, 0);
				\coordinate (B) at (1, 0);
				\coordinate (C) at (0, 1);
				\coordinate (D) at (1, 1);

				\fill[orange!25] (A) -- (B) -- (D) -- cycle;
				\fill[black!10] (A) -- (C) -- (D) -- cycle;

				\draw[ultra thick, orange] (A) -- (B);
				\draw[thick] (A) -- (C);
				\draw[thick] (A) -- (D);
				\draw[thick] (B) -- (D);
				\draw[ultra thick, blue] (C) -- (D);

				\draw[ultra thick, -stealth] (A) -- ($0.5*(A)+0.5*(C)$);
				\draw[ultra thick, -stealth] (D) -- ($0.5*(B)+0.5*(D)$);
				\draw[ultra thick, -stealth] ($0.5*(A)+0.5*(D)$) -- ($0.33*(A)+0.33*(C)+0.33*(D)$);

				\node at (A){$\bullet$};
				\node[magenta] at (B){$\bullet$};
				\node[green] at (C){$\bullet$};
				\node at (D){$\bullet$};

				\node[below left] at (A) {$(1,2)$};
				\node[below right] at (B) {$(0,0)$};
				\node[above left] at (C) {$(0,0)$};
				\node[above right] at (D) {$(2,1)$};

				\node[below] at ($0.5*(A)+0.5*(B)$) {$(1+\delta,2)$};
				\node[left] at ($0.5*(A)+0.5*(C)$) {$(1,2)$};
				\node[xshift=0.2in, yshift=-0.125in] at ($0.5*(A)+0.5*(D)$) {$(2+\delta,2)$};
				\node[right] at ($0.5*(B)+0.5*(D)$) {$(2,1)$};
				\node[above] at ($0.5*(C)+0.5*(D)$) {$(2+\delta,1)$};

				\node[below] at ($0.35*(A)+0.35*(B)+0.3*(D)$) {$(2+2\delta,2)$};
				\node at ($0.3*(A)+0.35*(C)+0.35*(D)$) {$(2+\delta,2)$};
			\end{tikzpicture}
		}
		\subcaptionbox{\label{figSub:AlgorithmExampleCompGprime}}[0.45\textwidth]{
			\centering
			\begin{tikzpicture}[scale=3]
				\coordinate (A) at (0, 0);
				\coordinate (B) at (1, 0);
				\coordinate (C) at (0, 1);
				\coordinate (D) at (1, 1);

				\fill[orange!25] (A) -- (B) -- (D) -- cycle;
				\fill[black!10] (A) -- (C) -- (D) -- cycle;

				\draw[ultra thick, orange] (A) -- (B);
				\draw[thick] (A) -- (C);
				\draw[thick] (A) -- (D);
				\draw[thick] (B) -- (D);
				\draw[ultra thick, orange] (C) -- (D);

				\draw[ultra thick, -stealth] (A) -- ($0.5*(A)+0.5*(C)$);
				\draw[ultra thick, -stealth] (D) -- ($0.5*(B)+0.5*(D)$);
				\draw[ultra thick, -stealth] ($0.5*(A)+0.5*(D)$) -- ($0.33*(A)+0.33*(C)+0.33*(D)$);

				\node at (A){$\bullet$};
				\node[magenta] at (B){$\bullet$};
				\node[green] at (C){$\bullet$};
				\node at (D){$\bullet$};

				\node[below left] at (A) {$(1,2)$};
				\node[below right] at (B) {$(0,0)$};
				\node[above left] at (C) {$(0,0)$};
				\node[above right] at (D) {$(2,1)$};

				\node[below] at ($0.5*(A)+0.5*(B)$) {$(1,2)$};
				\node[left] at ($0.5*(A)+0.5*(C)$) {$(1,2)$};
				\node[xshift=0.125in, yshift=-0.125in] at ($0.5*(A)+0.5*(D)$) {$(2,2)$};
				\node[right] at ($0.5*(B)+0.5*(D)$) {$(2,1)$};
				\node[above] at ($0.5*(C)+0.5*(D)$) {$(2,1)$};

				\node[below] at ($0.35*(A)+0.35*(B)+0.3*(D)$) {$(2,2)$};
				\node at ($0.3*(A)+0.35*(C)+0.35*(D)$) {$(2,2)$};
			\end{tikzpicture}
		}
		\caption{In \subref{figSub:AlgorithmExampleCompG}, a \mdm function $g$ and its gradient field output by \GenerateMDM, as shown in Figure \ref{fig:AlgorithmExample}. The critical components of $g$ with respect to $\Sim{g}$ are represented in green, pink, blue and orange. In \subref{figSub:AlgorithmExampleCompGprime}, the admissible map $f$ used to generate $g$. The critical components of $g$ with respect to $\Sim{g}'$ are represented in green, pink and orange.} \label{fig:AlgorithmExampleComponents}
	\end{figure}

	This is due to the fact that, as stated in Proposition \ref{prop:ComputeGworks}, we have that $g_1$ is not always exactly equal to $f_1$. Hence, although the difference between the two functions can be arbitrarily small, we do not necessarily have $g_1(\sigma)=g_1(\tau)$ when $f_1(\sigma)=f_1(\tau)$, so the tiny discrepancies induced by Algorithm \ref{algo:ComputeG} sometimes suffice to induce disparities between the desired critical components and the actual partition obtained with respect to $\Sim{g}$. This is illustrated in Figure \ref{fig:AlgorithmExampleComponents}.

	\pagebreak

	To resolve this issue, when $g$ is a \mdm function output by \GenerateMDM with a multifiltering function $f$ as input, we can substitute condition \ref{item:DefR_g1Exp} of the definition of $R_g$ given previously by
	\begin{itemize}
		\item[$(1')$] $f_i(\sigma) = f_i(\tau)$ for some $i = 1,...,\maxdim$.
	\end{itemize}
	We note $R_g'$ the updated relation on the critical simplices of $g$ and $\Sim{g}':= \bar{R}_g'$ the associated equivalence relation. Then, the critical components in Figure \ref{figSub:CritRelGworksNot} become those in Figure \ref{figSub:CritRelG'works}, so we obtain the desired result.

	\begin{figure}[ht]
		\centering
		\subcaptionbox{\label{figSub:CritRelG'worksNotPareto}}[0.56\textwidth]{
			\centering
			\includegraphics[width=\linewidth]{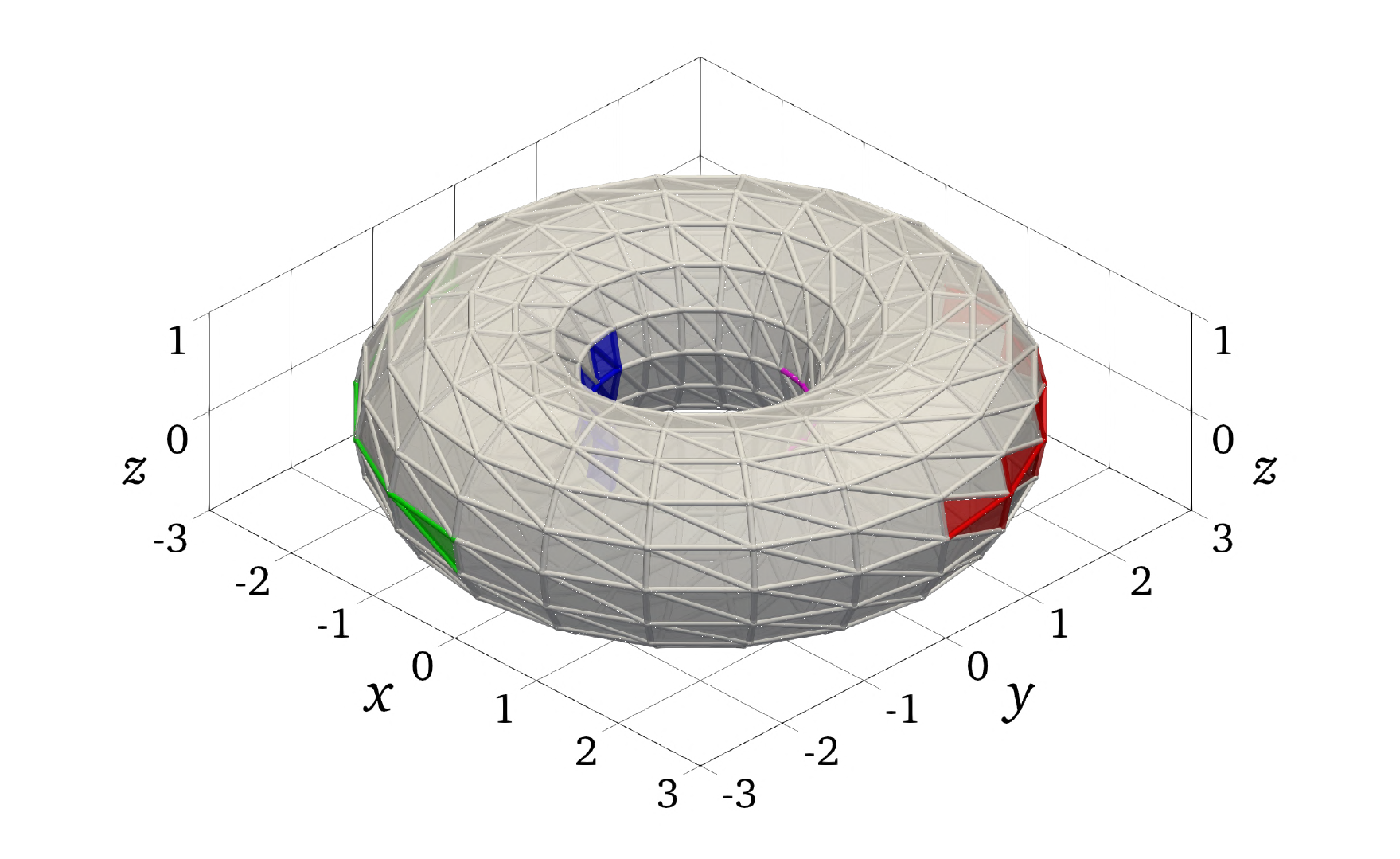}
		}
		\subcaptionbox{\label{figSub:CritRelG'worksNotParetoAbove}}[0.34\textwidth]{
			\centering
			\includegraphics[width=\linewidth]{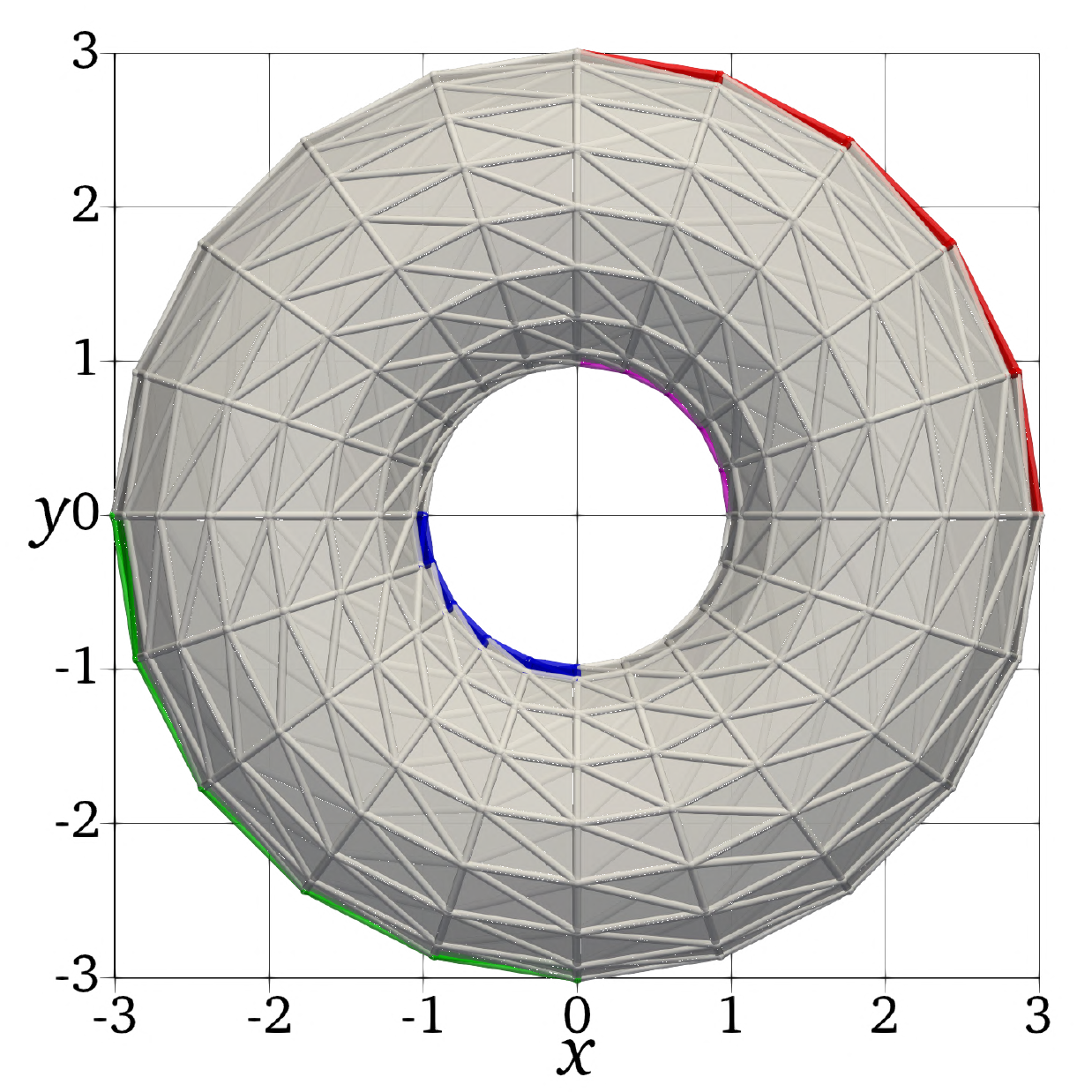}
		}
		\caption{In green, blue, purple and red, the four connected components of the Pareto set of the max-extension of the vertex map $v\mapsto(x,y)$ on a triangulated torus.}
		\label{fig:CritRelG'worksNotPareto}
	\end{figure}

	\begin{figure}[ht]
		\centering
		\subcaptionbox{\label{figSub:CritRelG'worksNot}}[0.56\textwidth]{
			\centering
			\includegraphics[width=\linewidth]{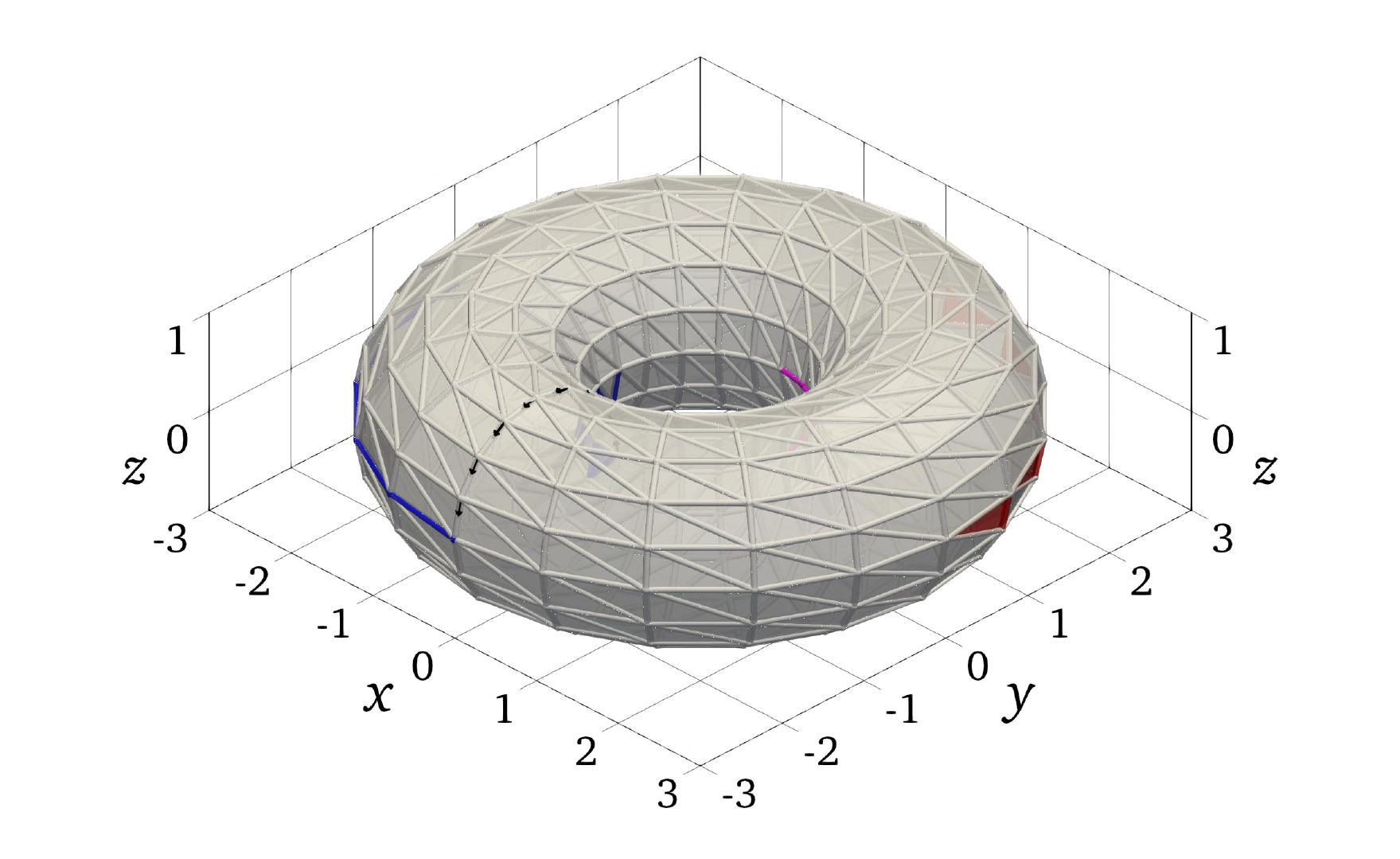}
		}
		\subcaptionbox{\label{figSub:CritRelG'worksNotCut}}[0.34\textwidth]{
			\centering
			\includegraphics[width=\linewidth]{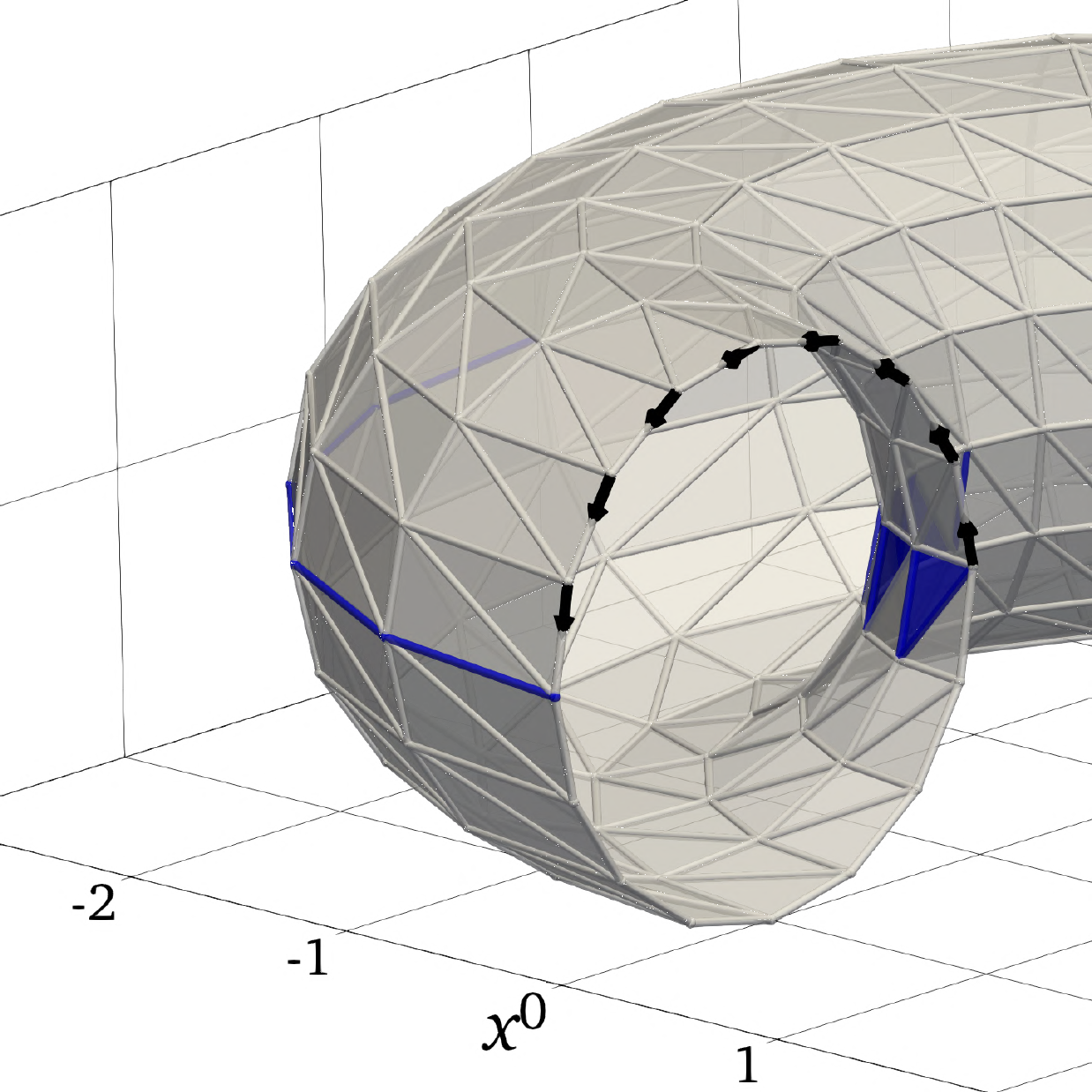}
		}
		\caption{In \subref{figSub:CritRelG'worksNot}, critical components of $g$ with respect to $\Sim{g}'$, where $g$ is the \mdm function output by \GenerateMDM when given the max-extension $f$ of the vertex map $v\mapsto(x,y)$ as input. There are three components, represented in blue, purple and red. The arrows show a gradient path connecting the critical simplices of the blue component along the $x=0$ plane. In \subref{figSub:CritRelG'worksNotCut}, a view of the torus zoomed in on the blue critical component.}
		\label{fig:CritRelG'worksNot}
	\end{figure}

	Still, if $f$ is the max-extension of a vertex map that is not component-wise injective, it becomes possible for critical simplices $\sigma$ and $\tau$ that are expected to belong to different critical components to be such that $\sigma R_g' \tau$. Indeed, consider $f$ to be the max-extension of the vertex map $v\mapsto(x,y)$ defined on the torus in Figure \ref{fig:CritRelG'worksNotPareto}. From the connected components of $\Pareto{f}$, we can deduce that a \mdm function generated using $f$ should have four critical components corresponding to a Pareto minimal component (in green), two Pareto "saddle" components (in blue and purple) and a Pareto maximal component (in red). However, Figure \ref{fig:CritRelG'worksNot} shows that the \mdm function $g$ output by \GenerateMDM only has three critical components. This is due to the vertices with coordinates $(0,-3,0)$ and $(0,-1,0)$, which have equal $x$ coordinates and are linked by a gradient path of $g$.

	Note that this could not happen if the input map $f$ was the max-extension of a component-wise injective vertex map. Nevertheless, to ensure the concept of critical components yields desirable results in our more general setting, we also substitute condition \ref{item:DefR_g2Exp} of the definition of $R_g$ in order to obtain the following definition.

	\begin{defn}\label{def:CritComponentsF}
		Assume $g:K\rightarrow\R^\maxdim$ is a \mdm function compatible with a multifiltering function $f:K\rightarrow\R^\maxdim$. For $\sigma\in K$, we note $C_\sigma$ the connected component of $\sigma$ in the level set $L_{f(\sigma)}$. Then, we let $\Sim{f}:=\bar{R}_f$, where the relation $R_f$ is defined on the critical simplices of $g$ so that $\sigma R_f \tau$ when:
		\begin{enumerate}
			\item \label{item:DefR_f1} $f_i(\sigma) = f_i(\tau)$ for some $i = 1,...,\maxdim$;
			\item \label{item:DefR_f2} there exists some $\sigma'\in C_\sigma$ and $\tau'\in C_\tau$ such that either $\sigma'\geq\tau'$ or $\sigma'\leq\tau'$.
		\end{enumerate}
	\end{defn}

	The relation $R_f$ between critical simplices follows the idea of $R_g'$. The first condition ensures that $\sigma$ and $\tau$ can enter the multifiltration induced by $f$ at a same step and the second implies that $\sigma$ and $\tau$ are connected in some way. The difference is in the idea of connectedness: the relation $R_g'$ uses the gradient field of $g$ to determine if $\sigma$ and $\tau$ are connected, while $R_f$ uses their level sets with respect to $f$. Furthermore, when $g$ is perfect relatively to $f$, we can deduce from condition \ref{item:DefR_f2} of Definition \ref{def:CritComponentsF} that, if two critical simplices $\sigma$ and $\tau$ are such that $\sigma\Sim{f}\tau$, then they are connected by a path entirely contained in $\Pareto{f}$. Hence, for a relative-perfect \mdm function $g$, each of its critical components with respect to $\Sim{f}$ is a subset of some connected component of $\Pareto{f}$.

	\begin{figure}[ht]
		\centering
		\subcaptionbox{\label{figSub:SeashellPareto}}[0.45\textwidth]{
			\centering
			\includegraphics[width=\linewidth]{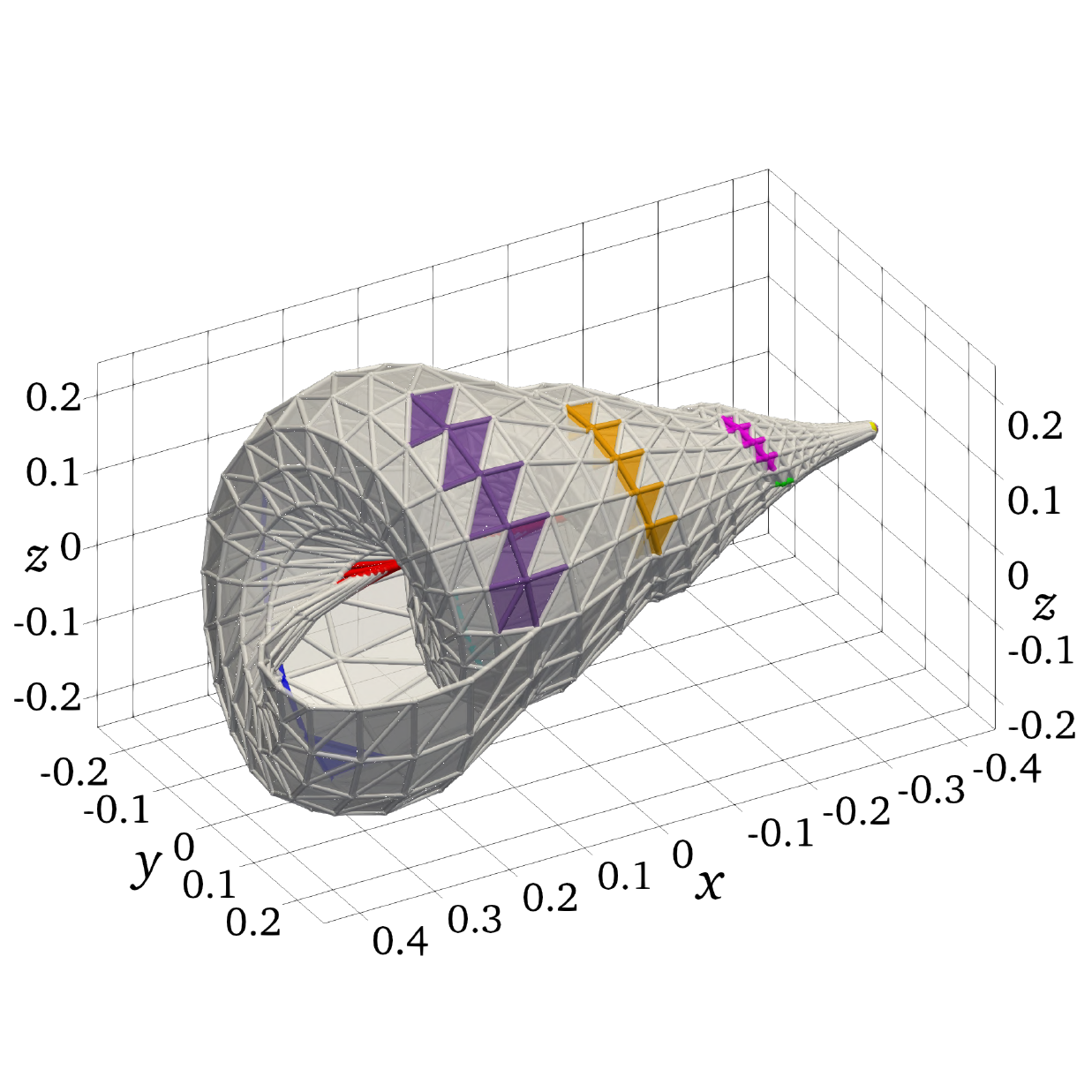}
		}
		\subcaptionbox{\label{figSub:SeashellComponents}}[0.45\textwidth]{
			\centering
			\includegraphics[width=\linewidth]{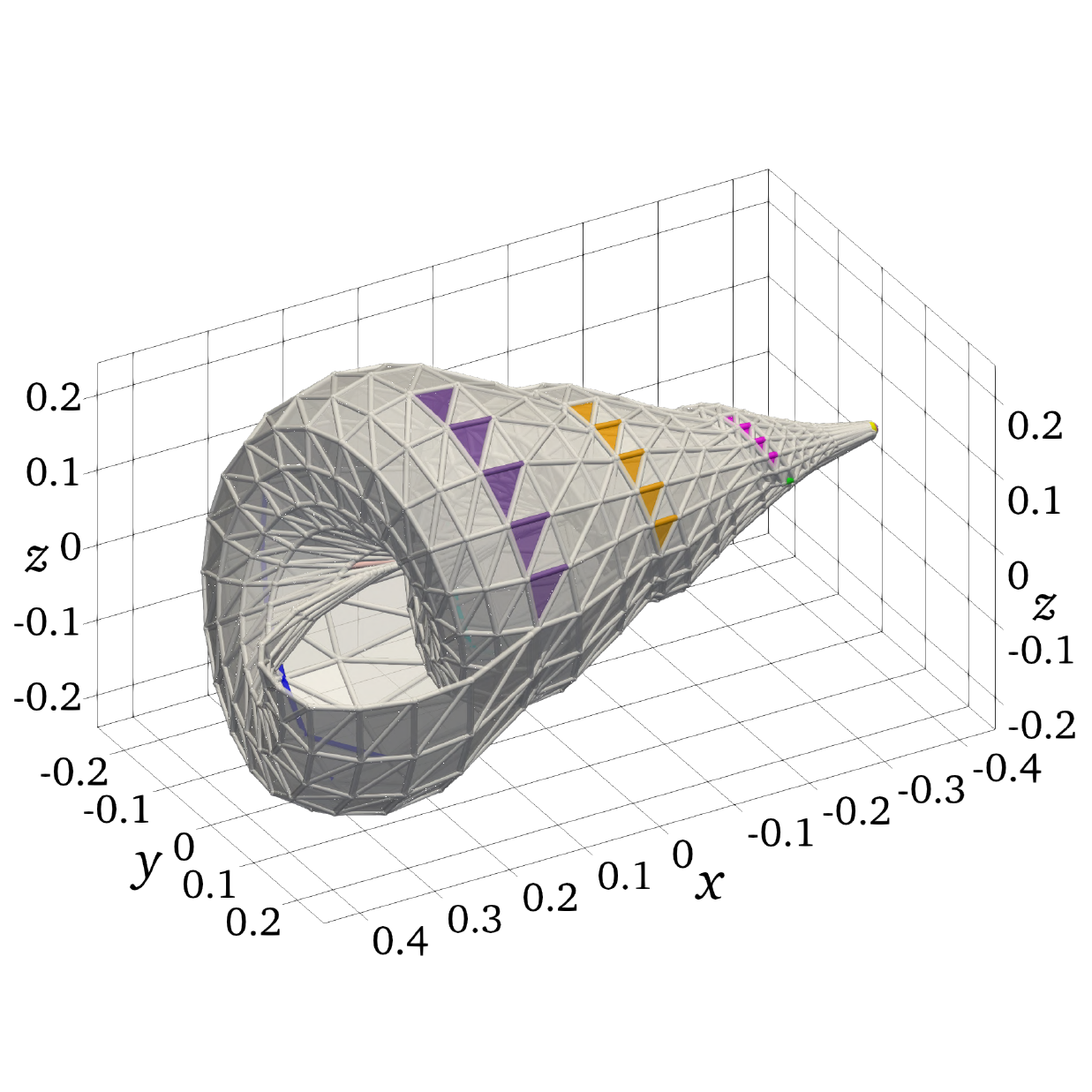}
		}
		\caption{In \subref{figSub:SeashellPareto}, the different colors show the connected components of the Pareto set $\Pareto{f}$, where $f$ is the max-extension of $v\mapsto(y,z)$ on the \texttt{seashell} dataset. In \subref{figSub:SeashellComponents}, the corresponding critical components of $g$ with respect to $\Sim{f}$, where $g$ is the \mdm function output by \GenerateMDM when given $f$ as input.}
		\label{fig:Seashell}
	\end{figure}

	In some cases, the critical components of the \mdm function $g$ with respect to $\Sim{f}$ are in correspondence with the connected components of $\Pareto{f}$. For example, we can verify that if $g$ is defined on a sphere as in Figure \ref{fig:CritRelGworksNot}, its critical components with respect to $\Sim{f}$ are the same as those with respect to $\Sim{g}'$. Also, if it is defined on a torus as in Figures \ref{fig:CritRelG'worksNotPareto} and \ref{fig:CritRelG'worksNot}, it has four components with respect to $\Sim{f}$, as desired. Moreover, if $f$ is the max-extension of $v\mapsto(y,z)$ on the \texttt{seashell} triangulation obtained from the GTS library \citep{GTS}, as in Figure \ref{fig:Seashell}, we can verify that all critical components of the generated \mdm function $g$ with respect to $\Sim{f}$ are in correspondence with the connected components of $\Pareto{f}$.

	\begin{figure}[ht]
		\centering
		\subcaptionbox{\label{figSub:HeadPareto}}[0.45\textwidth]{
			\centering
			\includegraphics[width=\linewidth]{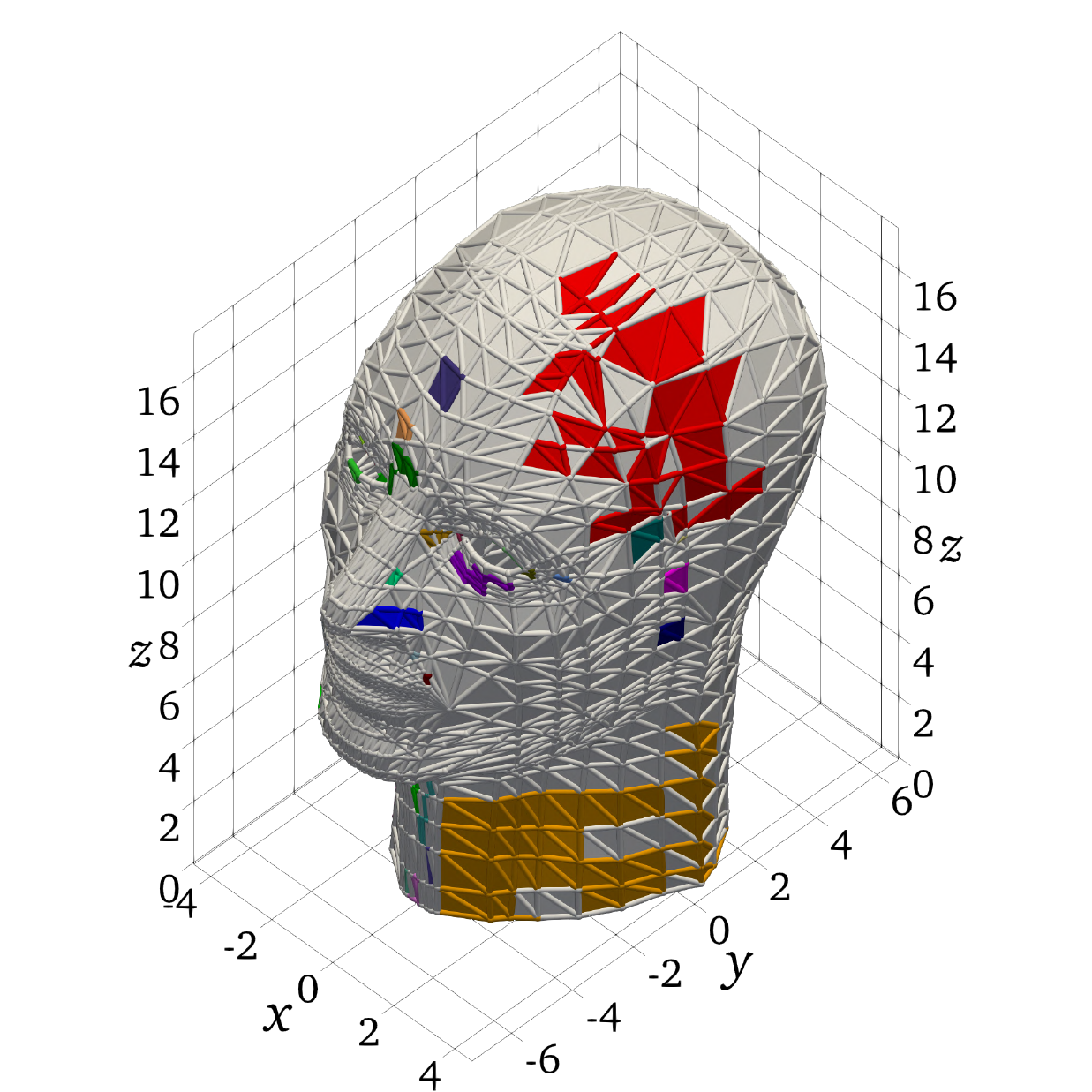}
		}
		\subcaptionbox{\label{figSub:HeadComponents}}[0.45\textwidth]{
			\centering
			\includegraphics[width=\linewidth]{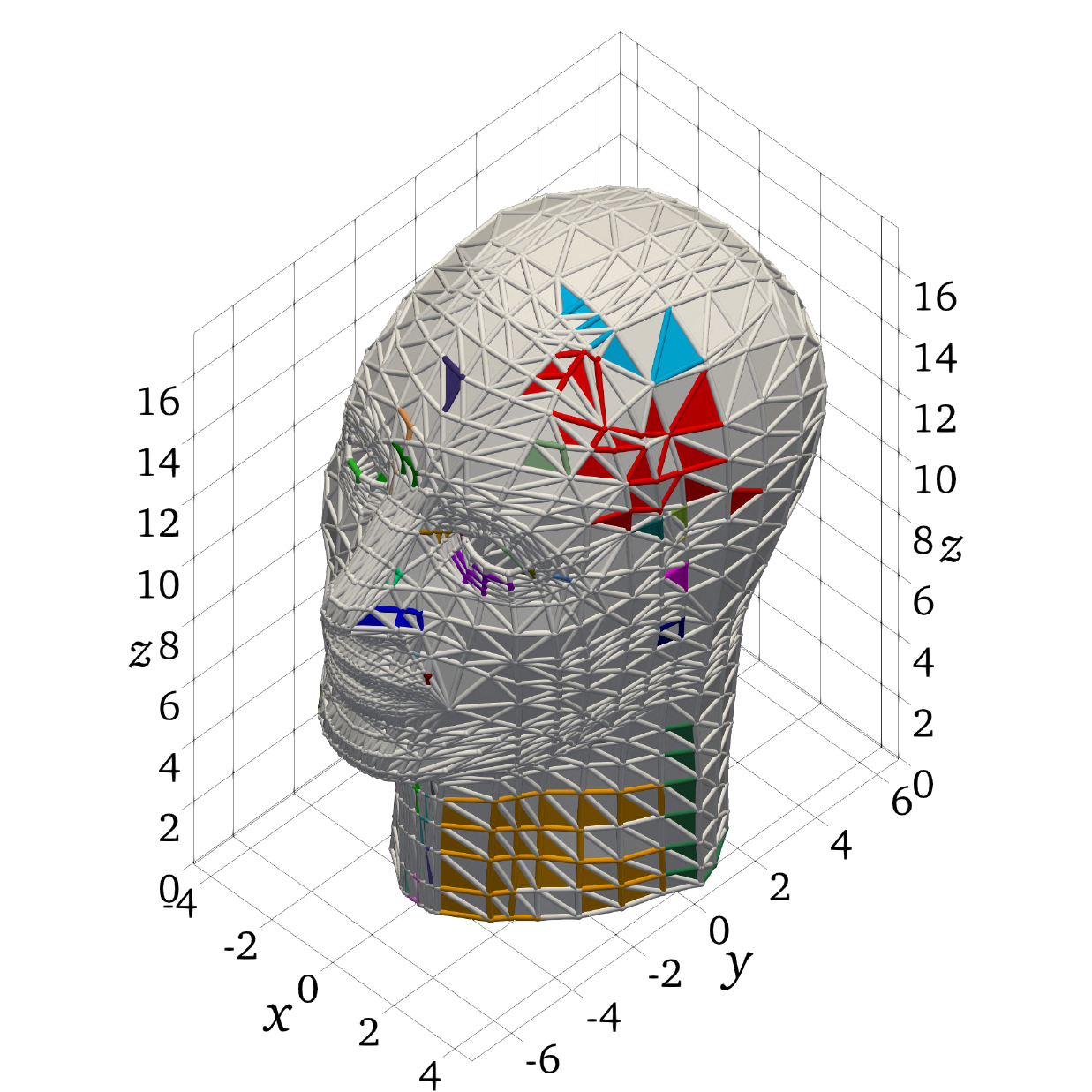}
		}
		\caption{In \subref{figSub:HeadPareto}, the different colors show the connected components of the Pareto set $\Pareto{f}$, where $f$ is the max-extension of $v\mapsto(x,z)$ on the \texttt{head} dataset from the GTS library. In \subref{figSub:HeadComponents}, the critical components of $g$ with respect to $\Sim{f}$, where $g$ is the \mdm function output by \GenerateMDM when given $f$ as input.}
		\label{fig:Head}
	\end{figure}

	Nevertheless, in many cases, some of the connected components of $\Pareto{f}$ contain multiple critical components of $g$ with respect to $\Sim{f}$. See, for instance, functions $f$ and $g$ as defined in Figure \ref{fig:Head}. When this happens, we can interpret the critical components of $g$ as follows. Consider two critical simplices $\sigma$ and $\tau$ of $g$ and let $C_\sigma$ and $C_\tau$ be their connected components in $L_{f(\sigma)}$ and $L_{f(\tau)}$ respectively. Also,we choose $\sigma$ and $\tau$ so that there exists $\sigma'\in C_\sigma$ and $\tau'\in C_\tau$ which satisfy $\sigma'\leq\tau'$, meaning that condition \ref{item:DefR_f2} of Definition \ref{def:CritComponentsF} is verified. Moreover, assuming $g$ is perfect relatively to $f$, it follows from Corollary \ref{coro:CriticalImpliesPareto} that $\sigma,\tau\in\Pareto{f}$ and, since $\sigma'\leq\tau'$,  we see that $\sigma$ and $\tau$ belong in the same connected component of $\Pareto{f}$. If $\sigma\not\Sim{f}\tau$, necessarily, condition \ref{item:DefR_f1} of Definition \ref{def:CritComponentsF} is not satisfied, so $f_i(\sigma) \neq f_i(\tau)$ for each $i=1,...,\maxdim$ and, because $f(\sigma)=f(\sigma')\preceq f(\tau')=f(\tau)$, we deduce that $f_i(\sigma) < f_i(\tau)$ for each $i$. Thus, although $\sigma$ and $\tau$ belong in the same connected component of $\Pareto{f}$, we see that $\sigma$ always appear before $\tau$ in the multifiltration induced by $f$, so they cannot possibly enter it at a same step. In other words, $\sigma$ and $\tau$ correspond to two distinct features of the multifiltration, which explains why they belong to two different critical components with respect to $\Sim{f}$.

	\section{Discussion and closing remarks}

	In this paper, we were able to extend results in \citep{Allili2019, Scaramuccia2020} and answer to some open problems therein. Notably, we presented \GenerateMDM, an algorithm which, in addition to producing a gradient field as its predecessors \Matching \citep{Allili2019} and \ComputeDiscreteGradient \citep{Scaramuccia2020}, may take as input any multifiltering function and computes a compatible \mdm function that approximates it. Furthermore, the concept of (discrete) Pareto set was defined and linked to the critical simplices produced by \GenerateMDM in order to shed light on some experimental results in \citep{Allili2019}. Moreover, we showed how the idea of critical components introduced in \citep{Brouillette2022} translates into practice and adapted their definition to the applied setting.

	Still, some questions remain open. From a theoretical perspective, as shown in smooth and continuous settings \citep{AssifPK2021, Budney2023, Cerri2019}, the Pareto set of a vector filtering function is strongly related to the homological changes in the associated multifiltration. Thus, analogous connections could seemingly be made in our discrete setting. This could be addressed algebraically, by analyzing how the simplices in the Pareto set of a discrete multifiltering function are related in the associated multiparameter persistence modules \citep{Landi2022}, or dynamically, by considering a partition into level sets as a combinatorial multivector field \citep{Lipinski2023}. Otherwise, by restraining the problem to the study of discrete vector functions induced by a component-wise injective vertex map, we could instead consider a piecewise linear approach, using the concepts defined in \citep{Edelsbrunner2008a} or \citep{Huettenberger2013} as a starting point.

	Furthermore, it would be possible to improve some computational aspects of \GenerateMDM. For instance, we could adapt the algorithm to multifiltering functions defined on a cubical complex (see \citep{Kaczynski2004}) so that it can be used to generate \mdm functions on images. Also, we may verify that almost every \mdm function $g$ obtained experimentally in Section \ref{sec:ExperimentalResults} is perfect relatively to its associated input map $f$, with only a few exceptions when $f=0$. Hence, it would be of interest to find theoretical conditions on the input of \GenerateMDM that would guarantee the optimality of the output, as it was done in \citep{Landi2022} for the algorithm \ComputeDiscreteGradient (and thus for \Matching, since it produces an equivalent output \citep{Scaramuccia2020}). Besides, for \GenerateMDM to be used to reduce the size of a simplicial complex prior to multipersistent homology computations, it should be improved in terms of speed. Indeed, using the multi-chunk algorithm as seen in \citep{Fugacci2019, Fugacci2023} or a minimal presentation of the homology of the complex \citep{Fugacci2023, Kerber2021, Lesnick2022} seems more efficient for this task.

	Nonetheless, \GenerateMDM could potentially be used in other applications related to either multipersistent homology or topological data analysis. We could, for example, adapt the concepts of Morse-Smale complex \citep{Gyulassy2014}, Morse connections graph \citep{Allili2007} or descending complex \citep{Bauer2024} to \mdm functions. Also, \mdm functions and their gradient fields could realistically be used to modelize fluid dynamics, similarly to PL vector functions \citep{Huettenberger2013}, or to extend the application of grayscale image analysis techniques \citep{DelgadoFriedrichs2015, Hu2021} to multispectral images. Finally, since multipersistent homology may be used to define topological descriptors of datasets in artificial intelligence applications \citep{Loiseaux2023}, it would be of interest to see if \mdm theory could also yield such descriptors.

	%\printbibliography
	\bibliographystyle{abbrvnat}
	\bibliography{Refs.bib}

	\appendix

	\def\thetheorem {\Alph{section}.\arabic{theorem}}

	\section{Proofs of Section \ref{sec:Correctness}}\label{sec:ProofCorrectness}

	The main arguments which explain how the algorithms from Section \ref{sec:AlgoDescription} work were given in Section \ref{sec:Correctness}. Here, we present the comprehensive proofs of the correctness of those algorithms.

	First, we show that every $\sigma\in K$ is processed exactly once. From Algorithm \ref{algo:GenerateMDM}, it is easy to see that each level set $L_u$ is used exactly once as a parameter in \ExpandMDM. Thus, it suffices to show that, for a given $L_u$, each $\sigma\in L_u$ is processed exactly once by \ExpandMDM.

	The process of a given simplex $\sigma\in L_u$ is illustrated by the flowchart in Figure \ref{fig:ExpandMDMprocess}. Before going over it in details, we make the following key observation.

	\begin{lem}\label{lem:UnprocImpliesCofaceUnproc}
		Suppose $\sigma\in L_u$ is such that $\processed(\sigma) = \True$ in Algorithm \ref{algo:ExpandMDM}. Then, $\processed(\alpha) = \True$ for all faces $\alpha<\sigma$ in $L_u$, meaning that a face $\alpha$ of $\sigma$ cannot be processed after $\sigma$. Equivalently, a coface of $\sigma$ cannot be processed before $\sigma$.
	\end{lem}

	\begin{proof}
		We prove that $\processed(\sigma) = \True$ implies $\processed(\alpha) = \True$ for a facet $\alpha^{(p-1)}<\sigma^{(p)}$. The more general result can then be deduced recursively.

		First, suppose $\sigma$ was processed as critical at line \ref{algoLine:DeclareProcessAfterCritical} of Algorithm \ref{algo:ExpandMDM}. Then, $\sigma$ was in \PQzero at some point, so all its facets have been processed and we have $\processed(\alpha)=\True$. In the case where $\sigma$ is paired with a facet $\tau<\sigma$, then $\sigma$ was popped from \PQone at line \ref{algoLine:PopFromPQone} and $\tau = \UnprocFacet(\sigma,L_u)$. For all $\alpha \neq \tau$, we have $\processed(\alpha) = \True$ because $\tau$ is the unique unprocessed facet of $\sigma$. For $\alpha = \tau$, then $\alpha$ is paired with $\sigma$ which are both declared as processed simultaneously at line \ref{algoLine:DeclareProcessAfterPairing}.

		We now show the lemma when $\sigma$ is paired with a cofacet by \ExpandMDM. To do so, we consider the set of all simplices in $L_u$ paired with a cofacet and proceed by induction on the order in which they are processed.

		Consider the first simplex $\sigma$ to be paired with a cofacet $\sigma'$, so $\sigma'$ is popped from \PQone at line \ref{algoLine:PopFromPQone} with $\NumUnprocFacets(\sigma',L_u) = 1$ and $\sigma$ is defined as $\UnprocFacet(\sigma',L_u)$ at line \ref{algoLine:UnprocessedFacet}. Consider a facet $\alpha<\sigma$ in $L_u$. We know from the structure of a simplicial complex that there exists some $\gamma\in K$ such that $\alpha<\gamma<\sigma'$ and $\gamma\neq\sigma$. Moreover, from the definition of $f$, we have that $f(\alpha) \preceq f(\gamma)\preceq f(\sigma')$, and since $f(\alpha) = f(\sigma') = u$, it follows that $f(\gamma) = u$ and $\gamma\in L_u$. Necessarily, $\processed(\gamma) = \True$ because $\sigma$ is the unique unprocessed facet of $\sigma'$ in $L_u$. Furthermore, since $\sigma$ is the first simplex in $L_u$ to be paired with a cofacet, $\gamma$ was necessarily processed as critical, meaning that $\gamma$ was in \PQzero at some point and $\gamma$ has no unprocessed facet. Thus, $\alpha<\gamma$ has to be processed.

		Now, assume $\sigma$ is the $n^\text{th}$ simplex in $L_u$ to be paired with a cofacet $\sigma'$ and suppose that the first $n-1$ such simplices had no unprocessed facet in $L_u$ when they were paired. Consider a facet $\alpha<\sigma$ in $L_u$. As above, we could show that there exists a processed facet $\gamma\in L_u$ of $\sigma'$ such that $\gamma\neq\sigma$ and $\alpha<\gamma<\sigma'$.
		\begin{itemize}
			\item If $\gamma$ was processed as critical, then it was in \PQzero at some point, meaning that all its facets, including $\alpha$, have been processed.

			\item If $\gamma$ is paired with a cofacet, then we know from the induction hypothesis that $\gamma$ has no unprocessed facet in $L_u$. Hence, $\processed(\alpha) = \True$.

			\item If $\gamma$ is paired with a facet $\tau<\gamma$, then $\gamma$ was popped from \PQone at line \ref{algoLine:PopFromPQone} and $\tau = \UnprocFacet(\gamma,L_u)$. If $\alpha \neq \tau$, then $\processed(\alpha) = \True$ since $\tau$ is the only unprocessed facet of $\gamma$. If $\alpha = \tau$, then $\alpha$ is processed along with $\gamma$ before $\sigma$.
		\end{itemize}
		In all cases, we see that $\alpha$ has to be processed before $\sigma$ is paired.
	\end{proof}

	Not only is this lemma the key to understanding why our algorithms work, it also allows us to deduce that when \ExpandMDM processes two simplices $\tau<\sigma$ as a pair, then $\tau$ is necessarily in \PQzero.

	\begin{lem}\label{lem:UnprocFacetIsInPQzero}
		Let a simplex $\sigma$ be popped from \PQone at line \ref{algoLine:PopFromPQone} and suppose it has a unique unprocessed facet $\tau\in L_u$. Necessarily, all facets of $\tau$ in $L_u$, if any, are processed and $\tau\in\PQzero$.
	\end{lem}

	\pagebreak

	\begin{proof}
		First notice that, when it is processed, $\tau$ has to be in either \PQzero or \PQone:
		\begin{itemize}
			\item If $\NumUnprocFacets(\tau,L_u) \leq 1$ when \ExpandMDM starts processing $L_u$, then $\tau$ is added to either \PQone at line \ref{algoLine:InitialAddToPQone} or \PQzero at line \ref{algoLine:InitialAddToPQzero}.
			\item Lemma \ref{lem:UnprocImpliesCofaceUnproc} tells us that all facets of $\tau$ are necessarily processed before $\tau$ is paired with $\sigma$, so even if $\NumUnprocFacets(\tau,L_u) \geq 2$ initially, we know that the algorithm eventually processes enough facets of $\tau$ so that $\NumUnprocFacets(\tau,L_u) = 2$. Then, the next facet $\alpha$ of $\tau$ to be processed by the algorithm is declared as processed at either line \ref{algoLine:DeclareProcessAfterPairing} or line \ref{algoLine:DeclareProcessAfterCritical}, and $\tau$ is added to \PQone by the \AddCofacets function at either line \ref{algoLine:AddCofacetsAfterPairing} or line \ref{algoLine:AddCofacetsAfterCritical}.
		\end{itemize}
		Also, because $\tau$ is still unprocessed when it is paired with $\sigma$, we know it is still in one of the two queues at that moment. However, it cannot be in \PQone: if it were, it would be popped from \PQone instead of $\sigma$ because $\tau<\sigma\Rightarrow I(\tau)<I(\sigma)$, meaning that $\tau$ has priority over $\sigma$ in \PQone.
	\end{proof}

	\begin{figure}[ht]
		\centering
		\begin{scaletikzpicturetowidth}{\textwidth}
			\begin{tikzpicture}[scale=\tikzscale, every node/.append style={transform shape},
				startstop/.style={rectangle, rounded corners, minimum width=3cm, minimum height=1cm,text centered, draw=black, fill=red!30},
				process/.style={rectangle, minimum width=3cm, minimum height=1cm, text centered, draw=black, fill=orange!30},
				decision/.style={diamond, minimum width=3cm, minimum height=1cm, text centered, draw=black, fill=green!30, aspect=5},
				io/.style={trapezium, trapezium left angle=70, trapezium right angle=110, minimum width=3cm, minimum height=1cm, text centered, draw=black, fill=blue!30},
				arrow/.style={thick,->,>=stealth}
				]

				\node (start) [startstop] {For $\sigma\in L_u$};

				\node (init0) [decision, below of=start, yshift=-1.5cm] {$\NumUnprocFacets(\sigma,L_u)=0$ ?};

				\node (init1) [decision, below of=init0, yshift=-1.75cm] {$\NumUnprocFacets(\sigma,L_u)=1$ ?};

				\node (add1) [process, below of=init1, xshift=-4.5cm, yshift=-1cm]{add $\sigma$ to \PQone};

				\node (add0) [process, below of=init1, xshift=4.5cm, yshift=-1cm]{add $\sigma$ to \PQzero};

				\node (wait) [process, xshift=-7cm] at ($(init0)!0.5!(init1)$) {\parbox{3cm}{\centering wait until some facet of $\sigma$ is processed}};

				\node (PQ1Priority) [decision, below of=add1, yshift=-1.5cm] {$\sigma$ has priority in \PQone ?};

				\node (processPQone) [process, xshift=-5cm] at ($(PQ1Priority)!0.5!(add1)$) {\parbox{3cm}{\centering process another simplex in \PQone}};

				\node (popPQone) [process, below of=PQ1Priority, yshift=-1.5cm] {remove $\sigma$ from \PQone};

				\node (numunproc0) [decision, below of=popPQone, yshift=-1.5cm] {$\NumUnprocFacets(\sigma,L_u) = 0$ ?};

				\node (PQoneEmpty) [decision, below of=add0, yshift=-1.5cm] {$\PQone=\emptyset$ ?};

				\node (PQ0Priority) [decision, below of=PQoneEmpty, yshift=-1.5cm] {$\sigma$ has priority in \PQzero ?};

				\node (processPQone0) [process, right of=PQoneEmpty, xshift=5cm] {\parbox{3cm}{\centering process a simplex $\sigma'$ in \PQone}};

				\node (paired) [decision, right of=add0, xshift=5cm] {$\sigma$ paired with $\sigma'$ ?};

				\node (popPQzero) [process, below of=PQ0Priority, yshift=-1.5cm] {remove $\sigma$ from \PQzero};

				\node (processPQzero) [process, right of=PQ0Priority, xshift=5cm] {\parbox{3cm}{\centering process another simplex in \PQzero}};

				\node (processed) [startstop, yshift=-4cm] at ($(popPQzero)!0.5!(popPQone)$) {$\processed(\sigma)=\True$};

				\draw [arrow] (start) -- node[anchor=west] {Line \ref{algoLine:InitialCheckZero}} (init0);

				\draw [arrow] (init0.south) node[anchor=north west] {NO} -- (init1.north) node[anchor=south west] {Line \ref{algoLine:InitialCheckOne}};

				\draw [arrow] (init0.east) node[anchor=south west] {YES} -| (add0.north) node[anchor=south west] {Line \ref{algoLine:InitialAddToPQzero}};

				\draw [arrow] (init1.north west) node[anchor=south east] {NO} |- (wait);

				\draw [arrow] (wait) |- (init1.west) node[anchor=south east] {Lines \ref{algoLine:AddCofacetsAfterPairing} or \ref{algoLine:AddCofacetsAfterCritical}};

				\draw [arrow] (init1.south) node[anchor=north west] {YES} |- (add1.east) node[anchor=south west] {Lines \ref{algoLine:InitialAddToPQone}, \ref{algoLine:AddCofacetsAfterPairing} or \ref{algoLine:AddCofacetsAfterCritical}};

				\draw [arrow] (add1) -- node[anchor=west] {Line \ref{algoLine:LoopPQoneBegins}} (PQ1Priority);

				\draw[arrow] (PQ1Priority.north west) node[anchor=south east] {NO} |- (processPQone);

				\draw[arrow] (processPQone) |- (PQ1Priority.west) node[anchor=south east] {Line \ref{algoLine:LoopPQoneBegins}};

				\draw [arrow] (PQ1Priority.south) node[anchor= north west] {YES} -- (popPQone.north) node[anchor=south west] {Line \ref{algoLine:PopFromPQone}};

				\draw [arrow] (popPQone) -- node[anchor=west] {Line \ref{algoLine:CheckZero}} (numunproc0);

				\draw [arrow] (numunproc0.south) node [anchor=north west] {NO} |- (processed.west) node[anchor=south east] {Line \ref{algoLine:DeclareProcessAfterPairing}};

				\draw [arrow] (numunproc0.east) node [anchor=south west] {YES} -- +(1,0) |- (add0.west) node[anchor=south east] {Line \ref{algoLine:MoveFromPQoneToPQzero}};

				\draw [arrow] (add0) -- node [anchor=east] {Line \ref{algoLine:LoopPQoneBegins}} (PQoneEmpty);

				\draw [arrow] (PQoneEmpty.east) node [anchor=south west] {NO} -- (processPQone0);

				\draw [arrow] (processPQone0) -- node[anchor=west] {Line \ref{algoLine:UnprocessedFacet}} (paired);

				\draw [arrow] (paired.south west) node [anchor=north east] {NO} -- +(0,-1) -| node[anchor=south west] {Line \ref{algoLine:LoopPQoneBegins}} (PQoneEmpty.north east);

				\draw [arrow] (paired.east) node [anchor = south west] {YES} -- +(0.5,0) |- (processed.east) node [anchor=south west] {Line \ref{algoLine:DeclareProcessAfterPairing}};

				\draw [arrow] (PQoneEmpty.south) node[anchor = north east] {YES} -- (PQ0Priority.north) node[anchor=south east] {Line \ref{algoLine:PQzeroNeqEmpty}};

				\draw [arrow] (PQ0Priority.south) node[anchor = north east] {YES} -- (popPQzero.north) node[anchor = south east] {Line \ref{algoLine:PopFromPQzero}};

				\draw [arrow] (PQ0Priority.east) node[anchor = south west] {NO} -- (processPQzero);

				\draw [arrow] (processPQzero) -- +(0,1) -| node[anchor=south west] {Line \ref{algoLine:LoopPQoneBegins}} (PQoneEmpty.south east);

				\draw [arrow] (popPQzero) -- +(0,-1) -| (processed.north) node[anchor=south west] {Line \ref{algoLine:DeclareProcessAfterCritical}};
			\end{tikzpicture}
		\end{scaletikzpicturetowidth}
		\caption{Flowchart of \ExpandMDM for a single simplex $\sigma$ in a given level set $L_u$.}\label{fig:ExpandMDMprocess}
	\end{figure}

	From these observations, we can summarize the process of a simplex $\sigma\in L_u$ with the flowchart in Figure \ref{fig:ExpandMDMprocess}. Although it contains a few loops, we can show $\sigma$ always ends up being processed.

	\propProcessedOnce*

	\pagebreak

	\begin{proof}
		The result follows from the three following statements:
		\begin{enumerate}
			\item\label{enum:PropProcessOnce1} Each simplex in $L_u$ eventually enters \PQone or \PQzero.
			\item\label{enum:PropProcessOnce2} Every simplex in \PQone and \PQzero is eventually processed.
			\item\label{enum:PropProcessOnce3} A processed simplex cannot enter \PQone or \PQzero again.
		\end{enumerate}

		We first prove Statement \ref{enum:PropProcessOnce3}. We see that a simplex $\sigma$ can only be added to \PQone at lines \ref{algoLine:InitialAddToPQone}, \ref{algoLine:AddCofacetsAfterPairing} and \ref{algoLine:AddCofacetsAfterCritical}, while it could be added to \PQzero at lines \ref{algoLine:InitialAddToPQzero} or \ref{algoLine:MoveFromPQoneToPQzero}.
		\begin{itemize}
			\item At lines \ref{algoLine:InitialAddToPQzero} and \ref{algoLine:InitialAddToPQone}, all simplices of $L_u$ are still unprocessed.

			\item At lines \ref{algoLine:AddCofacetsAfterPairing} and \ref{algoLine:AddCofacetsAfterCritical}, a simplex $\sigma$ can only be added to \PQone by function \AddCofacets after one of its facet is processed. However, when $\processed(\sigma)=\True$, we know from Lemma \ref{lem:UnprocImpliesCofaceUnproc} that all facets of $\sigma$ are processed, so \AddCofacets cannot add $\sigma$ to \PQone.

			\item Since a processed simplex cannot be added to \PQone, for a processed simplex $\sigma$ to be moved from \PQone to \PQzero at line \ref{algoLine:MoveFromPQoneToPQzero}, it would need to be declared as processed while inside \PQone. Yet, this cannot happen. Indeed, consider $\sigma\in\PQone$:
			\begin{itemize}
				\item if $\sigma$ has priority in \PQone, it is removed from the queue at line \ref{algoLine:PopFromPQone} before possibly being processed;

				\item if $\sigma$ does not have priority in \PQone, another simplex is popped from \PQone is either moved to \PQzero or paired with a facet $\tau\in\PQzero$ by Lemma \ref{lem:UnprocFacetIsInPQzero}, so $\sigma$ remains unprocessed.
			\end{itemize}
		\end{itemize}

		We now show Statement \ref{enum:PropProcessOnce2}, namely \ExpandMDM eventually processes all simplices in \PQone and \PQzero. We see from Algorithm \ref{algo:ExpandMDM} that every simplex in \PQone is eventually either moved to \PQzero at line \ref{algoLine:MoveFromPQoneToPQzero} or processed by being paired with a facet. Also, we know that \PQone has to become empty at some point because $L_u\subseteq K$ is finite and the simplices which are either moved to \PQzero or processed cannot enter \PQone again. Hence, Algorithm \ref{algo:ExpandMDM} eventually reaches line \ref{algoLine:PQzeroNeqEmpty}, and a simplex in \PQzero gets processed. The algorithm continues in this fashion for as long as $\PQone\neq\emptyset$ or $\PQzero\neq\emptyset$ and since
		\begin{itemize}
			\item $L_u\subset K$ is finite;
			\item simplices in \PQzero cannot enter \PQone;
			\item processed simplices cannot return to either \PQone or \PQzero;
		\end{itemize}
		it follows that Algorithm \ref{algo:ExpandMDM} eventually processes all simplices in both \PQone and \PQzero and terminates.

		Finally, we prove \ref{enum:PropProcessOnce1}, that says all simplices in $L_u$ eventually enter either queue \PQone or \PQzero. Let $p=\min_{\tau\in L_u}\dim\tau$ and consider $\sigma^{(p+r)}\in L_u$. We proceed by induction on $r$. For $r=0$, it is easy to see that $\sigma$ has no facet in $L_u$, so it enters \PQzero at line \ref{algoLine:InitialAddToPQzero}. Now, consider $r\geq 1$ and suppose \ref{enum:PropProcessOnce1} is true for $r-1$.
		\begin{itemize}
			\item Assume that, initially, $\NumUnprocFacets(\sigma,L_u)\leq 1$. Then, $\sigma$ is added to either \PQzero or \PQone at line \ref{algoLine:InitialAddToPQzero} or \ref{algoLine:InitialAddToPQone}.

			\item Assume $\NumUnprocFacets(\sigma,L_u) > 1$ at the beginning of Algorithm \ref{algo:ExpandMDM}. By the induction hypothesis, we know that all facets of $\sigma$ eventually enter either \PQone or \PQzero and, from \ref{enum:PropProcessOnce2}, are processed at some point. Hence, the algorithm eventually processes enough facets of $\sigma$ so that $\NumUnprocFacets(\sigma,L_u)$ decreases to 1, at which point $\sigma$ is added to \PQone by function \AddCofacets at either lines \ref{algoLine:AddCofacetsAfterPairing} or \ref{algoLine:AddCofacetsAfterCritical}.\qedhere
		\end{itemize}
	\end{proof}

	We now show that the dictionaries $g$ and $\cV$ produced by \GenerateMDM indeed represent a \mdm function $g:K\rightarrow\R^\maxdim$ and its gradient field $\cV:K\nrightarrow K$. To do so, a few preliminary observations are needed.

	\begin{lem}\label{lem:AllFacetsProcessedBefore}
		Consider a facet $\tau$ of a simplex $\sigma\in K$. In Algorithm \ref{algo:GenerateMDM}, $\tau$ is processed either before $\sigma$ or at the same time by being paired by \ExpandMDM.
	\end{lem}

	\begin{proof}
		Consider $u\in f(K)$ such that $\sigma\in L_u$. For a facet $\tau<\sigma$, we know that $f(\tau)\preceq f(\sigma)=u$, so either $f(\tau) = u'\precneqq u$ or $f(\tau) = u$.
		\begin{itemize}
			\item If $f(\tau) = u'\precneqq u$, $\tau$ is processed as part of $L_{u'}$ which is necessarily processed before $L_u$ by \GenerateMDM because the routine \LevelSets orders the level sets so that $L_{u'}$ comes before $L_u$ when $u'\precneqq u$.

			\item If $f(\tau) = u$, $\tau$ was processed within the same level set $L_u$ as $\sigma$, and we know from Lemma \ref{lem:UnprocImpliesCofaceUnproc} that $\tau$ cannot be processed before $\sigma$. \qedhere
		\end{itemize}
	\end{proof}

	\begin{lem}\label{lem:FirstProcessCriticalVertex}
		Let $\sigma_0$ be the first simplex to be processed by \GenerateMDM. Necessarily, $\sigma_0$ was processed as critical and is a vertex of $K$.
	\end{lem}

	\begin{proof}
		Consider $u\in f(K)$ such that $\sigma_0\in L_u$. Since $\sigma_0$ is the first simplex to be processed, it belongs to the first level set $L_u$ given by the \LevelSets routine, which implies that there exists no level set $L_{u'}$ with $u'\precneqq u$. Now, suppose $\sigma_0$ is processed as either one of the simplices $\sigma$ or $\tau$ from lines \ref{algoLine:UnprocessedFacet}--\ref{algoLine:AddCofacetsAfterPairing}. From the properties of a simplicial complex, we know $\sigma$ has at least another facet $\alpha\neq\tau$ and, since $f(\alpha)\preceq f(\sigma) = u$ and there exists no $u'\precneqq u$, it follows that $\alpha\in L_u$. However, $\tau$ is the unique unprocessed facet of $\sigma$ in $L_u$, so $\processed(\alpha)=\True$, which is a contradiction since $\sigma_0\neq\alpha$ and $\sigma_0$ is the first simplex to be processed.

		Also, notice that $\sigma_0\in\PQzero$ before it gets processed, meaning that all its facets have to be processed. Since $\sigma_0$ is the first simplex in $K$ to be processed, we conclude it cannot have any facet, namely $\sigma_0$ is a vertex.
	\end{proof}

	\propComputeGworks*

	\begin{proof}
		Let $K=\{\sigma_0,\sigma_1,...,\sigma_n,...,\sigma_N\}$ where each simplex is labelled so that if $\sigma_n$ was processed before $\sigma_{n'}$ by \GenerateMDM, then $n<n'$. We prove all three statements by induction on $n$.

		For $n=0$, consider $u\in f(K)$ such that $\sigma_0\in L_u$. We know from Lemma \ref{lem:FirstProcessCriticalVertex} that $\sigma_0$ is a critical vertex, so $g(\sigma_0)$ is defined by calling $\ComputeG(f,g,\delta,\sigma_0)$ at line \ref{algoLine:DefineGcrit} of Algorithm \ref{algo:ExpandMDM}. From Algorithm \ref{algo:ComputeG}, it is obvious that $\ComputeG(f,g,\delta,\sigma_0)$ returns $f(\sigma_0)\in\R^\maxdim$, so $g(\sigma_0):= f(\sigma_0)$, which proves both \ref{enum:propComputeG1} and \ref{enum:propComputeG3}. Also, because $\sigma_0$ is a vertex, it has no facet, which makes \ref{enum:propComputeG2} trivial.

		Now, consider $n\geq 1$ and suppose that the proposition is verified for simplices $\sigma_0,\sigma_1,...,\sigma_{n-1}$. If $\sigma_n$ is a critical vertex, we prove the proposition as above in the case $n=0$. Also, when $\sigma_n$ is processed as critical but not a vertex, the proof is similar to the case when $\sigma_n$ is paired with a facet by \ExpandMDM. Thus, for the remainder of the proof, we assume $\sigma_n$ is processed as part of a pair $\tau<\sigma$. In this case, note that $g(\sigma_n)$ is defined by $\ComputeG(f,g,\delta,\sigma,\tau)$ at line \ref{algoLine:DefineGpair} of Algorithm \ref{algo:ExpandMDM} and $\sigma$ is not a vertex, so \ComputeG necessarily enters the \textbf{else} statement at line \ref{algoLine:ElseNotVertex}.
		\begin{itemize}
			\item[\ref{enum:propComputeG1}] Consider the set $A$ of facets $\alpha<\sigma$ such that $\alpha\neq\tau$. Every $\alpha\in A$ was processed before $\sigma$ by Lemma \ref{lem:AllFacetsProcessedBefore}, so $g(\alpha)$ is well defined for each $\alpha\in A$ by the induction hypothesis. Hence, the instruction at line \ref{algoLine:DefineU1} and the \textbf{if} statement at line \ref{algoLine:IfU=GAlpha} of Algorithm \ref{algo:ComputeG} are executed without error, which proves \ref{enum:propComputeG1}.

			\item[\ref{enum:propComputeG2}] To show $g(\alpha)\preceq g(\sigma_n)$ for all facets $\alpha$ of $\sigma_n$, we consider separately the cases $\sigma_n = \sigma$ and $\sigma_n=\tau$.
			\begin{itemize}
				\item We first assume $\sigma_n=\sigma$. Consider the set $A$ of facets $\alpha<\sigma$ such that $\alpha\neq\tau$ and the vector value $w=(w_1,...,w_\maxdim)\in\R^\maxdim$ as defined at line \ref{algoLine:DefineU} of Algorithm \ref{algo:ComputeG}, so $w_1 = \max\left(\{f_1(\sigma)\}\cup\{g_1(\alpha)\ |\ \alpha\in A\}\right)$ and $w_i = f_i(\sigma)$ for $i>1$. For $i=1$, it is clear that $w_1\geq g_1(\alpha)$ for all $\alpha\in A$, while for $i>1$,
				\begin{gather*}
					w_i = f_i(\sigma) \geq f_i(\alpha) = g_i(\alpha) \ \forall\alpha\in A
				\end{gather*}
				because $f(\sigma) \succeq f(\alpha)$ by definition of an admissible map and we know that $f_i(\alpha) = g_i(\alpha)$ for all $i>1$ and all $\alpha\in A$ by the induction hypothesis. Thus, $w\succeq g(\alpha)$ for all $\alpha\in A$. Since $g_1(\sigma)$ is defined either as $w_1$ or $w_1+\delta$ and $g_i(\sigma) := w_i$ for $i>1$, it follows that $g(\sigma)\succeq w \succeq g(\alpha)$ for all $\alpha\in A$. Furthermore, if $w\succneqq g(\alpha)$ for all $\alpha\in A$, then $g(\sigma):= w \succneqq g(\alpha)$ for all $\alpha\in A$, and if $w= g(\alpha)$ for some $\alpha\in A$, then $g_1(\sigma):= w_1+\delta > w_1\geq g_1(\alpha)$. In both cases, we see that $g(\sigma)\succneqq g(\alpha)$ for all $\alpha\in A$, so if $g(\sigma) = g(\alpha)$ for some facet $\alpha<\sigma$, we necessarily have $\alpha\notin A$, and thus $\alpha=\tau$.

				\item Now, assume $\sigma_n=\tau$. Consider a facet $\gamma<\tau$. Since $\tau$ is paired with $\sigma$, we know $\tau$ is not paired with $\gamma$, so we have to show that $g(\gamma)\precneqq g(\tau)$.

				Recall from line \ref{algoLine:DefineGpair} of Algorithm \ref{algo:ExpandMDM} that $g(\tau)=g(\sigma)$. Moreover, we know from the properties of a simplicial complex that there exists a simplex $\alpha\neq\tau$ such that $\gamma < \alpha< \sigma$. Using the arguments from the previous case, we can see that $g(\sigma)\succneqq g(\alpha)$ since $\alpha$ is a facet of $\sigma$ such that $\alpha\neq\tau$. Also, $\alpha$ was processed before $\sigma$ and $\tau$ from Lemma \ref{lem:AllFacetsProcessedBefore}, so we can deduce from the induction hypothesis that the facet $\gamma$ of $\alpha$ is such that $g(\alpha)\succeq g(\gamma)$. We conclude that
				\begin{gather*}
					g(\tau) = g(\sigma)\succneqq g(\alpha)\succeq g(\gamma).
				\end{gather*}
			\end{itemize}

			\item[\ref{enum:propComputeG3}] It is easy to see that $g_i(\sigma_n) = f_i(\sigma_n)$ for $i>1$. Thus, we only prove that $g_1(\sigma_n) = f_1(\sigma_n) + m\delta$ where $m$ is bounded by the number of simplices processed before $\sigma_n$. Notice that if $\sigma_n$ is paired with $\sigma_{n-1}$, then \ref{enum:propComputeG3} becomes obvious from the induction hypothesis since $g(\sigma_n)=g(\sigma_{n-1})$. Hence, assume $\sigma_n$ is not paired with $\sigma_{n-1}$: then, the number of simplices processed before $\sigma_n$ is exactly $n$.

			Recall that $g_1(\sigma_n)$ is either $w_1$ or $w_1+\delta$ where $w_1$ is defined at line \ref{algoLine:DefineU1} of Algorithm \ref{algo:ComputeG} as $w_1 = \max\left(\{f_1(\sigma)\}\cup\{g_1(\alpha)\ |\ \alpha\in A\}\right)$ and $A$ is the set of facets $\alpha<\sigma$ such that $\alpha\neq\tau$. Also, we see that $f_1(\sigma) = f_1(\tau) = f_1(\sigma_n)$ since $\sigma_n$ is either $\sigma$ or $\tau$ and they both belong in a same level set.
			\begin{itemize}
				\item Assume $w_1 = f_1(\sigma)$. Then, we have either $g_1(\sigma_n) := f_1(\sigma_n)$ or $g_1(\sigma_n) := f_1(\sigma_n) + \delta$, so $g_1(\sigma_n) = f_1(\sigma_n) + m\delta$ where $m\in\{0,1\}$. Since $n\geq 1$ by hypothesis, we have $m\leq n$.

				\item Assume $w_1 = g_1(\alpha) > f_1(\sigma)$ for some $\alpha\in A$. By the induction hypothesis, $g_1(\alpha) = f_1(\alpha) + m\delta$ where $m$ is bounded by the number of simplices processed before $\alpha$. Thus, $w_1 = f_1(\alpha) + m\delta$ and $m\leq n-1$. Also, we know that $f_1(\alpha)\leq f_1(\sigma)$ because $f$ is an admissible map.
				\begin{itemize}[label=$\bullet$]
					\item If $f_1(\alpha) = f_1(\sigma)$, since $f_1(\sigma) = f_1(\sigma_n)$ and $g_1(\sigma_n)$ equals either $w_1$ or $w_1+\delta$, we necessarily have $g_1(\sigma_n) = f_1(\sigma_n) + m\delta$ or $g_1(\sigma_n) = f_1(\sigma_n) + (m+1)\delta$ where both $m, m+1\leq n$, hence the result.

					\item We show the inequality $f_1(\alpha) < f_1(\sigma)$ leads to a contradiction. Recall from line \ref{algoLine:DefineDeltaF1NonConstant} of Algorithm \ref{algo:GenerateMDM} that $0<\delta\leq \frac{f_1(\sigma) - f_1(\alpha)}{|K|}$. Also, $m<n<|K|$, so $m\delta < |K|\delta\leq f_1(\sigma) - f_1(\alpha)$. Therefore, $g_1(\alpha) = f_1(\alpha) + m\delta < f_1(\sigma)$, which contradicts the assumption that $w_1 = g_1(\alpha) > f_1(\sigma)$.
					\qedhere
				\end{itemize}
			\end{itemize}
		\end{itemize}
	\end{proof}

	We are now ready to prove the main theorem.

	\theogMDM*

	\begin{proof}
		It is easy to see that $\cV$ is $f$-compatible since it only pairs simplices that belong in a same level set of $f$. The remainder of the proof follows from Proposition \ref{prop:ComputeGworks}. From part \ref{enum:propComputeG3} of the proposition, it is clear that
		\begin{gather*}
			\lVert g(\sigma) - f(\sigma)\rVert = |g_1(\sigma) - f_1(\sigma)| = m\delta < |K|\delta \leq \epsilon
		\end{gather*}
		since, as defined in Algorithm \ref{algo:GenerateMDM}, the parameter $\delta$ is such that $\delta\leq\frac{\epsilon}{|K|}$. Next, we show $g$ is a \mdm function and $\cV$ is its gradient field. Let $\sigma\in K$.
		\begin{itemize}
			\item Assume $\sigma$ is paired by \ExpandMDM with some facet $\tau<\sigma$. From Proposition \ref{prop:ComputeGworks}\ref{enum:propComputeG2}, we see that for all cofacets $\beta>\sigma$, since $\beta$ is not paired with $\sigma$, we have that $g(\sigma)\precneqq g(\beta)$, so $\Head{g}{\sigma} = \emptyset$. This proves conditions \ref{enum:defMDMenum1} and \ref{enum:defMDMenum3} of Definition \ref{def:MDM} of a \mdm function. Then, for all facets $\alpha\neq\tau$ of $\sigma$, we also have from Proposition \ref{prop:ComputeGworks}\ref{enum:propComputeG2} that $g(\alpha)\precneqq g(\sigma)$ while $g(\tau) = g(\sigma)$, so $\Tail{g}{\sigma} = \{\tau\}$. This proves conditions \ref{enum:defMDMenum2} and \ref{enum:defMDMenum4} of Definition \ref{def:MDM}.

			Moreover, we see at line \ref{algoLine:DefineVpair} of Algorithm \ref{algo:ExpandMDM} that $\cV(\sigma)$ is undefined. This agrees with the definition of the gradient field of $g$ which is only defined for simplices $\sigma$ with $\Tail{g}{\sigma} = \emptyset$.

			\item Now, assume $\sigma$ is paired by \ExpandMDM with some cofacet $\gamma>\sigma$. From Proposition \ref{prop:ComputeGworks}\ref{enum:propComputeG2}, we have that $g(\sigma)\precneqq g(\beta)$ for all cofacets $\beta\neq\gamma$ of $\sigma$ while $g(\gamma) = g(\sigma)$, so $\Head{g}{\sigma} = \{\gamma\}$. Thus, conditions \ref{enum:defMDMenum1} and \ref{enum:defMDMenum3} of Definition \ref{def:MDM} are satisfied. Moreover, every facet $\alpha<\sigma$ is not paired with $\sigma$ so, from Proposition \ref{prop:ComputeGworks}\ref{enum:propComputeG2}, $g(\alpha)\precneqq g(\sigma)$ and $\Tail{g}{\sigma} = \emptyset$, which verifies conditions \ref{enum:defMDMenum2} and \ref{enum:defMDMenum4} of Definition \ref{def:MDM}.

			Also, we see from line \ref{algoLine:DefineVpair} of Algorithm \ref{algo:ExpandMDM} that $\cV(\sigma):=\gamma\in \Head{g}{\sigma}$, which again agrees with the definition of the gradient field of $g$.

			\item Finally, when $\sigma$ is processed as critical, meaning that it is not paired with any facet or cofacet, we see from Proposition \ref{prop:ComputeGworks}\ref{enum:propComputeG2} that $g(\alpha)\precneqq g(\sigma)\precneqq g(\beta)$ for all facets $\alpha<\sigma$ and cofacets $\beta>\sigma$, which suffices to verify all four conditions of Definition \ref{def:MDM} of a \mdm function. Also, this means that $\Head{g}{\sigma} = \Tail{g}{\sigma} = \emptyset$, so $\sigma$ is a critical point of the \mdm function $g$. In addition, we see from line \ref{algoLine:DefineVcrit} of Algorithm \ref{algo:ExpandMDM} that $\cV(\sigma):=\sigma$, which is consistent with the definition of the gradient field of $g$. \qedhere
		\end{itemize}
	\end{proof}

	\propfMDMiffFisG*

	\begin{proof}
		If $f=g$, since $g$ is \mdm, we obviously have $f$ \mdm as well. Thus, we suppose $f$ is \mdm and show that $f(\sigma)=g(\sigma)$ for all $\sigma\in K$.

		As in the proof of Proposition \ref{prop:ComputeGworks}, we label as $\{\sigma_0,\sigma_1,...,\sigma_n,...,\sigma_N\}$ the simplices of $K$ so that $n<n'$ when $\sigma_n$ is processed before $\sigma_{n'}$ by \GenerateMDM and we prove $f(\sigma_n)=g(\sigma_n)$ by induction on $n$.

		Since $\sigma_0$ is a vertex processed as critical by Lemma \ref{lem:FirstProcessCriticalVertex}, we have that $g(\sigma_0)$ is defined as $f(\sigma_0)$, as seen at line \ref{algoLine:DefU=fSigma} of Algorithm \ref{algo:ComputeG}. For $n\geq 1$, if $\sigma_n$ is a vertex and processed as critical, we use the same argument to prove $g(\sigma_n)=f(\sigma_n)$. Otherwise, $g(\sigma_n)$ is defined by going through lines \ref{algoLine:ElseNotVertex}--\ref{algoLine:U1+=Delta} of Algorithm \ref{algo:ComputeG} either by processing $\sigma_n$ as a critical simplex or as part of a pair $\tau<\sigma$.% In both cases, for all $\alpha\in A = \{\text{facets }\alpha<\sigma\ |\ \alpha\neq\tau\}$, we know that $\alpha$ is processed before $\sigma_n$, so $g_1(\alpha) = f_1(\alpha)$ by induction hypothesis. Hence, $w_1$ is defined at line \ref{algoLine:DefineU1} as
		%\begin{gather*}
		%	w_1 := \max\left(\{f_1(\sigma)\}\cup\{g_1(\alpha)\ |\ \alpha\in A\}\right) = \max\left(\{f_1(\sigma)\}\cup\{f_1(\alpha)\ |\ \alpha\in A\}\right) = f_1(\sigma)
		%\end{gather*}
		%because $f(\sigma)\succeq f(\alpha)$ by definition of an admissible map $f$. Therefore, at line \ref{algoLine:DefineU}, $w$ is defined as $f(\sigma)$.

		%Moreover, since $f$ is \mdm, we know we have either $\card\Tail{f}{\sigma} = 0$ or $\card\Tail{f}{\sigma} = 1$, in which case we could show that $\sigma$ is paired by \ExpandMDM with its unique facet $\tau\in\Tail{f}{\sigma}$. In both cases, by definition of a \mdm function, we have that $f(\sigma)\succneqq f(\alpha)$ for all $\alpha\in A = \{\text{facets }\alpha<\sigma\ |\ \alpha\neq\tau\}$. Hence, \ComputeG does not enter the \textbf{if} statement at line \ref{algoLine:IfU=GAlpha}, so $w=f(\sigma)$ is returned to define $g(\sigma_n)$. Because $\sigma_n$ is either $\sigma$ or $\tau$ (if processed within a pair) and $f(\sigma)=f(\tau)$, we have the result.

		\begin{itemize}
			\item If $\sigma_n$ is part of a pair $\tau<\sigma$, then for all $\alpha\in A = \{\text{facets }\alpha<\sigma\ |\ \alpha\neq\tau\}$, we know that $\alpha$ is processed before $\sigma_n$, so $g_1(\alpha) = f_1(\alpha)$ by induction hypothesis. Hence, $w_1$ is defined at line \ref{algoLine:DefineU1} as
			\begin{align*}
				w_1 := &\max\left(\{f_1(\sigma)\}\cup\{g_1(\alpha)\ |\ \alpha\in A\}\right),\\
				= &\max\left(\{f_1(\sigma)\}\cup\{f_1(\alpha)\ |\ \alpha\in A\}\right) = f_1(\sigma)
			\end{align*}
			because $f(\sigma)\succeq f(\alpha)$ by definition of an admissible map $f$. Therefore, at line \ref{algoLine:DefineU}, $w$ is defined as $f(\sigma)$.

			Moreover, since $f$ is \mdm, we know that $\card\Tail{f}{\sigma} \leq 1$ and, considering that $f(\tau)=f(\sigma)$, we deduce $\Tail{f}{\sigma}=\{\tau\}$. By definition of a \mdm function, we have that $f(\sigma)\succneqq f(\alpha)$ for all $\alpha\in A = \{\text{facets }\alpha<\sigma\ |\ \alpha\neq\tau\}$. Hence, \ComputeG does not enter the \textbf{if} statement at line \ref{algoLine:IfU=GAlpha}, so $w=f(\sigma)$ is returned to define $g(\sigma_n)$. Because $\sigma_n$ is either $\sigma$ or $\tau$ and $g(\sigma_n)=f(\sigma)=f(\tau)$, we have the result.

			\item If $\sigma_n$ is processed as a critical simplex, most of the previous reasoning still holds. More precisely, we see that for all $\alpha\in A = \{\text{facets }\alpha<\sigma_n\}$, we have $g_1(\alpha) = f_1(\alpha)$ by induction hypothesis and $w_1$ is defined as $w_1 := f_1(\sigma_n)$ meaning that, at line \ref{algoLine:DefineU}, $w:=f(\sigma_n)$.

			Furthermore, we could show that $\sigma_n$ being processed as critical means that $\sigma_n$ is a critical point of the \mdm function $f$, thus $\card\Tail{f}{\sigma_n} = 0$. Consequently, by definition of a \mdm function, we have $f(\sigma_n)\succneqq f(\alpha)$ for all facets $\alpha\in A$ and, again, \ComputeG does not enter the \textbf{if} statement at line \ref{algoLine:IfU=GAlpha}. We conclude that $w=f(\sigma_n)$ is returned to define $g(\sigma_n)$, hence the result.\qedhere
		\end{itemize}
	\end{proof}

\end{document}